\documentclass[11pt]{amsart}

\usepackage{fullpage}
\author{Yakov Berchenko-Kogan}
\title{Yang--Mills Replacement}

\usepackage{amsmath,amssymb,amsthm,mathtools}
\usepackage{hyperref}

\newtheorem{theorem}{Theorem}[section]
\newtheorem{proposition}[theorem]{Proposition}
\newtheorem{lemma}[theorem]{Lemma}
\newtheorem{corollary}[theorem]{Corollary}
\theoremstyle{definition}
\newtheorem{definition}[theorem]{Definition}

\renewcommand{\epsilon}{\varepsilon}

\newcommand{\forms}[1]{{\textstyle\bigwedge^{#1}}}
\newcommand{\ext}{\forms{*}}

\DeclareMathOperator{\range}{range}

\DeclareMathOperator{\Ad}{Ad}
\DeclareMathOperator{\ad}{ad}
\DeclareMathOperator{\vol}{vol}
\newcommand{\Dir}{\mathrm n}
\newcommand{\Neu}{\mathrm t}
\newcommand{\loc}{\mathrm{loc}}
\newcommand{\Id}{\mathrm{Id}}

\newcommand{\mc}{\mathcal}
\newcommand{\mb}{\mathbb}
\newcommand{\mf}{\mathfrak}

\newcommand{\Lie}{\mf L}

\newcommand{\pB}{{\partial B^4}}
\newcommand{\pX}{{\partial X}}
\newcommand{\Ap}{{A_\partial}}
\newcommand{\ap}{{a_\partial}}

\newcommand{\tA}{{\tilde A}}
\newcommand{\ta}{{\tilde a}}
\newcommand{\tB}{{\tilde B}}
\newcommand{\tb}{{\raisebox{0pt}[0pt]{$\tilde b$}}}

\newcommand{\abs}[1]{\left\lvert{#1}\right\rvert}
\newcommand{\norm}[1]{\left\lVert{#1}\right\rVert}
\newcommand{\pair}[2]{\left\langle{#1},{#2}\right\rangle}

\newcommand{\LL}[3]{{L^{#1}_{#2}(#3)}}
\newcommand{\Le}[3]{{\LL{#1}{#2}{{#3};\ext T^*{#3}}}}
\newcommand{\LEe}[3]{{\LL{#1}{#2}{{#3};M_N\otimes\ext T^*{#3}}}}

\newcommand{\La}[3]{{\LL{#1}{#2}{{#3};\mf g\otimes T^*{#3}}}}
\newcommand{\Laa}[3]{{\LL{#1}{#2}{{#3};\ad P\otimes T^*{#3}}}}
\newcommand{\LF}[3]{{\LL{#1}{#2}{{#3};\mf g\otimes\forms2T^*{#3}}}}
\newcommand{\LFF}[3]{{\LL{#1}{#2}{{#3};\ad P\otimes\forms2T^*{#3}}}}
\newcommand{\Lna}[3]{{\LL{#1}{#2}{{#3};M_N\otimes T^*{#3}\otimes T^*{#3}}}}
\newcommand{\Lg}[3]{{\LL{#1}{#2}{{#3};\mf g}}}
\newcommand{\LEa}[3]{{\LL{#1}{#2}{{#3};M_N\otimes T^*{#3}}}}

\newcommand{\LE}[3]{{\LL{#1}{#2}{{#3};M_N}}}
\newcommand{\LG}[3]{{\LL{#1}{#2}{{#3};G}}}
\newcommand{\LGG}[3]{{\LL{#1}{#2}{{#3};\Ad P}}}

\newcommand{\LB}[2]{{\LL{#1}{#2}{B^4}}}

\newcommand{\LeB}[2]{{\Le{#1}{#2}{B^4}}}
\newcommand{\LaB}[2]{{\La{#1}{#2}{B^4}}}
\newcommand{\LFB}[2]{{\LF{#1}{#2}{B^4}}}
\newcommand{\LgB}[2]{{\Lg{#1}{#2}{B^4}}}
\newcommand{\LEaB}[2]{{\LEa{#1}{#2}{B^4}}}

\newcommand{\LEB}[2]{{\LE{#1}{#2}{B^4}}}
\newcommand{\LGB}[2]{{\LG{#1}{#2}{B^4}}}
\newcommand{\LnaB}[2]{{\Lna{#1}{#2}{B^4}}}

\newcommand{\std}{\mathrm{std}}
\newcommand{\Bs}{B^4_\std}
\newcommand{\LBs}[2]{{\LL{#1}{#2}{\Bs}}}
\newcommand{\LaBs}[2]{{\La{#1}{#2}{\Bs}}}

\newcommand{\LpB}[2]{{\LL{#1}{#2}\pB}}

\newcommand{\LepB}[2]{{\Le{#1}{#2}\pB}}
\newcommand{\LapB}[2]{{\La{#1}{#2}\pB}}

\newcommand{\LgpB}[2]{{\Lg{#1}{#2}\pB}}

\newcommand{\LGpB}[2]{{\LG{#1}{#2}{\pB}}}

\newcommand{\LX}[2]{{\LL{#1}{#2}X}}

\newcommand{\LeX}[2]{{\Le{#1}{#2}X}}
\newcommand{\LaX}[2]{{\Laa{#1}{#2}X}}
\newcommand{\LFX}[2]{{\LFF{#1}{#2}X}}
\newcommand{\LpX}[2]{{\LL{#1}{#2}\pX}}
\newcommand{\LepX}[2]{{\Le{#1}{#2}\pX}}

\newcommand{\dd}[2][{}]{{\frac{d{#1}}{d{#2}}}}
\newcommand{\tdd}[2][{}]{{\tfrac{d{#1}}{d{#2}}}}

\newcommand{\normlb}[3]{\norm{#1}_{\LB{#2}{#3}}}
\newcommand{\normlbs}[3]{\norm{#1}_{\LBs{#2}{#3}}}
\newcommand{\normlx}[3]{\norm{#1}_{\LX{#2}{#3}}}

\newcommand{\normlpb}[3]{\norm{#1}_{\LpB{#2}{#3}}}
\newcommand{\normlpx}[3]{\norm{#1}_{\LpX{#2}{#3}}}

\newcommand{\pairlb}[2]{\pair{#1}{#2}_{\LB2{}}}
\newcommand{\pairlpb}[2]{\pair{#1}{#2}_{\LpB2{}}}
\newcommand{\pairlx}[2]{\pair{#1}{#2}_{\LX2{}}}

\begin{document}

%%%todo:
%%%Tristan Riviere might have an arxiv paper in the future about the Yang Mills Plateau problem in all dimensions
%%%add theorem numbers for isobe marini results

%%%thoughts:
%%%Fukaya's comment about Donaldson using alternating method, doesn't seem to fit well alas, reference: connections cohomology and the intersection forms of 4-manifolds
%%%Uhlenbeck, page 3 says that gauge multiplication is not continuous
   %%% maybe also Freed/Uhlenbeck, didn't find it on first look though

%%%done:
%%%cite topology of sobolev bundles by isobe for gauge multiplication being continuous and in introduction, also L^2_d, actually Shevchishin, also Lemma 2.2
    %also Limit Holonomy and Extension Properties of Sobolev and Yang-Mills Bundles for L^2_d

\begin{abstract}
  We develop an analog of harmonic replacement in the gauge theory context. The idea behind harmonic replacement dates back to Schwarz and Perron. The technique, as introduced by Jost and further developed by Colding and Minicozzi, involves taking a map $v\colon\Sigma\to M$ defined on a surface $\Sigma$ and replacing its values on a small ball $B^2\subset\Sigma$ with a harmonic map $u$ that has the same values as $v$ on the boundary $\partial B^2$. The resulting map on $\Sigma$ has lower energy, and repeating this process on balls covering $\Sigma$, one can obtain a global harmonic map in the limit. We develop the analogous procedure in the gauge theory context. We take a connection $B$ on a bundle over a four-manifold $X$, and replace it on a small ball $B^4\subset X$ with a Yang--Mills connection $A$ that has the same restriction to the boundary $\pB$ as $B$. As in the harmonic replacement results of Colding and Minicozzi, we have bounds on the difference $\normlb{B-A}21^2$ in terms of the drop in energy, and we only require that the connection $B$ have small energy on the ball, rather than small $C^0$ oscillation. Throughout, we work with connections of the lowest possible regularity $\LX21$, the natural choice for this context, and so our gauge transformations are in $\LX22$ and therefore almost but not quite continuous, leading to more delicate arguments than in higher regularity.
\end{abstract}

\maketitle

\section{Introduction}
The goal of this article is to adapt Jost's harmonic replacement technique \cite{j91} to the gauge theory context. In the classical harmonic replacement techniques of Schwarz \cite{s70} and Perron \cite{p23}, given a real-valued function $v$ on a domain $\Omega$ and a ball $B^n\subset\Omega$, a function $u$ on $\Omega$ is constructed by replacing $v$ on $B^n$ with a harmonic function with the same values on the boundary of $B^n$. In other words, outside of $B^n$, $u$ is equal to $v$, and on $B^n$, $u$ is the solution to the Dirichlet problem
\begin{align*}
  \Delta u&=0\text{ on }B^n,&u\vert_{\partial B^n}&=v\vert_{\partial B^n}\text{ on }\partial B^n.
\end{align*}
This procedure decreases energy, and, repeating this process for balls covering $\Omega$, one can obtain a harmonic function on all of $\Omega$. In \cite{j91}, Jost adapts this technique to the nonlinear context of maps $v\colon\Sigma\to M$ from a two-dimensional surface $\Sigma$ to a manifold $M$, where he replaces $v$ on a small ball $B^2\subset\Sigma$ with a harmonic map $u$. In our main result, Theorem \ref{globalymreplacementtheorem}, we do the analogous construction for connections on a principal $G$-bundle over a compact four-manifold $X$, where the Lie group $G$ is compact. That is, given such a connection $B$ and a $4$-ball $B^4\subset X$, we construct a connection $A$ by replacing $B$ on $B^4$ with a Yang--Mills connection whose restriction to the boundary $\pB$ is equal to that of $B$. We also show that there is an energy-monotone homotopy between the original connection $A$ and the Yang--Mills replacement $B$.

Like in Colding and Minicozzi's harmonic replacement results \cite{cm08}, we only require that the connection $B$ have small energy on the ball, rather than small $C^0$ oscillation as in Jost's original work. Our result applies to connections in the Sobolev space $\LX21$, which is the natural choice in four dimensions, but leads to more delicate arguments than for smooth connections. In particular, our gauge transformations are in the borderline Sobolev space $\LX22$, so we do not have a Sobolev embedding $\LX22\not\hookrightarrow C^0(X)$, as a result of which the group of $\LX22$ gauge transformations is not a Hilbert Lie group. However, working with smooth connections would be insufficient for our purposes, because, after replacing a smooth connection with a Yang--Mills connection on a ball $B^4\subset X$, the resulting connection is not smooth across the boundary $\pB$. 

In Theorem \ref{familyymreplacementtheorem}, we reframe our main result Theorem \ref{globalymreplacementtheorem} for compact families of connections. Namely, we show that we can homotope a compact family of connections to a family of connections that is Yang--Mills on a ball $B^4\subset X$, again with energy monotonicity. One should think of the compact family as representing a homology or homotopy class. Applying harmonic replacement to a compact family of maps is a key step in \cite{cm08}, where Colding and Minicozzi apply harmonic replacement to a one-parameter family of maps $v_t\colon\Sigma^2\to M^3$ representing a sweep-out of $M^3$ in order to prove finite extinction time of Ricci flow on homotopy $3$-spheres. In the gauge theory context, mapping compact families of connections to compact families of Yang--Mills connections is a way to relate the topology of the moduli space of anti-self-dual Yang--Mills connections to the much better understood space of all connections modulo gauge, as seen in work of Taubes \cite{t84,t88,t89} and Donaldson \cite{d93}. In turn, the topology of the moduli space of anti-self-dual Yang--Mills connections gives rise to Donaldson invariants, which have myriad applications and have been used to show that a topological manifold has no smooth structures \cite{d83} or infinitely many smooth structures \cite{fm88}. In addition, recent work of Feehan and Leness \cite{fl14b} relates Donaldson invariants to Seiberg-Witten invariants \cite{w94}.

Another potential source of applications of Yang--Mills replacement arises from its similarity to Yang--Mills gradient flow. Both Yang--Mills replacement and Yang--Mills gradient flow give energy-decreasing paths of connections. Yang--Mills gradient flow has been extensively studied recently \cite{f14a,s94,w16}, and perhaps the local nature of each replacement step and the greater control afforded by the choice of the balls $B^4\subset X$ will lead to the use of Yang--Mills replacement as an alternative to or in conjunction with Yang--Mills gradient flow.

In this article, we perform Yang--Mills replacement on a single ball $B^4\subset X$, but we can repeat this process on balls covering the compact manifold $X$. Repeating this process indefinitely, one would like to pass to the limit to obtain a global Yang--Mills connection, but doing so is a delicate matter and is the natural direction in which to continue this work. The difficulty arises because, although energy is decreasing, it may concentrate around finitely many points. Because the Yang--Mills replacement theorems above require there to be small energy on the ball $B^4$, to continue the replacement process indefinitely, we would need to choose balls whose radii shrink to zero. This bubbling behavior is a common feature of all of the nonlinear contexts discussed above and has been studied for sequences of maps on surfaces \cite{su81}, for general minimizing sequences of connections \cite{s82,m92}, and for Yang--Mills gradient flow \cite{f14a,w16}. Based on Sedlacek's work \cite{s82} for closed manifolds $X$ and Marini's work \cite{m92} for manifolds with boundary, one expects a weak limit of the sequence of connections to exist, but potentially on a different bundle. It is natural to ask if one can say more about the bubbling behavior for Yang--Mills replacement.

This article is structured as follows. We present preliminaries in Section \ref{preliminaries}. The key propositions that lead to our main results are in Section \ref{ymreplacementsection}. Supporting results on gauge fixing are in Section \ref{gaugefixingchapter}. The proofs of our main results Theorems \ref{globalymreplacementtheorem} and \ref{familyymreplacementtheorem} have a local and a global component. The local component, where we restrict our attention to the ball $B^4$, is Theorems \ref{ymreplacementlowenergy} and \ref{interpolationthm}, in which we construct the Yang--Mills replacement $B$ for a connection $A$ and the energy-monotone homotopy between them, respectively. However, when we consider the ball as a subset of a larger manifold $X$, there is a regularity issue across $\pB$ when we replace $B$ with $A$. This regularity issue is resolved with a global result, Theorem \ref{l2disl21}.

In the remainder of the introduction, we summarize the results contained in this article. We begin by showing the existence of a Yang--Mills connection $A$ on the ball $B^4$ with prescribed small boundary values $\Ap$ in Theorem \ref{dirichletproblem}, a key component of Theorem \ref{ymreplacementlowenergy}. For our work, the restriction to the boundary of the Yang--Mills connection $A$ must be equal to $\Ap$. This problem is different from the one solved by Marini \cite{m92}, where the restriction to the boundary of the Yang-Mills connection is required to be gauge equivalent to $\Ap$. The freedom gained from allowing homotopically nontrivial gauge transformations on $\partial B^4$ leads to a rich family of solutions found by Isobe and Marini \cite{im97,im12a,im12b,im10}. However, for the purposes of Yang--Mills replacement, we must require the stronger condition that the restriction of $A$ to the boundary is equal to $\Ap$. Viewing the ball $B^4$ as a subset of a manifold $X$, allowing a homotopically nontrivial gauge transformation on $\partial B^4$ would let us change the topology of the bundle over $X$. Requiring equality, on the other hand, ensures that the connection on $X$ resulting from Yang--Mills replacement is on the same bundle $P\to X$ as the original connection.

Both the work of Marini and Isobe \cite{im12a,im12b,im10} and this article deal with solving the Yang--Mills problem with prescribed small boundary values; however, the notion of ``small'' is different. In \cite{im12a}, the authors show that, for any smooth boundary value $\Ap=d+\ap$, there is an $\epsilon$ depending on $\Ap$ such that there exists a solution with boundary value $d+\epsilon'\cdot\ap$ for any $\epsilon'<\epsilon$. We prove the stronger claim that there is a uniform $\epsilon$ such that a solution with boundary value $\Ap$ exists whenever $\normlpb\ap2{1/2}<\epsilon$. Requiring only a bound on the $\LpB2{1/2}$ norm of the connection is needed for Theorem \ref{ymreplacementlowenergy}, where we improve Theorem \ref{dirichletproblem} to only require an energy bound. This energy bound gives us the energy bound of our main result, Theorem \ref{globalymreplacementtheorem}.

The claim of Theorem \ref{dirichletproblem} has been discussed by Rivi\`ere \cite{r14}, but instead of a direct energy minimization method, we use the inverse function theorem, allowing us to conclude smooth dependence of the solution $A$ on the boundary value $\Ap$. One of the key ideas in the inverse function theorem argument is an appropriate choice of codomain, motivating the definition of $\LeB{2,n}{-1}$ in Section \ref{subsectionl2dir} as the dual of those forms in $\LeB21$ that are normal to the boundary $\pB$. Using the gauge fixing results of Section \ref{gaugefixingchapter}, in Theorems \ref{ymreplacementlowenergy} and \ref{ymuniqueness}, we broaden the hypotheses of Theorem \ref{dirichletproblem} and prove a uniqueness result for the solutions. Along the way, in Proposition \ref{ymreplacementinequality}, we prove an energy convexity result similar to Colding and Minicozzi's result in the harmonic map context \cite{cm08}. Next, in Theorem \ref{interpolationthm}, we construct the local version of the energy-monotone homotopy of Theorem \ref{globalymreplacementtheorem}.

As we proceed to the global question of Yang--Mills replacement on a ball in a larger manifold, we run into regularity issues across the boundary of the ball that are not present in the harmonic map setting. In Theorem \ref{dirichletproblem}, we are able to prescribe the tangential component $i^*A$ of $A$ on the boundary $\pB$, but not the normal component. As a result, when we take a global $\LX21$ connection $B$ and construct a connection $\hat A$ that is a Yang--Mills connection $A$ on the ball $B^4$ and equal to $B$ outside the ball, the tangential components of $A$ and $B$ match on the boundary $\pB$, but the normal components might not. Thus, the resulting piecewise-defined global connection $\hat A$ is not in $\LX21$, but it is nonetheless in a space we call $\LeX2d$, defined in Section \ref{l2dsection} as those forms $\alpha$ such that $\alpha\in\LeX4{}$ and $d\alpha\in\LeX2{}$. However, losing regularity after a Yang--Mills replacement step is unsatisfactory because it prevents us from repeating the Yang--Mills replacement process on an overlapping ball. Because $\LX2d$ connections have well-defined $\LX2{}$ curvatures, one might ask if a gauge fixing argument could show that they are gauge equivalent to $\LX21$ connections. We answer this question in the affirmative in Corollary \ref{homeomorphicconfigurationspaces}, showing that, furthermore, the corresponding spaces of connections modulo gauge transformations are homeomorphic and Hausdorff. The argument is delicate because the gauge transformations involved are in the borderline Sobolev space $\LX41$ and not continuous. The proofs of Theorems \ref{globalymreplacementtheorem} and \ref{familyymreplacementtheorem} follow in Section \ref{mainproofssection}.

Finally, in Section \ref{gaugefixingchapter}, we develop the gauge fixing machinery that powers the results in Section \ref{ymreplacementsection}, building off of results by Uhlenbeck \cite{u82,u82b} and Marini \cite{m92}. In gauge fixing, we start with an $\LB21$ connection $A$ on a ball $B^4$ with small energy $\normlb{F_A}2{}<\epsilon$, and we find an $\LGB22$ gauge transformation $g$ that sends $A$ to a connection $\tA=d+\ta$ satisfying the Coulomb condition $d^*\ta=0$ and a bound on $\normlb{\ta}21$. However, there are two natural boundary conditions to impose on the connection, either the Neumann conditions $i^*{*\ta}=0$ or the Dirichlet conditions $d^*_\pB i^*\ta$, where $i^*$ is the restriction to the boundary. Uhlenbeck \cite{u82} provides a full treatment of the gauge fixing problem with Neumann boundary conditions, but her treatment of the problem with Dirichlet boundary conditions in \cite[Theorem 2.7]{u82b} and the later improvement by Marini \cite[Theorem 3.2]{m92} have additional regularity assumptions on $A$. We prove the result without these assumptions in Theorem \ref{dirichletcoulombfixing}. Along the way, in Theorem \ref{l2dgaugefixing} we extend Uhlenbeck's gauge fixing result to $\LB2d$ connections, in Proposition \ref{gaugefixingcontinuous} we improve the weak $\LaB21$ convergence of the Coulomb gauge representatives in \cite{u82} to strong $\LaB21$ convergence, and, in Proposition \ref{gaugefixingconstantboundary}, we prove the Coulomb gauge fixing result where we impose Dirichlet boundary conditions on the gauge transformation instead of on the connection.

\section{Preliminaries}\label{preliminaries}
In Section \ref{connectionsprelim}, we review connections and prove basic properties of the borderline regularity gauge group. Next, in Section \ref{hodgeprelim}, we review the Hodge decomposition theorem for manifolds with boundary. Finally, in Sections \ref{subsectionl2dir} and \ref{l2dsection}, we define and prove basic properties of the Banach spaces $\LeX{2,\Dir}{-1}$ and $\LeX2d$, where $d$ denotes the exterior derivative.

\subsection{Yang--Mills connections}\label{connectionsprelim}
Following the standard references \cite{kn63,dk90}, we introduce the notation we will use for principal $G$-bundles and connections. Let $G$ be a compact Lie group. We fix a unitary representation $G\hookrightarrow U_N\subset M_N$, where $M_N$ denotes the vector space of $N$ by $N$ complex matrices.
\begin{definition}
  Let $P\to X$ be a principal $G$-bundle over a compact manifold $X$, and let $\ad P$ denote the associated bundle $P\times_G\mf g$. Let $A$ be an $\LX21$ connection, and let $F_A\in\LL2{}{X;\ad P\otimes\forms2T^*X}$ be its curvature. The \emph{energy} of $A$ is
  \begin{equation*}
    \mc E(A)=\tfrac12\int_X\abs{F_A}^2=\tfrac12\normlx{F_A}2{}^2.
  \end{equation*}
\end{definition}

\begin{definition}
  A \emph{Yang--Mills connection} $A$ is a critical point of the functional $\mc E$. If $X$ has boundary, then we require $A$ to be a critical point with respect to variations $A_t$ such that $i^*A_t$ is gauge equivalent to $i^*A$ on the boundary, where $i\colon\pX\hookrightarrow X$ is the inclusion.
\end{definition}

Using variations that are fixed on the boundary, we see that a Yang--Mills connection $A$ satisfies the \emph{Yang--Mills equations}
\begin{equation*}
  \pairlx{F_A}{d_Ac}=0
\end{equation*}
for all $c\in\LL21{X;\ad P\otimes T^*X}$ with $i^*c=0$ on $\pX$. When we are working over a local trivialization $d$ of $P$ over $B^4\subset X$, we will also make use of the \emph{projected Yang--Mills equations}, where we only require that $\pairlb{F_A}{d_Ac}=0$ for $c\in\LaB21$ satisfying $d^*c=0$ on $B^4$ in addition to $i^*c=0$ on $\pB$.

We now review gauge transformations and their action on connections.

\begin{definition}
  A \emph{gauge transformation} is an automorphism of $P$. A gauge transformation can be represented by a section of the associated bundle $\Ad P=P\times_G G\subset P\times_G M_N$ with the conjugation action of $G$ on $G\subset M_N$. By an $L^p_k$ gauge transformation we mean an $L^p_k$ section $g$ of the vector bundle $P\times_G M_N$ such that $g(x)\in\Ad P$ a.e.\ on $X$.
\end{definition}

With respect to a local trivialization of $P$ over $B^4\subset X$, a gauge transformation is a $G$-valued function, and we can write down how explicitly how it acts on a connection $A$ expressed in this trivialization as $d+a$, where $a$ is a $\mf g$-valued one-form. We have
\begin{equation*}
  g(A)=d+gag^{-1}-(dg)g^{-1}.
\end{equation*}
Writing $g(A)=B=d+b$, we can rewrite the above equation as
\begin{equation*}
  dg=ga-bg,
\end{equation*}
where the terms in the equation are $M_N$-valued one-forms.

When $(k+1)p>\dim X$, it is well-known \cite{fu84,u82} that the group of $L^p_{k+1}$ gauge transformations has smooth multiplication and inversion and acts smoothly on $L^p_k$ connections, using the multiplication map $L^p_{k+1}\times L^p_{k+1}\to L^p_{k+1}$ and the Sobolev embedding $L^p_{k+1}\hookrightarrow C^0$.

In the borderline case $(k+1)p=\dim X$, the matter is more delicate. Because gauge transformations are $G$-valued and $G$ is compact, they are uniformly bounded in $L^\infty(X;P\times_GM_N)$. As a result, multiplication of borderline $L^p_{k+1}$ gauge transformations is still well-defined, as is their action on $L^p_k$ connections. However, these maps are only smooth with the $L^p_{k+1}\cap L^\infty$ topology on gauge transformations. With just the $L^p_{k+1}$ topology, the situation is more subtle. With this topology, multiplication of gauge transformations and the action of gauge transformations on connections are no longer smooth maps, but nonetheless they are continuous. We will prove this claim for $L^2_2$ gauge transformations on a $4$-manifold, but the argument works for general borderline regularity gauge groups. The key idea that gives us just enough power to prove continuity is that $G$-valued functions act as isometries on $L^p$ spaces. We begin with a lemma.

\begin{lemma}\label{weaktostronglemma}
  Let $P$ be a principal $G$-bundle over a compact $4$-manifold with compact gauge group $G\hookrightarrow U_N\subset M_N$. Let $1<p<\infty$, and let $f_i$ be a general matrix-valued sequence in $L^p(X;P\times_GM_N)$ that converges in $L^p(X;P\times_GM_N)$ to $f$. Let $\frac1q+\tfrac1p=1$, and let $g_i$ be a sequence of gauge transformations that converges in $L^q(X;\Ad P)$ to a gauge transformation $g$. Then $g_if_i$ and $f_ig_i$ converge in $L^p(X;P\times_GM_N)$ to $gf$ and $fg$, respectively.
  \begin{proof}
    We know immediately that $g_if_i$ converges to $gf$ in $L^1(X;P\times_GM_N)$. In addition, because $G$ is compact, we know that $g_i\in L^\infty(X;P\times_GM_N)$, so $g_if_i\in L^p(X;P\times_GM_N)$. However, because the $g_i$ do not in general converge in $L^\infty(X;P\times_GM_N)$, it is not immediately obvious that $g_if_i$ converges in $L^p(X;P\times_GM_N)$.

    We show that the $g_if_i$ converges in $L^p(X;P\times_GM_N)$ to $gf$ by showing that every subsequence of the $g_if_i$ has a further subsequence that converges to $gf$. We begin by passing to a subsequence of the $g_if_i$. The $g_i$, being $\Ad P$-valued a.e., are uniformly bounded in $L^\infty(X;P\times_GM_N)$. As a result, the $g_if_i$ are uniformly bounded in $L^p(X;P\times_GM_N)$, and hence after passing to a further subsequence the $g_if_i$ converge weakly in $L^p(X;P\times_GM_N)$. This weak limit must be $gf$ because we know that $g_if_i$ converges to $gf$ in $L^1(X;P\times_GM_N)$. To upgrade this weak convergence to strong convergence, we note that multiplication by an element of $G$ is an isometry of $M_N$, and hence
    \begin{equation*}
      \normlx{g_if_i}p{}=\normlx{f_i}p{}\to\normlx{f}p{}=\normlx{gf}p{}
    \end{equation*}
    For $L^p$ spaces with $1<p<\infty$ \cite{r13}, weak convergence along with convergence of the sequence of norms to the norm of the limit implies strong convergence. We conclude then that this further subsequence of the $g_if_i$ converges strongly in $L^p(X;P\times_GM_N)$ to $gf$, and hence so does the original sequence. The argument for $f_ig_i$ is analogous.
  \end{proof}
\end{lemma}

We can now prove the continuity of the gauge group and its action on connections. See also \cite[Lemma 2.1]{s02} and \cite[Corollary 6.1]{i09}.

\begin{proposition}\label{borderlinecontinuity}
  Let $P$ be a principal $G$-bundle over a compact $4$-manifold $X$ with compact group $G\hookrightarrow U_N\subset M_N$. The group of $\LX22$ gauge transformations has continuous multiplication and inversion maps, and $\LX22$ gauge transformations act continuously on $\LX21$ connections.
  \begin{proof}
    The question is local, so for simplicity we work over a closed ball in a local trivialization $B^4\subset X$. Consider a sequence of $\LGB22$ gauge transformations $g_i$ and $h_i$ converging in $\LGB22$ to $g$ and $h$, respectively. We aim to show that $g_ih_i$ converges to $gh$ in $\LGB22$.

    We compute that
    \begin{equation*}
      \nabla^2(g_ih_i)=(\nabla^2g_i)h_i+2(\nabla g_i)(\nabla h_i)+g_i(\nabla^2h_i).
    \end{equation*}
    For the middle term, we know that the sequences $\nabla g_i$ and $\nabla h_i$ converge in $\LEaB21$, and we have a Sobolev multiplication map $L^2_1\times L^2_1\to L^2$. For the other two terms, we apply Lemma \ref{weaktostronglemma}. We know that $g_i$ converges to $g$ in $\LGB22$ and hence in $\LGB2{}$, and we know that $\nabla^2h_i$ converges in $\LnaB2{}$ to $\nabla^2h$, so Lemma \ref{weaktostronglemma} tells us that $g_i(\nabla^2h_i)$ converges in $\LnaB2{}$ to $g(\nabla^2h)$. Likewise, $(\nabla^2 g_i)h_i$ converges in $\LnaB2{}$ to $(\nabla^2 g)h$.

    Inversion is continuous by a simpler argument. Because we have chosen a representation $G\hookrightarrow U_N$, inversion is the same as the conjugate transpose, which is a linear, and hence smooth, map $\LEB22\to\LEB22$. Using this fact, we can show that the action of gauge transformations on connections is continuous.

    We would like to show that the map $g(a)$ defined by
    \begin{align*}
      \LGB22\times\LaB21&\to\LaB21\\
      (g,a)&\mapsto gag^{-1}-(dg)g^{-1}
    \end{align*}
    is continuous. Again, we choose sequences $g_i$ and $a_i$ that converge in $\LGB22$ and $\LaB21$ to $g$ and $a$, respectively, and we aim to show that the $g_i(a_i)$ converge to $g(a)$ in $\LaB21$. We compute
    \begin{equation*}
      \nabla(g_i(a_i))=(\nabla g_i)a_ig_i^{-1}+g_i(\nabla a_i)g_i^{-1}+g_ia_i(\nabla g_i^{-1})-(\nabla dg_i)g_i^{-1}-dg_i(\nabla g_i^{-1}).
    \end{equation*}
    Using the Sobolev multiplication theorems, we know that $(\nabla g_i)a_i$, $\nabla a_i$, $a_i\nabla g_i^{-1}$, $\nabla dg_i$, and $dg_i(\nabla g_i^{-1})$ all converge in $\LnaB2{}$ to the expected limits. As a result, we can prove convergence of the $\nabla(g_i(a_i))$ in $\LnaB2{}$ by several applications of Lemma \ref{weaktostronglemma}.
  \end{proof}
\end{proposition}

\subsection{The Hodge decomposition theorem on manifolds with boundary}\label{hodgeprelim}

In this section, we summarize the treatment in \cite[Section 5.9]{t96} of appropriate boundary conditions for the Hodge Laplacian. Let $X$ be a smooth manifold with boundary $\pX$, and let $i\colon\pX\to X$ be the inclusion.

\begin{definition}\label{l21dir}
  Let $\LL{2,\Dir}1{X;\ext T^*X}$ denote the $\LX21$ differential forms $\alpha$ that are normal to the boundary, that is, they satisfy the Dirichlet boundary condition $i^*\alpha=0$. Likewise, let $\LL{2,\Neu}1{X;\ext T^*X}$ denote the $\LX21$ differential forms $\alpha$ that are tangent to the boundary, that is, they satisfy the Neumann boundary condition $i^*{*\alpha}=0$, where $*$ is Hodge star operator.
\end{definition}

\begin{definition}
  Let $\mc H^\Dir$ denote the harmonic forms in $\LeX{2,\Dir}1$. That is, $\mc H^\Dir$ contains those $\LeX21$ forms $\alpha$ such that $i^*\alpha=0$, $d\alpha=0$, and $d^*\alpha=0$. Likewise, let $\mc H^\Neu$ denote those $\LeX21$ forms $\alpha$ such that $i^*{*\alpha}=0$, $d\alpha=0$, and $d^*\alpha=0$.
\end{definition}

\begin{proposition}[{\cite[5.9.36, 5.9.38]{t96}}]
  The forms in $\mc H^\Dir$ and $\mc H^\Neu$ are smooth.
\end{proposition}

\begin{proposition}[{\cite[5.9.9]{t96}}]\label{cohomology}
  The natural map from the Dirichlet harmonic forms into the cohomology of $X$ rel boundary is an isomorphism. That is, $\mc H^\Dir\cong H^*(X,\pX)$. Likewise, $\mc H^\Neu\cong H^*(X)$.
\end{proposition}

\begin{proposition}[{\cite[5.9.8]{t96}}]\label{hodgedecomposition}
  There exists Green's functions
  \begin{align*}
    G^\Dir&\colon\LeX2{}\to\LeX22,\text{ and}\\
    G^\Neu&\colon\LeX2{}\to\LeX22
  \end{align*}
  such that:
  \begin{enumerate}
  \item For all $\LeX2{}$ differential forms $\alpha$,
    \begin{align*}
      \alpha&=dd^*G^\Dir\alpha+d^*dG^\Dir\alpha+\pi_{\mc H}^\Dir\alpha,\\
      \alpha&=dd^*G^\Neu\alpha+d^*dG^\Neu\alpha+\pi_{\mc H}^\Neu\alpha,
    \end{align*}
    where $\pi_{\mc H}^\Dir$ and $\pi_{\mc H}^\Neu$ denote the $\LX2{}$ projections to the finite-dimensional spaces $\mc H^\Dir$ and $\mc H^\Neu$, respectively.
  \item The operators $dd^*G^\Dir$, $d^*dG^\Dir$, and $\pi_{\mc H}^\Dir$ are $\LX2{}$-projections whose ranges are $\LX2{}$-orthogonal to each other. Likewise, the operators $dd^*G^\Neu$, $d^*dG^\Neu$, and $\pi_{\mc H}^\Neu$ are $\LX2{}$-projections whose ranges are $\LX2{}$-orthogonal to each other.
  \item The range of $G^\Dir$ satisfies the boundary conditions $i^*G^\Dir\alpha=0$ and $i^*d^*G^\Dir\alpha=0$.
  \item The range of $G^\Neu$ satisfies the boundary conditions $i^*{*G^\Neu\alpha}=0$ and $i^*d^*{*G^\Neu\alpha}=0$.
  \item For any $k\ge0$, $G^\Dir,G^\Neu\colon\LeX2k\to\LeX2{k+2}$.
  \end{enumerate}
\end{proposition}

\begin{corollary}\label{Di}
  Let $X$ be a smooth manifold with smooth boundary, and let $k\ge0$. Let $\LeX{2,\Dir}{k+1}$ denote those $\LeX 2{k+1}$ forms $\alpha$ such that $i^*\alpha=0$ on $\pX$. Then
  \begin{equation*}
    d+d^*\colon L^{2,\Dir}_{k+1}(X;\ext T^*X)\to L^2_k(X;\ext T^*X)
  \end{equation*}
  has kernel and cokernel $\mc H^\Dir$.

  Likewise, let $\LeX{2,\Neu}{k+1}$ denote those $\LeX2{k+1}$ forms $\alpha$ such that $i^*{*\alpha}=0$ on $\pX$. Then
  \begin{equation*}
    d+d^*\colon L^{2,\Neu}_{k+1}(X;\ext T^*X)\to L^2_k(X;\ext T^*X)
  \end{equation*}
  has kernel and cokernel $\mc H^\Neu$.

  \begin{proof}
    For smooth $\alpha$, the condition $i^*\alpha=0$ implies that 
    \begin{equation*}
      \pairlx{d\alpha}{d^*\alpha}=\pairlx{\alpha}{d^*d^*\alpha}+\int_\pX\alpha\wedge*d^*\alpha=0.
    \end{equation*}
    Smooth forms satisfying $i^*\alpha=0$ are dense in $\LeX{2,\Dir}{k+1}$ as a consequence of \cite[7.54]{a75}, so $\pairlx{d\alpha}{d^*\alpha}=0$ holds for general forms in $\LeX{2,\Dir}{k+1}$. Hence, $(d+d^*)\alpha=0$ implies $d\alpha=0$ and $d^*\alpha=0$, so $\alpha\in\mc H^\Dir$.
    
    Next, we show that the range has trivial intersection with $\mc H^\Dir$. Let $(d+d^*)\alpha=\phi$ and $i^*\alpha=0$, where $\phi\in\mc H^\Dir$. Then
    \begin{multline*}
      \normlx\phi2{}^2=\pairlx{d\alpha}\phi+\pairlx{d^*\alpha}\phi\\
      =\pairlx\alpha{d^*\phi}+\pairlx\alpha{d\phi}+\int_\pX\alpha\wedge*\phi-\phi\wedge*\alpha=0.
    \end{multline*}

    Finally, we show that the range of $d+d^*$ contains all $\beta\in\LeX2k$ where $\beta$ is $L^2$-orthogonal to $\mc H^\Dir$. Proposition \ref{hodgedecomposition} gives us $G^\Dir\colon\LeX2k\to\LeX2{k+2}$ such that $\Delta G^\Dir\beta=\beta$ if $\beta$ is orthogonal to $\mc H^\Dir$. Hence, our desired preimage is $\alpha=(d+d^*)G^\Dir\beta$. By Proposition \ref{hodgedecomposition}, we have boundary conditions $i^*G^\Dir\beta=0$ and $i^*d^*G^\Dir\beta=0$. Hence,
    \begin{equation*}
      i^*\alpha=i^*dG^\Dir\beta+i^*d^*G^\Dir\beta=di^*G^\Dir\beta=0,
    \end{equation*}
    so $\alpha\in\LeX{2,\Dir}{k+1}$, as desired.

    The second claim is analogous, or, alternatively, it follows from the identity
    \begin{equation*}
      d+d^*=(-1)^{n(p-1)+1}{*}(d+d^*){*},
    \end{equation*}
    where $(-1)^{n(p-1)+1}$ acts on $\ext T^*X$ by $(-1)^{n(p-1)+1}$ on the degree $p$ component of the exterior algebra, along with the facts that the isometry $*$ sends $\LeX{2,\Dir}{k+1}$ to $\LeX{2,\Neu}{k+1}$ and vice versa, and $*$ sends $\mc H^\Dir$ to $\mc H^\Neu$ and vice versa.
  \end{proof}
\end{corollary}

\subsection{The space $\LeX{2,\Dir}{-1}$}\label{subsectionl2dir}
The Yang--Mills operator $A\mapsto d_A^*F_A^{}$ is a second-order operator, so if our connection $A$ is in $\LaX21$, then $d_A^*F_A^{}\in\LaX2{-1}$. However, the space $\LX2{-1}$ ends up being insufficient for our purposes. By definition, $\LL2{-1}{X;\ext T^*X}$ is the dual of $\LL21{X;\ext T^*X}_0$, that is, $\LeX21$ forms that vanish on the boundary $\pX$. We need to instead define a larger space, $\LL{2,\Dir}{-1}{X;\ext T^*X}$, as the dual of $\LL{2,\Dir}1{X;\ext T^*X}$, that is, $\LeX21$ forms whose \emph{tangential} components vanish on the boundary $\pX$. In the remainder of the section, we prove basic results about the space $\LeX{2,\Dir}{-1}$ needed to show that $d_A^*F_A^{}$ and $\pi_{d^*}d_A^*F_A^{}$ remain well-defined when their target is $\LaX{2,\Dir}{-1}$ instead of $\LaX2{-1}$.

\begin{definition}
  Let $\LeX{2,\Dir}{-1}$ denote the dual Hilbert space of $\LeX{2,\Dir}1$. (See Definition \ref{l21dir}.)
\end{definition}

\begin{proposition}
  The space $\LeX{2,\Dir}1$ is reflexive, and smooth functions are dense in $\LeX{2,\Dir}{-1}$.
  \begin{proof}
    Since $\LeX{2,\Dir}1$ is a closed subspace of $\LeX21$, the reflexivity of $\LeX{2,\Dir}1$ follows from \cite[1.21]{a75} and the reflexivity of $\LeX21$ \cite[3.5]{a75}. Meanwhile, using the reflexivity of $\LeX{2,\Dir}1$, an argument like in \cite[3.12]{a75} shows that $\LeX2{}$ is dense in $\LeX{2,\Dir}{-1}$, and so $C^\infty(X)$ is dense in $\LeX{2,\Dir}{-1}$ also.
  \end{proof}
\end{proposition}

\begin{lemma}\label{dclosedrange}
  The operators
  \begin{align*}
    d&\colon\LeX21\to\LeX2{},\text{ and}\\
    d^*&\colon\LeX21\to\LeX2{}
  \end{align*}
  have closed ranges.
  \begin{proof}
    The operators
    \begin{align*}
      dd^*G^\Neu&\colon\LeX2{}\to\LeX2{},\text{ and}\\
      d^*dG^\Dir&\colon\LeX2{}\to\LeX2{}
    \end{align*}
    are projections and hence have closed ranges, so we proceed by showing that $\range(d)=\range(dd^*G^\Neu)$ and $\range(d^*)=\range(d^*dG^\Dir)$.

    To show that $\range(d)\subseteq\range(dd^*G^\Neu)$, consider $d\alpha$ for $\alpha\in\LeX21$. By Proposition \ref{hodgedecomposition}, since $d\alpha\in\LeX2{}$, we have an orthogonal decomposition $d\alpha=dd^*G^\Neu d\alpha+d^*dG^\Neu d\alpha+\pi_{\mc H}^\Neu\alpha$. We claim that, in fact, $d\alpha=dd^*G^\Neu d\alpha$. Because the decomposition is orthogonal, we simply check that
    \begin{align*}
      \pair{d\alpha}{d^*dG^\Neu d\alpha}_{\LeX2{}}&=\pair{dd\alpha}{dG^\Neu d\alpha}_{\LeX2{}}-\int_\pX d\alpha\wedge*dG^\Neu d\alpha=0,\\
      \pair{d\alpha}{\pi_{\mc H}^\Neu d\alpha}_{\LeX2{}}&=\pair\alpha{d^*\pi_{\mc H}^\Neu d\alpha}_{\LeX2{}}+\int_\pX\alpha\wedge*\pi_{\mc H}^\Neu d\alpha=0.
    \end{align*}
     Here, we used the boundary conditions $i^*{*dG^\Neu d\alpha}=\pm i^*d^*{*G^\Neu d\alpha}=0$ and $i^*{*\pi_{\mc H}^\Neu d\alpha}=0$, along with the fact that $i^*(d\alpha)$ is well-defined in $\LepX2{-1/2}$ even though $d\alpha\in\LeX2{}$ because of the identity $i^*d\alpha=di^*\alpha$. Hence, $\range(d)=\range(dd^*G^\Neu )$. The argument for $d^*$ is analogous.
  \end{proof}
\end{lemma}

\begin{proposition}\label{dsbounded}
  The operator $d^*\colon C^\infty(X;\ext T^*X)\to\LL{2,\Dir}{-1}{X;\ext T^*X}$ extends to a bounded operator $d^*\colon\LeX2{}\to\LeX{2,\Dir}{-1}$ with closed range.
  \begin{proof}
    Let $f\in C^\infty(X;\ext T^*X)$, and let $\phi\in\LeX{2,\Dir}1$. Because $i^*\phi=0$, we have
    \begin{equation*}
      \pair{d^*f}\phi_{\LX2{}}=\pair f{d\phi}_{\LX2{}}-\int_\pX\phi\wedge*f=\pair f{d\phi}_{\LX2{}}.
    \end{equation*}
    Hence,
    \begin{equation*}
      \abs{\pairlx{d^*f}\phi}\le\norm f_{\LX2{}}\norm\phi_{\LX21},
    \end{equation*}
    so $\norm{d^*f}_{\LX{2,\Dir}{-1}}\le\norm f_{\LX2{}}$. Because $C^\infty$ is dense in $\LeX2{}$, we conclude that $d^*$ extends to a bounded operator $d^*\colon\LeX2{}\to\LeX{2,\Dir}{-1}$, defined by the equation
    \begin{equation*}
      \pairlx{d^*f}\phi=\pair f{d\phi}_{\LX2{}}
    \end{equation*}
    for $f\in\LeX2{}$ and $\phi\in\LeX{2,\Dir}1$.

    By the closed range theorem, $d^*\colon \LeX2{}\to\LeX{2,\Dir}{-1}$ having closed range is equivalent to its transpose $d\colon\LeX{2,\Dir}1\to\LeX2{}$ having closed range. On a larger domain, we know that $d\colon\LeX21\to\LeX2{}$ has closed range by Lemma \ref{dclosedrange}. Because $\LeX21$ is a Hilbert space, $\ker d$ has a closed complement $(\ker d)^\perp$, so $d\colon(\ker d)^\perp\to\range(d)$ is an isomorphism of Banach spaces. Therefore, the image under $d$ of the closed space $(\ker d)^\perp\cap\ker i^*$ is closed. Summing with $\ker d\cap\ker i^*$, we see that $d(\ker i^*)=d((\ker d)^\perp\cap\ker i^*)$, so $d$ also has closed range as an operator $\LeX{2,\Dir}1\to\LeX2{}$. Therefore, by the closed range theorem, $d^*\colon\LeX2{}\to\LeX{2,\Dir}{-1}$ has closed range, as desired.
  \end{proof}
\end{proposition}

\begin{proposition}\label{pibounded}
  The operator
  \begin{equation*}
    \pi_{d^*}=d^*dG^\Dir\colon L^2(X;\ext T^*X)\to L^2(X;\ext T^*X)\to\LL{2,\Dir}{-1}{X;\ext T^*X}
  \end{equation*}
  extends to a bounded operator $\pi_{d^*}\colon\LL{2,\Dir}{-1}{X;\ext T^*X}\to\LL{2,\Dir}{-1}{X;\ext T^*X}$, and this operator is a projection to the range of $d^*\colon\LeX2{}\to\LeX{2,\Dir}{-1}$.
  \begin{proof}
    Let $y\in\LeX2{}$, and let $\phi\in\LeX{2,\Dir}1$. Because $\pi_{d^*}$ is a projection operator to a factor in the Hodge decomposition, we have
    \begin{equation*}
      \pair{\pi_{d^*}y}\phi_{\LX2{}}=\pairlx y{\pi_{d^*}\phi}.
    \end{equation*}
    By Proposition \ref{hodgedecomposition}, $\phi\in\LeX21$ implies $\pi_{d^*}\phi\in\LeX21$. Furthermore, we claim that $\phi\in\LeX{2,\Dir}1$ implies $\pi_{d^*}\phi\in\LeX{2,\Dir}1$. Indeed,
    \begin{equation*}
      0=i^*\phi=i^*(dd^*G^\Dir\phi+d^*dG^\Dir\phi+\pi_{\mc H^\Dir}\phi)=di^*d^*G^\Dir\phi+i^*\pi_{d^*}\phi=i^*\pi_{d^*}\phi
    \end{equation*}
    because $i^*(\mc H^\Dir)=0$ by definition and $i^*d^*G^\Dir=0$ by Proposition \ref{hodgedecomposition}. Hence,
    \begin{equation*}
      \abs{\pair{\pi_{d^*}y}\phi_{\LX2{}}}\le\norm y_{\LX{2,\Dir}{-1}}\norm{\pi_{d^*}\phi}_{\LX21}\le C\norm y_{\LX{2,\Dir}{-1}}\norm\phi_{\LX21},
    \end{equation*}
    for some constant $C$, so
    \begin{equation*}
      \norm{\pi_{d^*}y}_{\LX{2,\Dir}{-1}}\le C\norm y_{\LX{2,\Dir}{-1}}.
    \end{equation*}
    Since $\LeX2{}$ is dense in $\LeX{2,\Dir}{-1}$, we see that $\pi_{d^*}$ extends to a bounded operator $\LeX{2,\Dir}{-1}\to\LeX{2,\Dir}{-1}$, defined by the equation
    \begin{equation*}
      \pairlx{\pi_{d^*}y}\phi=\pairlx y{\pi_{d^*}\phi}.
    \end{equation*}

    To show that the operator $\pi_{d^*}\colon\LeX{2,\Dir}{-1}\to\LeX{2,\Dir}{-1}$ is a projection to the range of $d^*\colon\LeX2{}\to\LeX{2,\Dir}{-1}$, first recall that in the proof of Lemma \ref{dclosedrange} we showed that the range of $d^*\colon\LeX21\to\LeX2{}$ is equal to the range of the projection $\pi_{d^*}\colon\LeX2{}\to\LeX2{}$. Hence, if $y\in d^*(\LeX21)$, then $y=\pi_{d^*}y$. Since $\LeX21$ is dense in $\LeX2{}$ and $d^*\colon\LeX2{}\to\LeX{2,\Dir}{-1}$ is bounded, we know that $d^*(\LeX21)$ is dense in $d^*(\LeX2{})\subset\LeX{2,\Dir}{-1}$. Because $\pi_{d^*}\colon\LeX{2,\Dir}{-1}\to\LeX{2,\Dir}{-1}$ is continuous, the equation $y=\pi_{d^*}y$ remains true for all $y\in d^*(\LeX2{})$. Thus, $d^*(\LeX2{})\subseteq\pi_{d^*}\left(\LeX{2,\Dir}{-1}\right)$.

    Conversely, for all $y\in\LeX2{}$, we know that $\pi_{d^*}y$ is in $d^*(\LeX21{})$ and hence in $d^*(\LeX2{})$. We know that $\LeX2{}$ is dense in $\LeX{2,\Dir}{-1}$, that our operator $\pi_{d^*}\colon\LeX{2,\Dir}{-1}\to\LeX{2,\Dir}{-1}$ is continuous, and that $d^*(\LeX2{})$ is closed in $\LeX{2,\Dir}{-1}$ by Lemma \ref{dsbounded}, so $\pi_{d^*}y\in d^*(\LeX2{})$ remains true for all $y\in\LeX{2,\Dir}{-1}$. Hence, $\pi_{d^*}\left(\LeX{2,\Dir}{-1}\right)=d^*(\LeX2{})$, and the above fact that $y=\pi_{d^*}y$ for all $y\in d^*(\LeX2{})$ now implies that $\pi_{d^*}$ is a projection, as desired.
  \end{proof}
\end{proposition}

\begin{corollary}\label{pymsmooth}
  The operator $A\mapsto\pi_{d^*}d_A^*F_A^{}$ is well-defined and smooth as an operator
  \begin{equation*}
    \LaX21\to\LaX{2,\Dir}{-1}\cap\range(d^*).
  \end{equation*}

  \begin{proof}
    We know that $A\mapsto F_A$ is smooth as an operator
    \begin{equation*}
      \LaX21\to\LFX2{}.
    \end{equation*}
    By Proposition \ref{dsbounded}, $d^*\colon\LFX2{}\to\LaX{2,\Dir}{-1}$ is a bounded linear operator, and hence smooth. The multiplication map $\LX21\times\LX2{}\times\LX{2,\Dir}1\to\LX1{}\to\mb R$ is bounded, and hence by duality so is the bilinear multiplication map $\LX21\times\LX2{}\to\LX{2,\Dir}{-1}$, so $A\mapsto[a\wedge]^*F_A$ is smooth as a map $\LaX21\to\LaX{2,\Dir}{-1}$. Thus $A\mapsto d_A^*F_A^{}=d^*F_A+[a\wedge]^*F_A$ is smooth as a map $\LaX21\to\LaX{2,\Dir}{-1}$. Finally, by Proposition \ref{pibounded}, $\pi_{d^*}\colon\LaX{2,\Dir}{-1}\to\LaX{2,\Dir}{-1}$ is a bounded linear operator, and hence smooth, and its range is $\range(d^*)$.
  \end{proof}
\end{corollary}

\subsection{The space $\LeX2d$}\label{l2dsection}
Given an $\LaX21$ connection $B$ on $X$ and a ball $B^4$ in $X$, we will replace $B$ with a $\LaB21$ Yang--Mills connection $A$ on $B^4$ whose tangential components match $B$ on $\pB$. The resulting piecewise-defined global connection $\hat A$ is in $\LaX4{}$, but because the normal component of $B$ does not match that of $A$ on $\pB$, the new connection $\hat A$ is not in $\LaX21$. However, the fact that the tangential components match still gives us more regularity than $\LaX4{}$, In fact, $\hat A$ still has enough regularity to define curvature $F_{\hat A}\in\LFX2{}$. This leads us to define a space inbetween $\LaX4{}$ and $\LaX21$ which we call $\LaX2d$, where the $d$ subscript refers to the exterior derivative.
\begin{definition}
  Let $X$ be a compact smooth $4$-manifold, and let $\LeX2d$ be the completion of smooth forms $\alpha$ under the norm
  \begin{equation*}
    \normlx\alpha4{}+\normlx{d\alpha}2{}.
  \end{equation*}
\end{definition}

Such a definition has previously appeared context of gauge theory \cite{s02}, and a similar definition has appeared in the context of numerical analysis, namely, the $H(\operatorname{curl})$ spaces of Girault and Raviart \cite[Section 2.2]{gr86}. We now show that a form defined piecewise on $X$ with matching tangential components is in $\Le2dX$. For an analogous result in the numerical analysis context, see \cite[Proposition III.1.2]{bf91}. For our purposes, in the following proposition, $Y$ will be a $4$-ball $B^4$ contained in $X$, $Z$ will be the closure of its complement, and $S$ will be the shared boundary $\pB$. Intuitively, although we have a discontinuity in the normal component as we cross $S$, when taking $d$ we never take the normal derivative of the normal component, so we never see this discontinuity.

\begin{proposition}\label{l2dpatching}
  Let $X$ be a compact smooth $4$-manifold. We split $X$ into two pieces $Y$ and $Z$ by a smooth internal boundary $S$. Let $i\colon S\hookrightarrow X$ be the inclusion. Let $\beta\in\Le21Y$ and $\gamma\in\Le21Z$ be such that $i^*\beta=i^*\gamma$. Let $\alpha$ be the $L^4$ form on $X$ defined piecewise by $\beta$ and $\gamma$. Then $\alpha\in\Le2dX$.
  \begin{proof}
    We must show that $d\alpha\in\Le2{}X$. Let $\phi$ be a smooth differential form on $X$ compactly supported away from $\pX$ of the same degree as $d\alpha$. By the definition of the weak derivative, we have that $\pairlx{d\alpha}\phi=\pairlx\alpha{d^*\phi}$, and our goal is to show that this pairing gives a bounded map in $\phi$ with respect to the $\LX2{}$ norm. We compute that
    \begin{multline*}
      \pairlx{d\alpha}\phi=\pairlx\alpha{d^*\phi}=\pair\beta{d^*\phi}_{L^2(Y)}+\pair\gamma{d^*\phi}_{L^2(Z)}\\
      =\pair{d\beta}\phi_{L^2(Y)}+\pair{d\gamma}\phi_{L^2(Z)}-\int_{\partial Y}\beta\wedge*\phi-\int_{\partial Z}\gamma\wedge*\phi.
    \end{multline*}
    I claim that the two boundary integral terms cancel. Indeed, for the parts of $\partial Y$ and $\partial Z$ that are on $\partial X$, we have $\phi=0$ by assumption. The remaining parts of $\partial Y$ and $\partial Z$ are on the shared boundary $S$, where we have $i^*\beta=i^*\gamma$. Since $\partial Y$ and $\partial Z$ have opposite orientations on $S$, the two terms cancel. Hence,
    \begin{equation*}
      \pairlx{d\alpha}\phi=\pair{d\beta}\phi_{L^2(Y)}+\pair{d\gamma}\phi_{L^2(Z)}\le\norm\beta_{L^2_1(Y)}\normlx\phi2{}+\norm\gamma_{L^2_1(Z)}\normlx\phi2{}.
    \end{equation*}
    Since smooth forms compactly supported away from $\partial X$ are dense in $\Le2{}X$, we see that $\phi\mapsto\pairlx{d\alpha}\phi$ is a bounded functional on $\Le2{}X$, so $d\alpha\in\Le2{}X$, as desired.
  \end{proof}
\end{proposition}

\begin{proposition}
  Let $X$ be a compact smooth manifold with a principal $G$-bundle $P$, and let $A$ be an $L^2_d$ connection on $P$. Then $F_A$ is in $\LL2{}{X;\ad P\times\forms2T^*X}$. Moreover, if $g$ is an $\LL41{X;\Ad P}$ gauge transformation, then $g(A)$ is once again a $L^2_d$ connection.
  \begin{proof}
    The question is local, so we work on a compact subset $K$ of a trivialization. Let $A=d+a$, so $F_A=da+\frac12[a\wedge a]$. Since $a\in\La2dK$, we know that $da\in\LF2{}K$, and $a\in\La4{}K$, so $\frac12[a\wedge a]\in\LF2{}K$. Thus $F_A\in\LF2{}K$, as desired. Likewise, for $g\in\LG41K$, let $g(A)=B=d+b$. Then
    \begin{equation*}
      b=gag^{-1}-(dg)g^{-1}.
    \end{equation*}
    We have $a\in\La4{}K$, $g\in\LG41K$, and, since $g$ is a gauge transformation, $g\in\LE\infty{}K$. We can compute that then $b\in\La4{}K$. It remains to show that $db\in\LF2{}K$. We compute
    \begin{equation*}
      db=F_B-\tfrac12[b\wedge b]=gF_Ag^{-1}-\tfrac12[b\wedge b].
    \end{equation*}
    Since $F_A\in\LF2{}K$ and $g\in\LE\infty{}K$, we know that $gF_Ag^{-1}$ is in $\LF2{}K$. Likewise, we showed that $b\in\La4{}K$, so $\tfrac12[b\wedge b]\in\LF2{}K$, and so $db\in\LF2{}K$, as desired.
  \end{proof}
\end{proposition}

\section{Yang--Mills Replacement}\label{ymreplacementsection}
In this section, we build up to the proofs of Theorems \ref{globalymreplacementtheorem} and \ref{familyymreplacementtheorem} in Section \ref{mainproofssection}, using the gauge fixing results in Section \ref{gaugefixingchapter}. In Section \ref{dirichletproblemsection}, we solve the Yang--Mills equation on $B^4$ with prescribed small boundary data in $\LapB2{1/2}$ using the inverse function theorem. In Section \ref{localreplacementsection} we use gauge fixing to extend this result to a more general class of boundary values in Theorem \ref{ymreplacementlowenergy}. In addition, the inverse function theorem gives us local uniqueness of the solution, which we strengthen in Theorem \ref{ymuniqueness}. Along the way, in Proposition \ref{ymreplacementinequality}, we prove strict convexity in Coulomb gauge of the energy functional near small-energy Yang--Mills connections. In Section \ref{linearinterpolation} we show energy monotonicity of the linear path between an a connection and the Yang--Mills replacement that matches it on the boundary. Finally, in Section \ref{globalreplacementchapter}, we move the global question of Yang--Mills replacement on a ball in a larger manifold $X$, and we address the irregularity across the boundary of the ball by showing that every $L^2_d(X)$ connection is gauge equivalent to an $L^2_1(X)$ connection.

\subsection{The Dirichlet problem with small $L^2_{1/2}$ boundary data}\label{dirichletproblemsection}
In this section, we prove the the existence of Yang--Mills connections on the ball with prescribed small boundary data in $\LapB2{1/2}$. As discussed in more detail in the introduction, there is substantial work by Marini and Isobe \cite{m92,im97,im12a,im12b,im10} on related problems where the boundary data is smooth and prescribed up to gauge equivalence, or where the boundary data is required to be a sufficiently small multiple of a fixed smooth connection on $\pB$ rather than small in the $\LpB2{1/2}$ norm. The specific problem of Theorem \ref{dirichletproblem} has also been addressed in lecture notes of Rivi\`ere \cite{r14} using direct minimzation methods of Sedlacek \cite{s82}. Here we instead prove this result using the inverse function theorem, as a consequence of which we know that the Yang--Mills connection depends smoothly on the boundary data.
\begin{theorem}\label{dirichletproblem}
  Let $B^4$ be a smooth $4$-ball with arbitrary metric, let $i\colon\pB\to B^4$ be the inclusion, and let $P\to B^4$ be a principal $G$-bundle with trivializing connection $d$. There exist an $\epsilon>0$ and $\delta>0$ such that if $\Ap=d+\ap$ is an $\LapB2{1/2}$ connection with $\normlpb\ap2{1/2}<\epsilon$, then $\Ap$ extends to a \emph{unique} $\LaB21$ connection $A=d+a$ such that
  \begin{enumerate}
  \item $\normlb a21<\delta$,
  \item $A$ is Yang--Mills on $B^4$,
  \item $i^*A=\Ap$, and
  \item $A$ satisfies the Coulomb condition $d^*a=0$.
  \end{enumerate}
  Moreover, $A$ depends smoothly on $\Ap$.
\end{theorem}

We first prove the theorem replacing the Yang--Mills condition with the weaker projected Yang--Mills equation $\pi_{d^*}d_A^*F_A^{}=0$ (see Sections \ref{connectionsprelim} and \ref{subsectionl2dir}), and then prove that for small $a$ the weaker equation $\pi_{d^*}d_A^*F_A^{}=0$ actually implies that $A$ is Yang--Mills.

\begin{proposition}\label{dirichletproblemweak}
  There exist an $\epsilon>0$ and $\delta>0$ such that if $\Ap=d+\ap$ is an $\LapB2{1/2}$ connection with $\normlpb\ap2{1/2}<\epsilon$, then $\Ap$ extends to a \emph{unique} $\LaB21$ connection $A=d+a$ such that
  \begin{enumerate}
  \item $\normlb a21<\delta$,
  \item $A$ satisfies the projected Yang--Mills equation $\pi_{d^*}d_A^*F_A^{}=0$ on $B^4$,
  \item $i^*A=\Ap$, and
  \item $A$ satisfies the Coulomb condition $d^*a=0$.
  \end{enumerate}
  Moreover, $A$ depends smoothly on $\Ap$.
  \begin{proof}
    We consider the projected Yang--Mills operator
    \begin{equation*}
      pYM\colon\LaB21\cap\ker d^*
      \to\LaB{2,\Dir}{-1}\cap\range(d^*)\times\LapB2{1/2}
    \end{equation*}
    defined by
    \begin{equation*}
        pYM(a)=(\pi_{d^*}d_A^*F_A^{},i^*a).
    \end{equation*}
    We have that $A\mapsto\pi_{d^*}d_A^*F_A^{}$ is smooth by Corollary \ref{pymsmooth} and $i^*\colon\LaB21\to\LapB2{1/2}$ is a bounded linear operator and hence smooth.

    Proving the proposition amounts to showing that $pYM$ is an isomorphism on a neighborhood of $a=0$. Our desired $\LaB21$ extension $a$ is then the inverse image of $(0,\ap)$. We do so using the inverse function theorem. The linearization of $pYM$ at $a=0$ is
    \begin{equation*}
      (d^*d,i^*)\colon\LaB21\cap\ker d^*
      \to\LaB{2,\Dir}{-1}\cap\range(d^*)\times\LapB2{1/2}
    \end{equation*}
    It remains to show that this operator is an isomorphism of Banach spaces.

    Let $a\in\LaB21\cap\ker d^*$, and assume that $d^*da=0$ and $i^*a=0$. Thus $a\in\LaB{2,\Dir}1$, and so Proposition \ref{dsbounded} tells us that
    \begin{equation*}
      0=\pair{d^*da}a_{\LX2{}}=\pair{da}{da}_{\LX2{}}.
    \end{equation*}
    Thus, $da=0$. In addition, $d^*a=0$ and $i^*a=0$, so $a\in\mc H^\Dir$. But $H^1(B^4,\pB)=0$, so $a=0$. Hence $(d^*d,i^*)$ is injective.

    We first prove surjectivity onto $\LaB{2,\Dir}{-1}\cap\range(d^*)\times0$. Let $y=d^*f$ for $f\in\LFB2{}$. Then
    \begin{equation*}
      y=d^*f=d^*(dd^*G^\Dir f+d^*dG^\Dir f+\pi_{\mc H^\Dir}f)=d^*d(d^*G^\Dir f).
    \end{equation*}
    Clearly, $d^*G^\Dir f\in\ker d^*$. Moreover, $i^*(d^*G^\Dir f)=0$ by Proposition \ref{hodgedecomposition}. Hence, $d^*G^\Dir f$ is our desired preimage of $(y,0)$ under the map $(d^*d,i^*)$.

    Now, given $\ap\in\LapB2{1/2}$, the inverse trace map \cite[Theorem 7.53]{a75} gives us an $a_1\in\LaB21$ such that $i^*a_1=\ap$. Then $d^*da_1$ is in the space $\LaB{2,\Dir}{-1}\cap\range(d^*)$, so the previous paragraph gives us an $a_2\in\LaB21\cap\ker d^*$ such that $d^*da_2=d^*da_1$ and $i^*a_2=0$. Hence $(d^*d,i^*)(a_1-a_2)=(0,\ap)$, giving us surjectivity onto the other factor $0\times\LapB2{1/2}$ also. 

    We conclude that
    \begin{equation*}
      (d^*d,i^*)\colon\LaB21\cap\ker d^*
      \to\LaB{2,\Dir}{-1}\cap\range(d^*)\times\LapB2{1/2}
    \end{equation*}
    is an isomorphism of Banach spaces. Hence, by the inverse function theorem, the projected Yang--Mills operator
    \begin{equation*}
      pYM\colon\LaB21\cap\ker d^*
      \to\LaB{2,\Dir}{-1}\cap\range(d^*)\times\LapB2{1/2}
    \end{equation*}
    is a diffeomorphism between a neighborhood of $a=0$ in $\LaB21\cap\ker d^*$ and a neighborhood of $(y,\ap)=(0,0)$ in
    \begin{equation*}
      \LaB{2,\Dir}{-1}\cap\range(d^*)\times\LapB2{1/2}.
    \end{equation*}
    In particular, for $\ap$ sufficiently small in $\LapB2{1/2}$, we can uniquely solve $pYM(a)=(0,\ap)$ for $a\in\LaB21\cap\ker d^*$, giving us our desired small $a$ depending smoothly on $\ap$ satisfying $\pi_{d^*}d_A^*F_A^{}=0$, $i^*a=\ap$, and $d^*a=0$.
  \end{proof}
\end{proposition}

To complete the proof of Theorem \ref{dirichletproblem}, it remains to prove the following.
\begin{proposition}\label{pymimpliesym}
  There exists an $\epsilon>0$ such that for any $\LaB21$ connection $A=d+a$, if
  \begin{enumerate}
  \item $\normlb a21<\epsilon$,
  \item $A$ satisfies the projected Yang--Mills condition $\pi_{d^*}d_A^*F_A^{}=0$, and
  \item $A$ satisfies the Coulomb condition $d^*a=0$,
  \end{enumerate}
  then $A$ is Yang--Mills on $B^4$.
\end{proposition}

In higher regularity spaces, this proposition can be proved using bounds on $d_A^*F_A^{}$, but for $\LaB21$ connections, we must proceed directly by showing that $A$ locally minimizes energy. We prove a convexity inequality similar to one used by Colding and Minicozzi for harmonic maps \cite[Theorem 3.1]{cm08}.

\begin{proposition}\label{ymreplacementinequality}
  Let $B^4$ be a smooth $4$-ball with arbitrary metric, let $i\colon\pB\to B^4$ be the inclusion, and let $P\to B^4$ be a principal $G$-bundle with trivializing connection $d$. There exist constants $\epsilon_4, \epsilon_F$ and $C$ with the following significance. Let $A=d+a$ and $B=d+b$ be $\LaB21$ connections such that
  \begin{enumerate}
  \item $\normlb a4{}<\epsilon_4$ and $\normlb b4{}<\epsilon_4$,
  \item $\normlb{F_A}2{}<\epsilon_F$,
  \item $A$ satisfies the projected Yang--Mills equation $\pi_{d^*}d_A^*F_A^{}=0$,
  \item $A$ and $B$ match on the boundary, that is, $i^*A=i^*B$, and
  \item we have a Coulomb condition $d^*a=d^*b$.
  \end{enumerate}
  Then
  \begin{equation}\label{convexityinequality}
    \normlb{B-A}21^2\le C\left(\normlb{F_B}2{}^2-\normlb{F_A}2{}^2\right).
  \end{equation}
  In particular,
  \begin{equation*}
    \normlb{F_A}2{}\le\normlb{F_B}2{}.
  \end{equation*}
  \begin{proof}
    The projected Yang--Mills equation $\pi_{d^*}d_A^*F_A^{}=0$ can be restated as the condition that $\pairlb{F_A}{d_Ac}=0$ for all $c\in\LaB21$ with $i^*c=0$ and $d^*c=0$. In particular, let $c=B-A=b-a$, so $i^*c=i^*B-i^*A=0$ and $d^*c=d^*b-d^*a=0$. Hence, taking the square of the $\LB2{}$ norm of the curvature equation $F_B=F_A+d_Ac+\tfrac12[c\wedge c]$, we obtain
    \begin{equation*}
      \begin{split}
        \norm{F_B}_{\LB2{}}^2&=\norm{F_A}_{\LB2{}}^2+\norm{d_A c}_{\LB2{}}^2+\norm{\tfrac12[c\wedge c]}_{\LB2{}}^2\\
        &\qquad{}+2\pair{F_A}{\tfrac12[c\wedge c]}_{\LB2{}}+2\pair{d_A c}{\tfrac12[c\wedge c]}_{\LB2{}}.
      \end{split}
    \end{equation*}
     We then have the inequality
    \begin{equation*}
      \begin{split}
        \norm{d_A c}_{\LB2{}}^2&\le\norm{F_B}_{\LB2{}}^2-\norm{F_A}_{\LB2{}}^2-\norm{\tfrac12[c\wedge c]}_{\LB2{}}^2\\
        &\qquad+2\norm{F_A}_{\LB2{}}\norm{\tfrac12[c\wedge c]}_{\LB2{}}+2\norm{d_A c}_{\LB2{}}\norm{\tfrac12[c\wedge c]}_{\LB2{}}\\
        &\le\norm{F_B}_{\LB2{}}^2-\norm{F_A}_{\LB2{}}^2-\norm{\tfrac12[c\wedge c]}_{\LB2{}}^2\\
        &\qquad+2\norm{F_A}_{\LB2{}}\norm{\tfrac12[c\wedge c]}_{\LB2{}}+\tfrac12\norm{d_A c}_{\LB2{}}^2+2\norm{\tfrac12[c\wedge c]}_{\LB2{}}^2.
      \end{split}
    \end{equation*}
    Rearranging,
    \begin{equation}\label{dbaal21}
      \begin{split}
        \norm{d_A c}_{\LB2{}}^2&\le2\left(\norm{F_B}_{\LB2{}}^2-\norm{F_A}_{\LB2{}}^2\right)\\
        &\qquad{}+\left(2\norm{F_A}_{\LB2{}}+\tfrac12\norm{[c\wedge c]}_{\LB2{}}\right)\norm{[c\wedge c]}_{\LB2{}}\\
        &\le2\left(\norm{F_B}_{\LB2{}}^2-\norm{F_A}_{\LB2{}}^2\right)\\
        &\qquad{}+\left(2\norm{F_A}_{\LB2{}}+\tfrac12 C_\Lie\norm c_{\LB4{}}^2\right)C_\Lie C_S^2\norm c_{\LB21}^2,
      \end{split}
    \end{equation}
    where $C_\Lie$ is the operator norm of the Lie bracket $\mf g\times\mf g\to\mf g$, and $C_S$ is the operator norm of the Sobolev embedding $\LB21\hookrightarrow\LB4{}$.

    The next step is to bound $\norm c_{\LB21}$ in terms of $\norm{d_A c}_{\LB2{}}$. Since $H^1(B^4,\pB)=0$, by Corollary \ref{Di}, $d+d^*\colon\LeB{2,\Dir}1\to\LeB2{}$ is a Fredholm operator with no kernel on one-forms. Thus, we have the estimate $\norm c_{\LB21}\le C_G\norm{(d+d^*)c}_{\LB2{}}$ for some constant $C_G$ independent of $c\in\LaB{2,\Dir}1$. Recalling that $d^*c=0$, we can compute
    \begin{equation*}
      \begin{split}
        \norm c_{\LB21}&\le C_G\norm{dc}_{\LB2{}}=C_G\norm{d_A c-[a\wedge c]}_{\LB2{}}\\
        &\le C_G\norm{d_A c}_{\LB2{}}+C_G C_\Lie C_S\norm{a}_{\LB4{}}\norm c_{\LB21}
      \end{split}
    \end{equation*}

    Requiring $\epsilon_4\le\frac12(C_G C_\Lie C_S)^{-1}$, we obtain $\norm c_{\LB21}\le C_G\norm{d_A c}_{\LB2{}}+\frac12\norm c_{\LB21}$. Rearranging, we obtain our desired bound
    \begin{equation*}
      \norm c_{\LB21}\le2C_G\norm{d_A c}_{\LB2{}}.
    \end{equation*}

    Combining the above inequality with \eqref{dbaal21}, we have
    \begin{multline}
      \norm c_{\LB21}^2\le 8C_G^2\left(\norm{F_A}_{\LB2{}}^2-\norm{F_B}_{\LB2{}}^2\right)\\
      {}+\left(8C_G^2C_\Lie C_S^2\norm{F_A}_{\LB2{}}+2C_G^2C_\Lie^2C_S^2\norm c_{\LB4{}}^2\right)\norm c_{\LB21}^2.
    \end{multline}
    Requiring, for example, that $\epsilon_F\le\frac14(8C_G^2C_\Lie C_S^2)^{-1}$ and $(2\epsilon_4)^2\le\frac14(2C_G^2C_\Lie^2C_S^2)^{-1}$, and noting that $\normlb c4{}<2\epsilon_4$, the above inequality becomes 
    \begin{equation*}
      \norm c_{\LB21}^2\le 8C_G^2\left(\norm{F_A}_{\LB2{}}^2-\norm{F_B}_{\LB2{}}^2\right)+\tfrac12\norm c_{\LB21}^2.
    \end{equation*}
    Rearranging, we obtain
    \begin{equation*}
      \norm c_{\LB21}^2\le16C_G^2\left(\norm{F_A}_{\LB2{}}^2-\norm{F_B}_{\LB2{}}^2\right),
    \end{equation*}
    so our desired inequality is true with $C=16C_G^2$.
  \end{proof}
\end{proposition}

Under the assumptions of Propositon \ref{pymimpliesym} for an appropriate $\epsilon$, Proposition \ref{ymreplacementinequality} tells us that $A$ has smaller energy than any nearby connection $B$ that matches $A$ on the boundary and satisfies the Coulomb condition $d^*b=0$. It remains to remove this last condition, and to allow $B$ to match $A$ on the boundary only up to gauge. To do so, we use gauge fixing results that we will prove in Section \ref{gaugefixingchapter}.

\begin{proof}[Proof of Proposition \ref{pymimpliesym}]
    Let $B$ be any $\LaB21$ connection satisfying $\normlb b21<\epsilon$ and the boundary condition $i^*B=i^*A$. We will prove that $\normlb{F_A}2{}\le\normlb{F_B}2{}$.

    Let $\epsilon_U$ and $C_U$ be the constants from Proposition \ref{gaugefixingconstantboundary}. Requiring $\epsilon\le\epsilon_U$, we can apply Proposition \ref{gaugefixingconstantboundary} to $B$ to give us a gauge equivalent connection $\tB$ satisfying $d^*\tb=0$, $i^*\tB=i^*B=i^*A$, and
    \begin{equation*}
      \normlb\tb21\le C_U\left(\normlb{F_B}2{}+\normlpb{i^*b}2{1/2}\right).
    \end{equation*}

    Let $\epsilon_4$ and $\epsilon_F$ be the corresponding constants from Proposition \ref{ymreplacementinequality}. We have the continuity of the trace map $\LB21\to\LpB2{1/2}$ and Sobolev maps $\LpB2{1/2}\to\LpB3{}$ and $\LB21\to\LB4{}$. Using these along with the continuiuty of $F_A$, we can choose $\epsilon$ small enough so that $\normlb a21<\epsilon$ implies $\normlb a4{}<\epsilon_4$, $\normlb{F_A}2{}<\epsilon_F$. Likewise, the inequality $\normlb\tb21\le C_U\left(\normlb{F_B}2{}+\normlpb{i^*b}2{1/2}\right)$ and the continuity of $b\mapsto F_B$ and $b\mapsto i^*b$ let us choose $\epsilon$ small enough so that $\normlb b21<\epsilon$ implies $\normlb\tb4{}<\epsilon_4$.

    Since $i^*A=i^*\tB$, we can apply Proposition \ref{ymreplacementinequality} to $A$ and $\tB$, since $d^*a=0$ by assumption and $d^*\tb=0$ by Theorem \ref{gaugefixingconstantboundary}. We conclude that
    \begin{equation*}
      \normlb{F_A}2{}\le\normlb{F_{\tB}}2{}=\normlb{F_B}2{},
    \end{equation*}
    as desired. Hence, $A$ locally minimzes energy among connections whose restrictions to $\pB$ is $i^*A$.

    We now assume only that $i^*B$ is gauge equivalent to $i^*A$, along with $\normlb b21<\epsilon$, and show that $\normlb{F_A}2{}\le\normlb{F_B}2{}$.  They key fact here is that, unlike the projected Yang--Mills condition, the condition we proved, namely that $A$ locally minimizes energy among connections whose restriction to the boundary is equal to $i^*A$, is a gauge-invariant condition.

    We require $\epsilon$ be small enough so that $\normlb a21,\normlb b21<\epsilon$ implies that $\normlb{F_A}2{}$ and $\normlb{F_B}2{}$ are small enough so that we can apply Theorem \ref{dirichletcoulombfixing}, giving us connections $\tA$ and $\tB$ gauge equivalent to $A$ and $B$, respectively, such that $d^*\ta=d^*\tb=0$ and $d^*_\pB i^*\ta=d^*_\pB i^*\tb=0$, along with bounds on $\normlb\ta21$ and $\normlb\tb21$. Choosing $\epsilon$ small enough, these bounds imply bounds on $\normlpb{i^*\ta}3{}$ and $\normlpb{i^*\tb}3{}$ that are sufficient for applying Proposition \ref{uniqueboundaryfixing} to find that the gauge transformation $g$ sending $i^*\tA$ to $i^*\tB$ is constant. We can apply this constant gauge transformation to $\tA$ on all of $B^4$ without affecting the Dirichlet Coulomb conditions $d^*\ta=0$ and $d^*_\pB i^*\ta=0$, so we can assume without loss of generality that $g=1$ and $i^*\tA=i^*\tB$. Since $A$ locally minimzes energy among connections whose restrictions match it on the boundary, so does $\tA$, and so $\normlb{F_A}2{}=\normlb{F_{\tA}}2{}\le\normlb{F_{\tB}}2{}=\normlb{F_B}2{}$. Thus, $A$ is Yang--Mills.
  \end{proof}

  \subsection{Yang--Mills replacement for small-energy connections on a ball}\label{localreplacementsection}

  Using gauge fixing, we can strengthen Theorem \ref{dirichletproblem} to solve the Yang--Mills equation for a larger class of boundary values, namely restrictions to the boundary of small-energy connections on the ball. This result is exactly what we need for replacing a global connection on a small ball with a Yang--Mills connection.

  \begin{theorem}\label{ymreplacementlowenergy}
    There exists an $\epsilon>0$ with the following significance. Let $B$ be an $\LB21$ connection with $\normlb{F_B}2{}<\epsilon$. Then we can construct a Yang--Mills connection $A$ that depends continuously on $B$ such that $i^*A=i^*B$.
    \begin{proof}
      The idea is to apply a gauge fixing result to $B$ in order to make the boundary value small enough to apply Theorem \ref{dirichletproblem}. We will use Theorem \ref{uhlenbeckgaugefixing}, Uhlenbeck's gauge fixing result with Neumann boundary conditions, though Theorem \ref{dirichletcoulombfixing} with Dirichlet boundary conditions would work equally well. Let $\epsilon_U$ and $C$ be the constants from Theorem \ref{uhlenbeckgaugefixing}, and let $\epsilon_\partial$ be the bound on the boundary value in Theorem \ref{dirichletproblem}. Require $\epsilon\le\epsilon_U$ and $C_TC\epsilon\le\epsilon_\partial$, where $C_T$ is the norm of the trace map $i^*\colon\LB21\to\LpB2{1/2}$.

      We can thus apply Theorem \ref{uhlenbeckgaugefixing} to obtain a $\LGB22$ gauge transformation $g$ sending $B$ to $\tB=d+\tb$ in Coulomb gauge with $\normlb\tb21\le C\normlb{F_B}2{}$, so
      \begin{equation*}
        \normlpb{i^*\tb}2{1/2}\le C_T\normlb\tb21\le C_TC\normlb{F_B}2{}<\epsilon_\partial.
      \end{equation*}
      Thus we can apply Theorem \ref{dirichletproblem} to $i^*\tB$ to obtain an $\LB21$ Yang--Mills connection $\tA$ such that $i^*\tA=i^*\tB$. We let $A=g^{-1}(\tA)$, so $i^*A=i^*B$, as desired.

      It remains to show that this construction is continuous, which is made more complex by the fact that $g$ is not uniquely determined by $B$ and hence might not depend contiuously on $B$. With appropriately chosen $\epsilon$, we know that $g$ is unique up to a constant gauge transformation $c$ by \cite[Theorem 2.3.7]{dk90} or an argument analogous to Corollary \ref{dirichletcoulombuniqueness}. We show that $A$ does not depend on the choice of $g$, so let $g'=cg$, let $\tB'=g'(B)=c(\tB)$, and let $\tA'$ be the Yang--Mills connection given by applying Theorem \ref{dirichletproblem} to $i^*\tB'$. We claim that $\tA'=c(\tA)$. Indeed, $i^*(c(\tA))=i^*\tB'$, and the other conditions of Theorem \ref{dirichletproblem} are preserved under constant gauge transformations and hence are true of $c(\tA)$. Since the connection given by Theorem \ref{dirichletproblem} is unique, we conclude that $\tA'=c(\tA)$, and so $A'=(g')^{-1}(\tA')=g^{-1}(\tA)=A$, as desired.

      We can use this uniqueness to show that this construction is continuous. Indeed, let $B_i\to B$ be a sequence of connections converging in $\LaB21$. Let $A_i$ and $A$ be the corresponding Yang--Mills connections constructed above. We will show that $A_i$ converges to $A$ by showing that any subsequence of the $A_i$ has a further subsequence that converges to $A$.

      Hence, we begin by passing to a subsequence of the $B_i$. By Proposition \ref{gaugefixingcontinuous}, after passing to a further subsequence, we can have the Coulomb gauge representatives $\tB_i$ converging to a Coulomb gauge representative $\tB$ of $B$. However, $\tB$ is only determined up to a constant gauge transformation and may depend on our initial choice of subsequence. In addition, Lemma \ref{stronggaugeconvergence} gives us that, after passing to a subsequence, the gauge transformations $g_i$ sending $B_i$ to $\tB_i$ converge in $\LGB22$ to the gauge transformation $g$ sending $B$ to $\tB$. Theorem \ref{dirichletproblem} gives us that the $\tA_i$ depend smoothly on the $i^*\tB_i$, which depend linearly on the $\tB_i$, so we know that the $\tA_i$ converge to $\tA$. Finally, because the $g_i$ converge strongly to $g$ in $\LGB22$, we know by Proposition \ref{borderlinecontinuity} that the $A_i=g_i^{-1}(\tA_i)$ converge strongly to $A$ in $\LaB21$. By our previous argument, even though $\tB$ might depend up to a constant gauge transformation on our initial choice of subsequence, the limit $A$ is unique and thus is independent of the initial choice of subsequence of the $B_i$. Thus, $A$ depends continuously on $B$, as desired.
    \end{proof}
  \end{theorem}

  We now use gauge fixing to prove a stronger uniqueness result for the Yang--Mills solution.
  \begin{theorem}\label{ymuniqueness}
    There exists an $\epsilon>0$ such that if $A$ and $B$ are $\LaB21$ Yang--Mills connections with
    \begin{enumerate}
    \item energy bounds $\normlb{F_A}2{},\normlb{F_B}2{}<\epsilon$, and
    \item gauge equivalent boundary values $i^*A$ and $i^*B$,
    \end{enumerate}
    then $A$ is gauge equivalent to $B$.
    \begin{proof}
      Choose $\epsilon$ small enough so that we can apply Theorem \ref{dirichletcoulombfixing}, giving us $\LaB21$ connections $\tA$ and $\tB$ gauge equivalent to $A$ and $B$, respectively, satisfying the Dirichlet Coulomb conditions $d^*\ta=d^*\tb=0$ and $d^*_\pB i^*\ta=d^*_\pB i^*\ta=0$, as well as the bounds $\normlb\ta21\le C\normlb{F_A}2{}$ and $\normlb\tb21\le C\normlb{F_B}2{}$. We have that $i^*\tA$ and $i^*\tB$ are gauge equivalent, and we would like to apply Proposition \ref{uniqueboundaryfixing}. We require that $\epsilon$ be small enough so that the bounds $\normlb\ta21,\normlb\tb21<C\epsilon$ and the Sobolev and trace maps $\LB21\xrightarrow{i^*}\LpB2{1/2}\hookrightarrow\LpB3{}$ suffice to give us the $\LapB3{}$ bounds required by Proposition \ref{uniqueboundaryfixing}, which then tells us that the gauge transformation sending $i^*\tA$ to $i^*\tB$ is a constant $c\in G$. Now viewing $c$ as a gauge transformation on all of $B^4$, apply $c$ to $\tA$, and note that $c(\tA)$ satisfies all of the properties above required of $\tA$. Hence, we may, without loss of generality replace $\tA$ by $c(\tA)$, or, in other words, assume that $c=1$, so $i^*\tA=i^*\tB$.

      Next, we require $\epsilon$ be small enough so that our bounds $\normlb{F_\tA}2{},\normlb{F_\tB}2{}<\epsilon$ and $\normlb\ta21,\normlb\tb21<C\epsilon$ suffice to give us the bounds needed for Proposition \ref{ymreplacementinequality} via the Sobolev embedding $\LB21\hookrightarrow\LB4{}$. The Yang--Mills condition is gauge-invariant, so $\tA$ is Yang--Mills, and, in particular, also satisfies the projected Yang--Mills equation. Hence, we can apply Proposition \ref{ymreplacementinequality} to $\tA$ and $\tB$ to conclude that $\normlb{F_{\tA}}2{}\le\normlb{F_{\tB}}2{}$. But by the same argument $\normlb{F_{\tB}}2{}\le\normlb{F_{\tA}}2{}$, so $\normlb{F_{\tA}}2{}=\normlb{F_{\tB}}2{}$. We then have $\tA=\tB$ by the inequality \eqref{convexityinequality} in Proposition \ref{ymreplacementinequality}, so $A$ and $B$ are gauge equivalent, as desired.
    \end{proof}
  \end{theorem}

  \subsection{Linear interpolation}\label{linearinterpolation}
  We know that a small Yang--Mills $A$ connection on $B^4$ locally minimizes energy among connections $B$ with the same restriction to the boundary. We go further by showing that, in Coulomb gauge, the linear interpolation from $B$ to $A$ is an energy-decreasing path. As before, for small connections in Coulomb gauge it suffices to assume only that $A$ satisfies the projected Yang--Mills equation instead of the full Yang--Mills equaiton.
  \begin{proposition}\label{interpolationprop}
    There exist constants $\epsilon_4$ and $\epsilon_F$ with the following significance. Let $A=d+a$ and $B=d+b$ be $\LaB21$ connections such that
    \begin{enumerate}
    \item $\normlb a4{},\normlb b4{}<\epsilon_4$,
    \item $\normlb{F_A}2{}<\epsilon_F$,
    \item $A$ satisfies the projected Yang--Mills equation $\pi_{d^*}d_A^*F_A^{}=0$.
    \item $i^*A=i^*B$, and
    \item $d^*a=d^*b$.
    \end{enumerate}
    Let $B_t=(1-t)A+tB$ be the linear interpolation between $B_0=A$ and $B_1=B$. Then $\norm{F_{B_s}}_{\LB2{}}\le\norm{F_{B_t}}_{\LB2{}}$ if $s<t$, with equality only if $A=B$.
    \begin{proof}
      Since $B_t$ satisfies all of the conditions above required of $B_1=B$, in order to prove the general statement it suffices to show that $\left.\dd t\right\rvert_{t=1}\norm{F_{B_t}}_{\LB2{}}^2\ge0$.

      Let $c=B-A$. Using the equations
      \begin{align*}
        F_B&=F_A+d_A c+\tfrac12[c\wedge c],\text{ and}\\
        F_A&=F_B+d_B(-c)+\tfrac12[(-c)\wedge(-c)]=F_B-d_B c+\tfrac12[c\wedge c],
      \end{align*}
      along with the projected Yang--Mills condition $\pairlb{F_A}{d_A c}=0$, we compute
      \begin{equation*}
        \begin{split}
          \tfrac12\left.\tdd t\right\rvert_{t=1}\norm{F_{B_t}}^2_{\LB2{}}&=\pairlb{F_B}{d_B\left.\tdd t\right\rvert_{t=1}B_t}\\
          &=\pairlb{F_B}{d_B c}\\
          &=\pairlb{F_B}{F_B-F_A+\tfrac12[c\wedge c]}\\
          &=\normlb{F_B}2{}^2+\pairlb{F_B}{\tfrac12[c\wedge c]}\\
          &\qquad{}-\pairlb{F_A+d_A c+\tfrac12[c\wedge c]}{F_A}\\
          &=\normlb{F_B}2{}^2-\normlb{F_A}2{}^2+\pairlb{F_B-F_A}{\tfrac12[c\wedge c]}.\\
        \end{split}
      \end{equation*}
      We bound the last term:
      \begin{equation*}
        \begin{split}
          &\abs{\pairlb{F_B-F_A}{\tfrac12[c\wedge c]}}\\
          ={}&\abs{\pairlb{dc+\tfrac12[b\wedge b]-\tfrac12[a\wedge a]}{\tfrac12[c\wedge c]}}\\
          \le{}&\left(\normlb c21+\tfrac12C_\Lie\normlb b4{}^2+\tfrac12 C_\Lie\normlb a4{}^2\right)\tfrac12C_\Lie\norm c_{\LB4{}}^2\\
          \le{}&\left(\tfrac14 C_\Lie^2C_S^2(\normlb a4{}^2+\normlb b4{}^2)+\tfrac12 C_\Lie C_S\normlb c4{}\right)\normlb c21^2
        \end{split}
      \end{equation*}
      Since $\normlb c4{}\le\normlb a4{}+\normlb b4{}$, we can choose $\epsilon_4$ small enough so that $\normlb a4{},\normlb b4{}<\epsilon_4$ guarantees that 
      \begin{equation*}
        \abs{\pairlb{F_B-F_A}{\tfrac12[c\wedge c]}}\le\tfrac12C^{-1}\normlb c21^2,
      \end{equation*}
      where $C$ is the constant in Proposition \ref{ymreplacementinequality}. Choosing $\epsilon_4$ and $\epsilon_F$ small enough so that we can apply Proposition \ref{ymreplacementinequality}, we have $\norm c_{\LB21}^2\le C\left(\normlb{F_B}2{}^2-\normlb{F_A}2{}^2\right)$, so
      \begin{multline*}
        \tfrac12\left.\tdd t\right\vert_{t=1}\normlb{F_{B_t}}2{}^2\ge\normlb{F_B}2{}^2-\normlb{F_A}2{}^2-\tfrac12 C^{-1}\normlb c21^2\\
        \ge\tfrac12\left(\normlb{F_B}2{}^2-\normlb{F_A}2{}^2\right)\ge0,
      \end{multline*}
      with equality only if $\normlb{F_B}2{}^2-\normlb{F_A}2{}^2=0$, in which case $B=A$ by Proposition \ref{ymreplacementinequality}.
    \end{proof}
\end{proposition}

Again, we can use gauge fixing to prove this energy monotonicity result for a wider class of connections.

\begin{theorem}\label{interpolationthm}
  There exists a constant $\epsilon$ with the following significance. Let $B$ be an $\LB21$ connection with $\normlb{F_B}2{}<\epsilon$, and let $A$ be the Yang--Mills connection with $i^*A=i^*B$ constructed in Theorem \ref{ymreplacementlowenergy}. Let $B_t=(1-t)A+tB$ be the linear interpolation between $B_0=A$ and $B_1=B$. Then $\normlb{F_{B_s}}2{}\le\normlb{F_{B_t}}2{}$ if $s<t$, with equality only if $A=B$.
  \begin{proof}
    From the construction in Theorem \ref{ymreplacementlowenergy}, there exists a gauge transformation $g$ sending $A$ and $B$ to $\tA=d+\ta$ and $\tB=d+\tb$, respectively, such that $d^*\tb=0$, and we know that $d^*\ta=0$ because we obtained it from Theorem \ref{dirichletproblem}. We would like to apply Proposition \ref{interpolationprop}. The construction in Theorem \ref{ymreplacementlowenergy} gives us a bound $\normlb\tb21\le C\normlb{F_B}2{}<C\epsilon$. Using the Sobolev inequalities we can choose a small enough $\epsilon$ to obtain the required bounds on $\normlb\tb4{}$. Since $\tA$ depends continuously on $\tB$, our choice of $\epsilon$ gives us bounds on $\normlb\ta21$, which lets us obtain the required bounds on $\normlb{F_{\tA}}2{}$ and $\normlb\ta4{}$. Finally, $A$ is Yang--Mills, so $\tA$ is also Yang--Mills and in particular satisfies the projected Yang--Mills equation. Hence we can apply Proposition \ref{interpolationprop} to $\tA$ and $\tB$, giving us that the linear interpolation between $\tA$ and $\tB$ has monotone energy. But both energy and affine combinations are preserved after applying a gauge transformation, so the linear interpolation between $A$ and $B$ also has monotone energy, as desired.
  \end{proof}
\end{theorem}

\subsection{Yang--Mills replacement for global connections}\label{globalreplacementchapter}

In this section, we consider an arbitrary compact $4$-manifold $X$ and use the solution to the Dirichlet problem for Yang--Mills connections on the ball in order to construct an energy-decreasing map on global $\LX21$ connections modulo $\LX22$ gauge transformations. Namely, given a connection $B$ on $X$ and a ball $B^4\subset X$ on which $B$ has small energy, we will replace $B$ on $B^4$ with a Yang--Mills connection $A$ that has the same restriction to the boundary $\pB$, thereby constructing a piecewise connection $\hat A$ that is Yang--Mills on the ball and equal to $B$ outside the ball. However, only the tangential components of $A$ and $B$ match on the boundary $\pB$, and the normal components may disagree. As a result, this new connection $\hat A$ is no longer in $\LX21$, but it is still in the space $\LX2d$ defined in Section \ref{l2dsection}. However, we will show that this piecewise-defined connection is nonetheless gauge equivalent to an $\LX21$ connection.

More generally, we prove that any $\LX2d$ connection is gauge equivalent via a $\LX41$ gauge transformation to a $\LX21$ connection, and so we show that the space of $\LX2d$ connections up to $\LX41$ gauge transformations is actually homeomorphic to the space of $\LX21$ connections up to $\LX22$ gauge transformations. This space is Hausdroff by Lemma \ref{stronggaugeconvergence}. Theorem \ref{l2dgaugefixing} tells us that, locally, every $\LB2d$ connection is gauge equivalent via an $\LB41$ gauge transformation to an $\LB21$ connection. However, patching these local gauge transformations to a global transformation is a delicate matter because $\LB41$ gauge transformations need not be continuous, so na\"{\i}ve cutoff methods fail. %borderline

A crucial lemma for dealing with this issue is due to Taubes, which tells us that a gauge transformation between two $L^2_1$ connections in Coulomb gauge is not only in $L^2_2$ as we expect but also in $C^0_\loc$.

\begin{lemma}[{\cite[Lemma A.1]{t84}}]\label{taubesregularity}
  Let $U$ be an open ball in a Riemannian $4$-manifold. Let $\mf F$ be the set of triples $(g,a,b)\in\LE22U\times\La21U\times\La21U$ such that $g$ is a gauge transformation sending $A=d+a$ to $B=d+b$ and $d^*a=d^*b=0$. Then the projection to the first factor $\mf F\to\LE22U$ sending $(g,a,b)$ to $g$ factors continuously through $C^0_\loc(U;M_N)$.
\end{lemma}

The next lemma we need is due to Uhlenbeck. It tells us that if two bundles over a compact manifold $X$ are described by transition functions $g_{\alpha,\beta}$ and $h_{\alpha,\beta}$ that are sufficiently close to each other in $C^0$, then the two bundles are isomorphic.

\begin{lemma}[{\cite[Proposition 3.2]{u82}}]\label{uhlenbeckbundleisomorphism}
  Let $X$ be a compact manifold with principal $G$-bundles $P$ and $Q$ and a finite cover by local trivializations $\{U_\alpha\}$ with continuous transition maps $\phi_{\alpha,\beta},\psi_{\alpha,\beta}\colon U_\alpha\cap U_\beta\to G$, respectively. There exists an $\epsilon$ depending on the cover but not on the transition maps such that if, for all $\alpha$ and $\beta$, $\phi_{\alpha,\beta}\psi_{\beta,\alpha}$ is a neighborhood of the identity on which $\exp^{-1}$ is defined and
  \begin{equation*}
    \norm{\exp^{-1}(\phi_{\alpha,\beta}\psi_{\beta,\alpha})}_{C^0(U_\alpha\cap U_\beta)}<\epsilon,
  \end{equation*}
  then there exists a cover $\{V_\alpha\}$ of $X$ with $V_\alpha\subset U_\alpha$ and an isomorphism between $P$ and $Q$ defined by $\rho_\alpha\colon V_\alpha\to G$ satisfying $\psi_{\alpha,\beta}\rho_\alpha=\rho_\beta \phi_{\alpha,\beta}$.
\end{lemma}

With these tools in hand, we can prove that any $\LX2d$ connection is gauge equivalent to an $\LX21$ connection. We follow Uhlenbeck's patching argument in \cite[Theorem 3.6]{u82}, using local gauge fixing to take an $\LX2d$-convergent sequence of connections on our bundle $P$ to an $\LX21$-convergent sequence of connections on a sequence of bundles $Q_i$. Uhlenbeck constructs the limit bundle $Q$ only when the gauge transformations are in a Sobolev space that is above the borderline, but we can use Taubes's lemma to construct the limit bundle $Q$ even at the borderline.
\begin{theorem}\label{l2disl21}
  Let $P$ be a smooth principal $G$-bundle over a compact manifold $X$. Let $A$ be a $\LX2d$ connection on $P$. Then there exists an $\LX41$ gauge transformation on $P$ sending $A$ to an $\LX21$ connection $B$. Moreover, the gauge equivalence class of $B$ modulo $\LX22$ gauge transformations depends continuously on $A$.
  \begin{proof}
    Let $A_i$ be a sequence of smooth connections converging to $A$ in $\LX2d$. Let $\{U_\alpha\}$ be a finite cover of $X$ by balls contained inside trivializations of $P$ and small enough so that, for all $\alpha$, $\norm{F_A}_{L^2(U_\alpha)}<\epsilon$ for the $\epsilon$ in Proposition \ref{gaugefixingcontinuous}. A priori, the $\epsilon$ required depends on the metric of the ball $B^4$. However, as discussed in Uhlenbeck \cite{u82}, we can note that the energy of connections is invariant under conformal changes of metric, and dilations in particular. Thus, we can rescale small exponential neighborhoods to balls of unit size with metric close to that of the standard unit ball, and choose an $\epsilon$ uniformly for all of the balls. Take a tail of the sequence to guarantee that $\norm{F_{A_i}}_{L^2(U_\alpha)}<\epsilon$ also. For a fixed $U_\alpha$, pass to a subsequence of $A_i$ given by Proposition \ref{gaugefixingcontinuous}, giving us gauge transformations $g_{i,\alpha}$ and $g_\alpha$ on $U_\alpha$ sending $A_i$ to $\tA_{i,\alpha}$ and $A$ to $\tA_\alpha$, respectively, such that the $\tA_{i,\alpha}$ are in Coulomb gauge on $U_\alpha$ and converge to $\tA_\alpha$ strongly in $\La21{U_\alpha}$, and the $g_{i,\alpha}$ converge to $g_\alpha$ in $\LG41{U_\alpha}$. Repeat this construction for the other $U_\alpha$, taking further subsequences. By the smoothness and uniqueness claim for the gauge transformation doing the gauge fixing in \cite[Theorem 2.3.7]{dk90}, we know that the $g_{i,\alpha}$ are smooth because the $A_i$ are smooth.

    Let $\phi_{\alpha,\beta}\colon U_\alpha\cap U_\beta\to G$ be the transition functions for $P$. Since the $g_{i,\alpha}$ are smooth, we define new transition functions by $\psi_{i,\alpha,\beta}=g_{i,\beta}\phi_{\alpha,\beta}g_{i,\alpha}^{-1}$ for a bundle $Q_i$ over $X$, so the $g_{i,\alpha}$ define a bundle isomorphism between $P$ and $Q_i$ sending $A_i$ to the connection defined by the $\tA_{i,\alpha}$ on $Q_i$.

    We would like to pass to the limit bundle $Q$ defined by $\psi_{\alpha,\beta}=g_\beta\phi_{\alpha,\beta}g_\alpha^{-1}$, where the $g_\alpha$ define an $L^4_1$ bundle isomorphism between $P$ and $Q$ sending to $A$ to the $L^2_1$ connection on $Q$ defined by the $\tA_\alpha$. However, the issue is that the $g_\alpha$ are not necessarily continuous, so we do not yet know that the $\psi_{\alpha,\beta}$ define a continuous bundle $Q$, nor do we know that $P$ and $Q$ are isomorphic as continuous bundles.

    However, we know that, on $U_\alpha\cap U_\beta$, $\psi_{\alpha,\beta}$ is a gauge transformation between $\tA_\alpha$ and $\tA_\beta$, that is,
    \begin{equation*}
      d\psi_{\alpha,\beta}=\psi_{\alpha,\beta}\ta_\alpha-\ta_\beta\psi_{\alpha,\beta},
    \end{equation*}
    and likewise $\psi_{i,\alpha,\beta}$ are gauge transformations between between $\tA_{i,\alpha}$ and the $\tA_{i,\beta}$. Let $K$ be a compact subset of $U_\alpha\cap U_\beta$. Since the $\tA_{i,\alpha}$ and the $\tA_{i,\beta}$ converge strongly in $\La21K$ to $\tA_\alpha$ and $\tA_\beta$, by Lemma \ref{stronggaugeconvergence}, after passing to a subsequence, the $\psi_{i,\alpha,\beta}$ converge in $\LG22K$. The limit must be $\psi_{\alpha,\beta}$ because the $\psi_{i,\alpha,\beta}$ converge to $\psi_{\alpha,\beta}$ in a weaker norm. Indeed, in the formula $\psi_{i,\alpha,\beta}=g_{i,\beta}\phi_{\alpha,\beta}g_{i,\alpha}^{-1}$, we know that the $g_{i,\beta}$ and $g_{i,\alpha}$ converge strongly in $\LG4{}K$ to $g_\beta$ and $g_\alpha$, so $\psi_{i,\alpha,\beta}$ converges strongly to $\psi_{\alpha,\beta}$ in $\LG2{}K$. Consequently, by Lemma \ref{taubesregularity}, the $\psi_{i,\alpha,\beta}$ converge to $\psi_{\alpha,\beta}$ in $C^0_\loc(V;G)$ for any open ball $V$ contained in $K$, so they converge in $C^0(K';G)$ for compact subsets $K'$ of $V$. Hence, by slightly shrinking the $U_\alpha$ to open balls $U_\alpha'\subset U_\alpha$ that still cover $X$, we have that the $\psi_{i,\alpha,\beta}$ converge to $\psi_{\alpha,\beta}$ in $C^0(U_\alpha'\cap U_\beta';G)$.

    Thus, we can choose a sufficiently large $i$ so that $\psi_{i,\alpha,\beta}$ is sufficiently close in $C^0(U_\alpha'\cap U_\beta';G)$ to $\psi_{\alpha,\beta}$ in order to satisfy the conditions of Lemma \ref{uhlenbeckbundleisomorphism}. Applying this lemma, we conclude that there is a continuous bundle isomorphism $\rho$ between $Q$ and $Q_i$, which in turn is smoothly isomorphic to $P$ via $g_{i,\alpha}^{-1}$. Moreover, the argument in \cite[Corollary 3.3]{u82} applies also to $L^2_2\cap C^0(U_\alpha'\cap U_\beta';G)$, so $\rho$ is in fact an $\LX22$ bundle map that depends continuously on $\psi_{i,\alpha,\beta}$ and $\psi_{\alpha,\beta}$ in $L^2_2\cap C^0(U_\alpha'\cap U_\beta';G)$. The $g_\alpha$ define an $\LX41$ bundle map between $P$ and $Q$ sending $A$ to an $\LX21$ connection defined by the $\tA_{\alpha}$, so $g_{i,\alpha}^{-1}\circ\rho\circ g_\alpha$ is an $\LX41$ bundle isomorphism $P\to Q\to Q_i\to P$ sending $A$ to an $\LX21$ connection $B$ on $P$, as desired.

    To prove that the gauge equivalence class of $B$ depends continuously on $A$, note that we can choose sufficiently large $i$ so that for $j\ge i$, we can construct an $L^2_2\cap C^0(X)$ bundle isomorphism $\rho_j$ between $Q_j$ and $Q_i$ just like we constructed $\rho$ above. Because the construction of $\rho_j$ depends continuously on the transition maps and the $\psi_{j,\alpha,\beta}$ converge to $\psi_{\alpha,\beta}$ in $L^2_2\cap C^0(U_\alpha'\cap U_\beta';G)$, we have that the $\rho_{j,\alpha}$ also converge in $L^2_2\cap C^0(U_\alpha'\cap U_\beta';G)$ to $\rho_\alpha$. The $\tA_{j,\alpha}$ converge to $\tA_\alpha$ in $\La21{U_\alpha'}$, so the $\rho_j(\tA_{j,\alpha})$ converge in $\La21{U_\alpha'}$ to $\rho(\tA_\alpha)$ as connections on the same bundle $Q_i$. Applying the smooth bundle map $g_{i,\alpha}^{-1}$, let $B_i$ be the connection on $P$ defined on trivializations by $(g_{i,\alpha}^{-1}\circ\rho_j)(\tA_{j,\alpha})=(g_{i,\alpha}^{-1}\circ\rho_j\circ g_{j,\alpha})(A_j)$. We see that the $B_i$ converge as $L^2_1$ connections on the bundle $P$ to $(g_{i,\alpha}^{-1}\circ\rho)(\tA_{\alpha})=(g_{i,\alpha}^{-1}\circ\rho\circ g_{\alpha})(A)=B$. That is, we have constructed $\LX21$ connections $B_i$ and $B$ gauge equivalent to $A_i$ and $A$, respectively, such that the $B_i$ converge in $\LX21$ to $B$.

    The first issue to complete the proof of continuity is that we passed to subsequences in the proof, and our choice of $B$ may depend on our choice of subsequence. However, the gauge equivalence class of $B$ modulo $\LX22$ gauge transformations does not depend on this choice. Indeed, if $A$ is gauge equivalent by $\LX41$ gauge transformations to $\LX21$ connections $B$ and $B'$, then $B$ and $B'$ are gauge equivalent via an $\LX41$ gauge transformation $g$. But by Lemma \ref{gaugetransformationregularity}, $g$ is an $\LX22$ gauge transformation, and so $[B]=[B']$. Hence, for any subsequence of the $A_i$, our argument above shows that after passing to a further subsequence, the $[A_i]=[B_i]$ converge to $[B]$ in the space of $\LX21$ connections modulo $\LX22$ gauge transformations. Thus the original sequence also converges to $[B]$.

    The second issue is that we assumed that the $A_i$ are smooth, but to show continuity we need a general sequence of $\LX2d$ connections on $P$. Now let $A_i$ be an $\LX2d$ sequence of connections converging in $\LX2d$ to $A$, and let $[B_i]$ and $[B]$ be the gauge equivalence class of $\LX21$ connections constructed above. To show that the $[B_i]$ converge to $[B]$, we consider sequences of smooth connections $A_{i,j}$ that converge to $A_i$, so from the above argument we know that the $[A_{i,j}]$ converge to $[B_i]$. We then use a diagonalization argument, constructing $j(i)$ such that $A_{i,j(i)}$ converges to $A$ and such that $[A_{i,j(i)}]$ is within $1/i$ of $[B_i]$. Because the $A_{i,j(i)}$ are smooth, the argument above gives us that the $[A_{i,j(i)}]$ converge to $[B]$, and so the $[B_i]$ must also converge to $[B]$.
  \end{proof}
\end{theorem}

\begin{corollary}\label{homeomorphicconfigurationspaces}
  Let $P\to X$ be a principal $G$-bundle over a compact manifold $X$. The space of $\LX2d$ connections modulo $\LX41$ gauge transformations is homeomorphic to the space of $\LX21$ connections modulo $\LX22$ gauge transformations.
\end{corollary}

Note that this space is Hausdorff by Lemma \ref{stronggaugeconvergence}.

\subsection{The main results}\label{mainproofssection}
We now have all the ingredients for our Yang--Mills replacement result.

\begin{theorem}\label{globalymreplacementtheorem}
  Let $P\to X$ be a smooth principal $G$-bundle over a compact $4$-manifold $X$ with compact gauge group $G$, and let $B^4\subset X$ be a smooth $4$-ball. Let $\mc C$ be the space of $\LX21$ connections on $P$ modulo $\LX22$ gauge transformations of $P$, and let $\mc C_{\epsilon,B^4}\subset\mc C$ be those gauge equivalence classes of connections $[B]$ with small energy on $B^4$, that is, $\normlb{F_B}2{}<\epsilon$. Then for small enough $\epsilon$ there is an energy-decreasing continuous map $r\colon\mc C_{\epsilon,B^4}\to\mc C_{\epsilon,B^4}$ sending $[B]$ to an equivalence class of connections $[A]$, where $A$ is Yang--Mills on $B^4$ and gauge equivalent to $B$ outside $B^4$.

  Moreover, there is a homotopy $h_t\colon\mc C_{\epsilon,B^4}\to\mc C_{\epsilon,B^4}$ between the above replacement map $h_0=r$ and the identity map $h_1=\Id$ such that $h_t([B])$ has monotone nondecreasing energy, and restricting to the complement of $B^4$ the homotopy is constant $h_t([B])=[B]$.

\begin{proof}
  Given $B$ with $[B]\in\mc C_{\epsilon,B^4}$, on $B^4$ we construct $A$ and $B_t$ as in Theorems \ref{ymreplacementlowenergy} and \ref{interpolationthm}. By Theorem \ref{ymreplacementlowenergy}, $A$ and $B_t$ depend continuously on $B$, and from the construction it is clear that if we choose a different representative $B'=g(B)$ of $[B]$, then the resulting $A'$ and $B_t'$ satisfy $A'=g(A)$ and so $B_t'=g(B_t)$, so $[B_t]$ is independent of our choice of $B\in[B]$. We define a global connection $\hat B_t$ piecewise by $B_t$ on the ball and $B$ outside the ball. By Proposition \ref{l2dpatching}, $\hat B_t$ is an $\LX2d$ connection that depends continuously on $B$. By Theorem \ref{l2disl21}, $\hat B_t$ is gauge equivalent to an $\LX21$ connection $\tilde B_t$, and $[\tilde B_t]$ depends continuously on $\hat B_t$, so we can set $h_t([B])=[\tilde B_t]$.
\end{proof}
\end{theorem}

We can also reframe the above result in terms of comact families of connections.

\begin{theorem}\label{familyymreplacementtheorem}
  Let $P\to X$ be a smooth principal $G$-bundle over compact $4$-manifold $X$ with compact gauge group $G$, and let $\mc C$ be the space of $\LX21$ connections modulo $\LX22$ gauge transformations. Let $\mc K$ be a compact family in $\mc C$. Then around any point $x\in X$ there exists a ball $x\in B^4\subset X$ and homotopy $h_t\colon\mc K\to\mc C$ such that $h_1$ is the identity, $h_0$ sends $\mc K$ to connections that are Yang--Mills on $B^4$, $h_t([B])$ has monotone nondecreasing energy, and restricting to the complement of $B^4$ the homotopy is constant $h_t([B])=[B]$.

\begin{proof}
  Since $\mc K$ is compact, we can choose a ball $B^4$ around $x$ small enough so that for all $[B]\in\mc K$, $\normlb{F_B}2{}<\epsilon$. Then we apply Theorem \ref{globalymreplacementtheorem} for this ball $B^4$.
\end{proof}
\end{theorem}

\section{Gauge fixing}\label{gaugefixingchapter}

In this section, we prove several gauge fixing results which are used to prove the main results in Section \ref{ymreplacementsection}. In Section \ref{gaugeconvergencesection}, we show that weak or strong convergence of connections implies the weak or strong convergence of the gauge transformations between them, implying as a corollary that the space of $L^2_1(X)$ connections modulo $L^2_2(X)$ gauge transformations is Hausdorff. In Section \ref{coulombgaugeconvergencesection}, we extend Uhlenbeck's gauge fixing result with Neumann boundary conditions \cite{u82} to $\LB2d$ connections, and we show that strong convergence of a sequence of connections implies strong convergence of the respective Coulomb gauge representatives. Finally, in Sections \ref{fixedboundarysection} and \ref{dirichletcoulombsection}, we prove gauge fixing results with Dirichlet boundary conditions on the gauge transformation and on the connection, respectively, improving on earlier results of Uhlenbeck \cite{u82,u82b} and Marini \cite{m92} by working at the lowest possible level regularity.

\subsection{Convergence of gauge transformations}\label{gaugeconvergencesection}
In this section, we consider a sequence of gauge transformations that takes one sequence of connections to another. Working with gauge transformations in borderline Sobolev spaces, we show that weak convergence of the two sequences of connections implies weak convergence of a subsequence of the gauge transformations, and likewise for strong convergence. The weak convergence argument is straightforward even for gauge transformations in borderline Sobolev spaces. However, the strong convergence argument is more delicate, but we are able to proceed using Lemma \ref{weaktostronglemma}. Throughout, let $P$ be a principal $G$-bundle over a compact $4$-manifold $X$, where $G$ is compact.
\begin{lemma}\label{gaugetransformationregularity}
  Let $A$ and $B$ be two $\Laa21X$ connections that are gauge equivalent via a gauge transformation $g$. Then $g\in\LGG22X$.
  \begin{proof}
    Because $X$ is compact, the question is local, so we work on a trivialization over a closed ball $B^4\subset X$. With respect to this trivialization, $A=d+a$ and $B=d+b$, where $a$ and $b$ are in $\LaB21$, and $g$ is in $\LGB22$. We have the equation
    \begin{equation*}
      dg=ga-bg.
    \end{equation*}
    Since $G$ is compact, we know that $g\in\LEB\infty{}$. Sobolev embedding tells us that $a,b\in\LaB4{}$. Thus, $dg=ga-bg\in\LEaB4{}$, so $g\in\LGB41$. Next,
    \begin{equation*}
      \nabla(dg)=(\nabla g)a-b(\nabla g)+g(\nabla a)-(\nabla b)g\in L^4\cdot L^4-L^4\cdot L^4+L^\infty\cdot L^2-L^2\cdot L^\infty\subseteq L^2.
    \end{equation*}
    Hence, $dg\in\LEaB21$, so $g\in\LGB22$.
  \end{proof}
\end{lemma}

\begin{lemma}\label{weakgaugeconvergence}
  Consider two sequences of $\Laa21X$ connections $A_i$ and $B_i$ converging weakly to $A$ and $B$ in $\Laa21X$, respectively, such that $A_i$ and $B_i$ are gauge equivalent via a gauge transformation $g_i$. Then a subsequence of the $g_i$ converges weakly in $\LGG22X$ to a gauge transformation $g$ sending $A$ to $B$.

  If the connections are only $\Laa4{}X$ and converge weakly in $\Laa4{}X$, then a subsequence of the gauge transformations converges weakly in $\LGG41X$.
  \begin{proof}
    As before, since $X$ is compact it suffices to prove the lemma locally, so we work on a trivialization over a closed ball $B^4\subset X$. With the same notation, we have equations
    \begin{equation}\label{sequencegaugeequivalence}
      dg_i=g_ia_i-b_i g_i.
    \end{equation}
    The gauge transformations are assumed to be unitary and hence uniformly bounded in $\LEB\infty{}$. Since $a_i$ and $b_i$ convergely weakly and are hence bounded in $\LaB4{}$, we know that the $dg_i$ are bounded in $\LaB4{}$, and hence the sequence $g_i$ is bounded in $\LEB41$. Passing to a subsequence, we can assume that $g_i$ has a weak limit $g$ in $\LEB41$. Thus, $g_i$ converges strongly to $g$ in $\LEB4{}$, so we know that $g$ is in $G$ a.e.\ because a subsequence of the $g_i$ converges pointwise a.e.\ to $g$. Thus $g$ is a gauge transformation.

    It remains to show that $g$ sends $A$ to $B$. We would like to take the limit of the equation $dg_i=g_ia_i-b_ig_i$, but the issue is that the product of sequences that converge weakly need not converge, even weakly. However, we have that the $g_i$ converge strongly to $g$ in $\LEB4{}$, and hence using the multiplication map $\LL4{}{B^4}\times\LL21{B^4}\to\LL2{}{B^4}$, we know that the $g_ia_i$ converge weakly to $ga$ in $\LEB2{}$, and, similarly, the $b_ig_i$ converge weakly to $bg$ in $\LEB2{}$. The weak convergence of $dg_i$ follows from the linearity of $d$. Therefore, we can take the weak limit of \eqref{sequencegaugeequivalence} in $\LEaB2{}$ to find that
    \begin{equation*}
      dg=ga-bg,
    \end{equation*}
    so $g$ indeed sends $A$ to $B$.

    If, moreover, the $a_i$ and $b_i$ are bounded in $\LaB21$, then because the $g_i$ are bounded in $\LEB41$ and $\LEB\infty{}$, we see that
    \begin{equation*}
      \nabla(dg_i)=(\nabla g_i)a_i-b_i(\nabla g_i)+g_i(\nabla a_i)-(\nabla b_i)g_i
    \end{equation*}
    is bounded in $\LnaB2{}$. Hence, the $g_i$ are bounded in $\LEB22$, so, after passing to a subsequence, they converge weakly in $\LEB22$, and the limit is $g$, because the weak $\LEB22$ limit must agree with the weak $\LEB41$ limit.
  \end{proof}
\end{lemma}

In both the preceding and the following lemma, the reason we need to take a subsequence is that the limit connection might have a nontrivial, but compact, stabilizer, so we might even have constant sequences $A_i=A$ and $B_i=A$ such that the sequence of gauge transformations fails to converge. However, in a situation where there is no stabilizer and we have a unique gauge transformation between $A$ and $B$, we can eliminate the need for taking a subsequence, using the fact that if every subsequence of the $g_i$ has a further subsequence that converges to $g$, then the original sequence $g_i$ converges to $g$ also.

\begin{lemma}\label{stronggaugeconvergence}
  Consider two sequences of $\Laa21X$ connections $A_i$ and $B_i$ converging strongly to $A$ and $B$ in $\Laa21X$, respectively, such that $A_i$ and $B_i$ are gauge equivalent via a gauge transformation $g_i$. Then a subsequence of the $g_i$ converges strongly in $\LGG22X$ to a gauge transformation $g$ sending $A$ to $B$.

  If the connections are only $\Laa4{}X$ and converge strongly in $\Laa4{}X$, then a subsequence of the gauge transformations converges strongly in $\LGG41X$.
  \begin{proof}
    Again, since $X$ is compact it suffices to prove the lemma on a closed ball $B^4\subset X$. By Lemma \ref{weakgaugeconvergence}, after passing to a subsequence, the $g_i$ converge weakly in $\LGB41$ to a gauge transformation $g$ sending to $A$ to $B$. The main difficulty in promoting weak convergence to strong convergence is that, even though the $g_i$ are bounded in $\LEB\infty{}$, we cannot get strong convergence in $\LEB\infty{}$ for gauge transformations in a borderline Sobolev space. As before, we deal with this issue using Lemma \ref{weaktostronglemma}. Again, we use the equations
    \begin{align*}
      dg_i&=g_ia_i-b_ig_i,\\
      dg&=ga-bg.
    \end{align*}
    We begin by showing that, after passing to a subsequence, the $g_i$ converge strongly to $g$ in $\LGB41$. We know that the $a_i$ converge strongly in $\LEaB4{}$ to $a$, and weak $\LGB41$ convergence of the $g_i$ implies that the $g_i$ converge strongly to $g$ in $\LGB{4/3}{}$, so by Lemma \ref{weaktostronglemma} the $g_ia_i$ converge strongly in $\LEaB4{}$ to $ga$. Likewise, the $b_ig_i$ converge strongly in $\LEaB4{}$ to $bg$. Thus, the $dg_i$ converge strongly in $\LEaB4{}$ to $dg$, so the $g_i$ converge strongly to $g$ in $\LGB41$, as desired.
    
    Next, we improve the strong $\LGB41$ convergence of the $g_i$ to strong $\LGB22$ convergence. We compute
    \begin{equation*}
      \nabla dg_i=dg_i\otimes a_i+g_i\nabla a_i-\nabla b_ig_i-b_i\otimes dg_i.
    \end{equation*}
    Because the $dg_i$ converge to $dg$ in $\LEaB4{}$ and the $a_i$ converge to $a$ in $\LaB4{}$, we know that the $dg_i\otimes a_i$ converge to $dg\otimes a$ in $\LnaB2{}$. For the $g_i\nabla a_i$ term, we again use Lemma \ref{weaktostronglemma}. We know that the $g_i$ converge strongly in $\LGB2{}$ to $g$ and the $\nabla a_i$ converge strongly in $\LnaB2{}$ to $\nabla a$, so by Lemma \ref{weaktostronglemma} the $g_i\nabla a_i$ converge strongly in $\LnaB2{}$ to $g\nabla a$. The $\nabla b_ig_i$ and $b_i\otimes dg_i$ terms are analogous. Thus, the $\nabla dg_i$ converge to $\nabla dg$ in $\LnaB2{}$, so the $g_i$ converge to $g$ in $\LGB22$, as desired.
  \end{proof}
\end{lemma}

\begin{corollary}\label{hausdorffcorollary}
  Let $P\to X$ be a principal $G$-bundle over a compact manifold $X$. The quotient space of $\Laa21X$ connections modulo $\LGG22X$ gauge transformations is Hausdorff, and likewise so is the space of $\Laa4{}X$ connections modulo $\LGG41X$ gauge transformations.
\end{corollary}

\subsection{Convergence of Coulomb gauge representatives}\label{coulombgaugeconvergencesection}
In this section, we extend Uhlenbeck's gauge fixing result with Neumann boundary conditions \cite{u82} to $\LB2d$ connections. Moreover, we show that if a sequence of small-energy connections converges strongly in $\LaB2d$, then a subsequence of the Coulomb gauge representatives converges strongly in $\LB21$. These results can also be proved analogously for the gauge fixing result with Dirichlet boundary conditions in Theorem \ref{dirichletcoulombfixing}. Either boundary condition will suffice for our purposes, so we present only the results for gauge fixing with Neumann boundary conditions. We begin by stating Uhlenbeck's gauge fixing theorem.

\begin{theorem}[{\cite[Theorem 2.1]{u82}}]\label{uhlenbeckgaugefixing}
  There exist constants $\epsilon$ and $C$ such that if $A$ is an $\LB21$ connection with $\normlb{F_A}2{}<\epsilon$, then there exists an $\LGB22$ gauge transformation sending $A$ to a connection $\tA=d+\ta$ satisfying
  \begin{enumerate}
  \item the Coulomb condition $d^*\ta=0$ on $B^4$,
  \item the boundary condition $i^*{*\ta}=0$ on $\pB$, and
  \item the bound $\normlb\ta21\le C\normlb{F_A}2{}$.
  \end{enumerate}
\end{theorem}

We aim to extend this theorem to $\LB2d$ connections. See also \cite[Lemma 2.2]{s02}.

\begin{theorem}\label{l2dgaugefixing}
  There exist constants $\epsilon$ and $C$ such that if $A$ is an $\LB2d$ connection with $\normlb{F_A}2{}<\epsilon$, then there exists an $\LGB41$ gauge transformation sending $A$ to an $\LB21$ connection $\tA=d+\ta$ satisfying
  \begin{enumerate}
  \item the Coulomb condition $d^*\ta=0$ on $B^4$,
  \item the boundary condition $i^*{*\ta}=0$ on $\pB$, and
  \item the bound $\normlb\ta21\le C\normlb{F_A}2{}$.
  \end{enumerate}
\end{theorem}

Theorem \ref{uhlenbeckgaugefixing} tells us that Theorem \ref{l2dgaugefixing} holds for $\LB21$ connections, so we use a closedness argument analogous to \cite[Lemma 2.4]{u82} and Lemma \ref{closedconstantboundary} to extend the result to $\LB2d$ connections.

\begin{proposition}\label{l2dclosed}
  Let $A_i$ be a sequence of $\LaB2d$ connections with $\normlb{F_{A_i}}2{}<\epsilon$ that satisfy Theorem \ref{l2dgaugefixing} and converge in $\LaB2d$ to a connection $A$. Then $A$ also satisfies Theorem \ref{l2dgaugefixing}.
  \begin{proof}
    Let $g_i$ and $\tA_i$ be the gauge transformations and connections given to us by Theorem \ref{l2dgaugefixing}. The $\normlb{F_{A_i}}2{}$ are bounded, and so the inequality $\normlb{\ta_i}21\le C\normlb{F_{A_i}}2{}$ tells us that the $\normlb{\ta_i}21$ are bounded also. Hence, passing to a subsequence, the $\ta_i$ have a weak $\LaB21$ limit, which we call $\ta$. The conditions $d^*\ta_i=0$ and $i^*{*\ta_i}=0$ are linear, and so are preserved under weak limits, giving us $d^*\ta=0$ and $i^*{*\ta}=0$. Meanwhile, $A_i$ converging to $A$ in $\LaB2d$ implies that $\normlb{F_{A_i}}2{}$ converges to $\normlb{F_A}2{}$, and norms are lower semicontinuous under weak limits, so the inequality $\normlb{\ta_i}21\le C\normlb{F_{A_i}}2{}$ is preserved in the weak limit, giving us $\normlb\ta21\le C\normlb{F_A}2{}$. Finally, since the $a_i$ and $\ta_i$ converge weakly in $\LaB4{}$, by Lemma \ref{weakgaugeconvergence}, a subsequence of the $g_i$ converges weakly in $\LGB41$ to a gauge transformation $g$ sending $A$ to $\tA$.
  \end{proof}
\end{proposition}

\begin{proof}[Proof of Theorem \ref{l2dgaugefixing}]
  Let $A$ be an $\LB2d$ connection with $\normlb{F_A}2{}<\epsilon$. Consider a sequence of smooth connections $A_i$ that converge to $A$ in $\LB2d$ and also have $\normlb{F_A}2{}<\epsilon$. The connections $A_i$ satisfy Theorem \ref{uhlenbeckgaugefixing}, and hence also Theorem \ref{l2dgaugefixing}. Then $A$ satisfies Theorem \ref{l2dgaugefixing} by Proposition \ref{l2dclosed}.
\end{proof}

In addition, we show that the weak subsequence convergence of the Coulomb gauge representatives above can be strengthened to strong subsequence convergence. Note, however, that taking a subsequence is necessary because Coulomb gauge is invariant under constant gauge transformations. By applying constant gauge transformations to a fixed connection in Coulomb gauge, we can construct a sequence of gauge equivalent connections in Coulomb gauge that does not converge, but because the gauge group is compact, we still expect convergence of a subsequence. In higher regularity, we can resolve this issue by considering infinitesimal gauge transformations that are orthogonal to the constant gauge transformations, but in a borderline Sobolev space, infinitemsimal gauge transformations are not so well-behaved, so a more delicate argument would be necessary. For our purposes, convergence of a subsequence will suffice.

\begin{proposition}\label{gaugefixingcontinuous}
  There exists an $\epsilon>0$ with the following significance. Let $A_i$ be a sequence of $\LaB2d$ connections with $\normlb{F_{A_i}}2{}<\epsilon$ that converge strongly in $\LaB2d$ to a connection $A$. Let $g_i$ and $g$ be the gauge transformations sending $A_i$ and $A$ to the Coulomb gauge representatives $\tA_i$ and $\tA$ constructed in Proposition \ref{l2dclosed}. Then, after passing to a subsequence, the $g_i$ converge strongly to $g$ in $\LGB41$ and the $\tA_i$ converge strongly to $\tA$ in $\LaB21$.
  \begin{proof}
    To promote the weak convergence in Proposition \ref{l2dclosed} to strong convergence, we use Lemma \ref{weaktostronglemma}. We know that, after passing to a subsequence, the $g_i$ converge weakly to $g$ in $\LGB41$. From the compact Sobolev embedding $\LGB41\hookrightarrow\LGB2{}$, we therefore know that the $g_i$ converge strongly to $g$ in $\LGB2{}$. In addition, we know that the $F_{A_i}$ converge strongly in $\LFB2{}$ to $F_A$. Hence, Lemma \ref{weaktostronglemma} tells us that the $F_{\tA_i}=g_iF_{A_i}g_i^{-1}$ converge strongly in $\LFB2{}$ to $gF_Ag^{-1}=F_\tA$.

    The next step is to show that convergence of curvature and Coulomb gauge implies convergence of the connections. Since $H^1(B^4)=0$, Corollary \ref{Di} tells us that $d+d^*\colon\LeB{2,\Neu}1\to\LeB2{}$ is a Fredholm operator with no kernel on one-forms. Hence, for some constant $C_G$, we have the inequality
    \begin{equation*}
      \normlb b21\le C_G\normlb{(d+d^*)b}2{}
    \end{equation*}
    for all $b\in\LaB{2,\Neu}1$. In particular, noting that $\ta_i,\ta\in\LaB{2,\Neu}1$ and $d^*\ta_i=d^*\ta=0$, we have
    \begin{equation*}
      \begin{split}
        \normlb{\ta_i-\ta}21&\le C_G\normlb{d(\ta_i-\ta)}2{}\\
        &=\normlb{F_{\tA_i}-F_\tA-\tfrac12\left([\ta_i\wedge\ta_i]-[\ta\wedge\ta]\right)}2{}\\
        &\le\normlb{F_{\tA_i}-F_\tA}2{}+\tfrac12\normlb{[(\ta_i+\ta)\wedge(\ta_i-\ta)]}2{}\\
        &\le\normlb{F_{\tA_i}-F_\tA}2{}+\tfrac12 C_\Lie C_S^2\left(\normlb{\ta_i}21+\normlb\ta21\right)\normlb{\ta_i-\ta}21\\
        &\le\normlb{F_{\tA_i}-F_\tA}2{}+\tfrac12 C_\Lie C_S^2C\left(\normlb{F_{A_i}}2{}+\normlb{F_A}2{}\right)\normlb{\ta_i-\ta}21,
      \end{split}
    \end{equation*}
    where $C_\Lie$ is the operator norm of the Lie bracket $[\cdot,\cdot]$, $C_S$ is the operator norm of the Sobolev embedding $\LB21\hookrightarrow\LB4{}$, and $C$ is the constant from Theorem \ref{l2dgaugefixing}. Thus we require that $\epsilon<\tfrac12(C_\Lie C_S^2C)^{-1}$, so our bounds $\normlb{F_{A_i}}2{}<\epsilon$ and $\normlb{F_A}2{}\le\epsilon$ imply that
    \begin{equation*}
      \normlb{\ta_i-\ta}21\le\normlb{F_{\tA_i}-F_\tA}2{}+\tfrac12\normlb{\ta_i-\ta}21.
    \end{equation*}
    Rearranging,
    \begin{equation*}
      \normlb{\ta_i-\ta}21\le2\normlb{F_{\tA_i}-F_\tA}2{}.
    \end{equation*}
    Hence, because the $F_{\tA_i}$ converge strongly to $F_\tA$ in $\LFB2{}$, the $\ta_i$ converge strongly to $\ta$ in $\LaB21$, as desired. Finally, by Lemma \ref{stronggaugeconvergence}, after passing to a subsequence, the $g_i$ converge strongly to $g$ in $\LGB41$.
  \end{proof}
\end{proposition}

\subsection{Coulomb gauge with fixed boundary}\label{fixedboundarysection}
In this section, we prove a gauge fixing result where the gauge transformation is fixed to be the identity on the boundary $\pB$, as a prelude to proving the gauge fixing result with Dirichlet boundary conditions on the connection in Section \ref{dirichletcoulombsection}. This result is present in Uhlenbeck's paper \cite{u82b} but with $L^\infty$ bounds on the connection. Here, our connection is in $\LaB21$ and may be unbounded in $L^\infty$. In our result in this section, we require bounds on the connection itself rather than its curvature as in \cite{u82} and Section \ref{dirichletcoulombsection}. Indeed, because the boundary value of the connection $A=d+a$ is fixed under the gauge transformation, curvature bounds alone are insufficient, and we also need to control the $\LapB2{1/2}$ norm of the boundary value $i^*a$. We finish this section with Proposition \ref{uniqueconstantboundary} where we show that the Coulomb gauge representative found by Proposition \ref{gaugefixingconstantboundary} is unique.

\begin{proposition}\label{gaugefixingconstantboundary}
  Let $B^4$ be a smooth $4$-ball with an arbitrary Riemannian metric. There exists constants $\epsilon$ and $C$ such that if $A=d+a$ is any $\LaB21$ connection with $\normlb a21<\epsilon$, then there exists an $\LGB22$ gauge transformation $g$ sending $A$ to $\tA=d+\ta$ such that

  \begin{enumerate}
  \item $i^*g$ is the identity gauge transformation on $\pB$,
  \item $\tA=g(A)$ is in Coulomb gauge, that is, $d^*\ta=0$, and
  \item $\normlb\ta21\le C\left(\normlb{F_A}2{}+\normlpb{i^*a}2{1/2}\right)$.\label{constantboundaryfixinginequality}
  \end{enumerate}
  Moreover, if $A$ is actually in $\LaB22$, then $g$ is in $\LGB23$.
\end{proposition}

We would like to find $g$ using the implicit function theorem. However, doing so requires the gauge group to have a differentiable exponential map, which is only true in higher regularity. Hence, we proceed similarly to the proof of Uhlenbeck's gauge fixing theorem with Neumann boundary conditions in \cite{u82}: We show that the space of $\LaB22$ connections satisfying Proposition \ref{gaugefixingconstantboundary} is both open and closed in $\LaB22$, and that the space of $\LaB21$ connections satisfying Proposition \ref{gaugefixingconstantboundary} is closed in $\LaB21$.

We first prove a priori bounds on connections in Coulomb gauge.

\begin{lemma}\label{aprioribounds}
  There exist constants $\epsilon$ and $C$, such that if $A=d+a$ is an $\LB21$ connection with $\normlb a4{}<\epsilon$ in Coulomb gauge $d^*a=0$, then
  \begin{equation*}
    \normlb a21\le C\left(\normlb{F_A}2{}+\normlpb{i^*a}2{1/2}\right).
  \end{equation*}
  Furthermore, if $A$ is an $\LaB22$ connection, then
  \begin{equation*}
    \normlb a22\le C\left(\normlb{\nabla_AF_A}2{}+\normlb{F_A}2{}+\normlpb{i^*a}2{3/2}\right).
  \end{equation*}
  \begin{proof}
    Because $H^1(B,\pB)=0$, Corollary \ref{Di} tells us that
    \begin{equation*}
      d+d^*\colon\LeB{2,\Dir}1\to\LeB2{}
    \end{equation*}
    is a Fredholm operator that is injective on one-forms. Because the trace map $i^*\colon\LB21\to\LpB2{1/2}$ is surjective, we conclude that
    \begin{equation*}
      (d+d^*,i^*)\colon\LeB21\to\LeB2{}\times\LepB2{1/2}
    \end{equation*}
    is also a Fredholm operator that is injective on one-forms. Indeed, the kernel of $(d+d^*,i^*)$ on $\LeB21$ is the same as the kernel of $d+d^*$ on $\LeB{2,\Dir}1$. Moreover, Corollary \ref{Di} tells us that $\range(d+d^*,i^*)\oplus(\mc H^\Dir\times\{0\})$ contains all of $\LeB2{}\times\{0\}$, and the surjectivity of $i^*$ tells us that, for any $\alpha_\partial\in\LepB2{1/2}$, the image contains a $(\beta,\alpha_\partial)$ for some $\beta$. Hence, $\range(d+d^*,i^*)\oplus(\mc H^\Dir\times\{0\})$ contains all of $\LeB2{}\times\LepB2{1/2}$, so $(d+d^*,i^*)$ is Fredholm.

    Likewise, because $d+d^*\colon\LeB{2,\Dir}2\to\LeB21$ is Fredholm and injective on one-forms and $i^*\colon\LeB22\to\LepB2{3/2}$ is surjective, we know that
    \begin{equation*}
      (d+d^*,i^*)\colon\LeB22\to\LeB21\times\LepB2{3/2}
    \end{equation*}
    is Fredholm and injective on one-forms. Hence, using $d^*a=0$, we have bounds
    \begin{align}
      \normlb a21&\le C_G\left(\normlb{da}2{}+\normlpb{i^*a}2{1/2}\right),\text{ and}\label{Dibound1}\\
      \normlb a22&\le C_G\left(\normlb{da}21+\normlpb{i^*a}2{3/2}\right).\label{Dibound2}
    \end{align}
    for some constant $C_G$. It remains to bound $da$ in terms of $F_A$. Using $da=F_A-\frac12[a\wedge a]$, we compute
    \begin{equation*}
      \normlb{da}2{}\le\normlb{F_A}2{}+\tfrac12 C_\Lie C_S\normlb a4{}\normlb a21,
    \end{equation*}
    where $C_\Lie$ is the operator norm of the Lie bracket and $C_S$ is the operator norm of the Sobolev embedding $\LB21\hookrightarrow\LB4{}$. Hence, requiring $\epsilon\le(C_\Lie C_S C_G)^{-1}$, we have that $\normlb a4{}<\epsilon$ implies
    \begin{equation}\label{dabound}
      \normlb{da}2{}\le\normlb{F_A}2{}+\tfrac12 C_G^{-1}\normlb a21.
    \end{equation}
    Combining with \eqref{Dibound1}, we see that
    \begin{equation*}
      \normlb a21\le C_G\left(\normlb{F_A}2{}+\normlpb{i^*a}2{1/2}\right)+\tfrac12\normlb a21.
    \end{equation*}
    Rearranging,
    \begin{equation*}
      \normlb a21\le 2C_G\left(\normlb{F_A}2{}+\normlpb{i^*a}2{1/2}\right),
    \end{equation*}
    as desired.

    We now proceed to the higher regularity. We compute
    \begin{equation*}
      \nabla da=\nabla F_A-\nabla\tfrac12[a\wedge a]=\nabla_AF_A-[a\otimes F_A]-\nabla\tfrac12[a\wedge a].
    \end{equation*}
    The squaring term $\tfrac12[a\wedge a]$ is more subtle at this regularity. We compute
    \begin{multline*}
      \norm{\nabla\tfrac12[a\wedge a]}_{\LB2{}}=\norm{[\nabla a\wedge a]}_{\LB2{}}\\
      \le C_\Lie C_S\normlb{\nabla a}21\normlb a4{}\le C_\Lie C_S\normlb a22\normlb a4{}.
    \end{multline*}
    Hence, we can improve on the Sobolev multiplication inequalities with
    \begin{equation*}
      \normlb{\tfrac12[a\wedge a]}21\le C_s\normlb a22\normlb a4{},
    \end{equation*}
    for a suitable constant $C_s$. As a result,
    \begin{equation*}
      \normlb{F_A}21\le \normlb{da}21+\normlb{\tfrac12[a\wedge a]}21
      \le \normlb a22+C_s\normlb a22\normlb a4{}.
    \end{equation*}
    Using these bounds, we compute
    \begin{equation*}
      \begin{split}
        \normlb{\nabla da}2{}&\le\normlb{\nabla_AF_A}2{}+C_\Lie C_S\normlb a4{}\normlb{F_A}21+\normlb{\nabla\tfrac12[a\wedge a]}2{}\\
        &\le\normlb{\nabla_AF_A}2{}+\left(2C_\Lie C_S\normlb a4{}+C_\Lie C_SC_s\normlb a4{}^2\right)\normlb a22
      \end{split}
  \end{equation*}
  By requiring $\epsilon<\tfrac18(2C_\Lie C_SC_G)^{-1}$ and $\epsilon^2<\tfrac18(C_\Lie C_SC_sC_G)^{-1}$, we have
    \begin{equation*}
      \normlb{\nabla da}2{}\le\normlb{\nabla_AF_A}2{}+\tfrac14C_G^{-1}\normlb a22.
    \end{equation*}
    Hence, using \eqref{dabound},
    \begin{equation*}
        \normlb{da}21\le\normlb{\nabla da}2{}+\normlb{da}2{}\le\normlb{\nabla_AF_A}2{}+\normlb{F_A}2{}+\tfrac34C_G^{-1}\normlb a22.
    \end{equation*}
    Combining with \ref{Dibound2}, we have
    \begin{equation*}
      \normlb a22\le C_G\left(\normlb{\nabla_AF_A}2{}+\normlb{F_A}2{}+\normlpb{i^*a}2{3/2}\right)+\tfrac34\normlb a22.
    \end{equation*}
    Thus,
    \begin{equation*}
      \normlb a22\le 4C_G\left(\normlb{\nabla_AF_A}2{}+\normlb{F_A}2{}+\normlpb{i^*a}2{3/2}\right).
    \end{equation*}
  \end{proof}
\end{lemma}

Now, we prove that the space of $\LaB22$ connections satisfying Proposition \ref{gaugefixingconstantboundary} is open.

\begin{lemma}\label{implicitfunctiontheoremconstantboundary}
  There exists an $\epsilon$ with the following significance. Let $A=d+a$ be an $\LB22$ connection with $\normlb a4{}<\epsilon$ and $d^*a=0$. Then there exists an open $\LaB22$ neighborhood of $A$ such that any connection $B$ in this neighborhood has an $\LGB23$ gauge transformation $g$ with $i^*g$ the identity that sends $B$ to a connection $\tB$ satisfying the Coulomb condition $d^*\tb=0$. Moreover, $g$ depends smoothly on $B$.
  \begin{proof}
    We search for a $g$ of the form $e^{-\gamma}$, where $\gamma$ is a $\mf g$-valued $\LB23$ function. For $\alpha$ small in $\LaB22$, we want a solution $\gamma$ to the equation
    \begin{equation}\label{gaugefixingconstantboundaryequation}
      \begin{split}
        d^*(e^{-\gamma}(a+\alpha)e^\gamma-(de^{-\gamma})e^\gamma)&=0,\\
        i^*\gamma&=0.
    \end{split}
    \end{equation}
    Letting $\LgB{2,\Dir}3$ denote $\LaB23$ functions $\gamma$ satisfying the boundary condition $i^*\gamma=0$, we consider the map
    \begin{align*}
      \LgB{2,\Dir}3\times\LaB22&\to\LgB21\\
      (\gamma,\alpha)&\mapsto d^*(e^{-\gamma}(a+\alpha)e^\gamma-(de^{-\gamma})e^\gamma)
    \end{align*}
    Because our gauge transformations are in a Sobolev space above the borderline, the exponential map is smooth, as are the relevant multiplication maps and linear maps in the above formula. To apply the implicit function theorem, we must show that the derivative of this map with respect to the $\gamma$ variable at $(\gamma,\alpha)=(0,0)$ is an isomorphism. This derivative map is
    \begin{equation*}
      \gamma'\mapsto d^*[a,\gamma']+d^*d\gamma'.
    \end{equation*}
    Call this map $T\colon\LgB{2,\Dir}3\to\LgB21$. By Propositions \ref{cohomology} and \ref{hodgedecomposition} and the fact that $H^0(B^4,\pB)=0$, we know that $\gamma'\mapsto d^*d\gamma'$ is an isomorphism as a map $\LgB{2,\Dir}3\to\LgB21$.

    Next, we show that the other term, $d^*[a,\gamma']$, is a compact operator as a map $\LgB{2,\Dir}3\to\LgB21$, so $T$ is a compact perturbation of $d^*d$, and hence is a Fredholm operator of index zero. We compute
    \begin{equation*}
      d^*[a,\gamma']=-*d[*a,\gamma']=-*[d(*a),\gamma']+*[*a\wedge d\gamma']
      =[d^*a,\gamma']+*[*a\wedge d\gamma']=*[*a\wedge d\gamma'],
    \end{equation*}
    We can view this term as a composition
    \begin{equation*}
      \LgB{2,\Dir}3\xrightarrow d\LaB22\hookrightarrow\LaB2{3/2}\xrightarrow{*[*a\wedge\cdot]}\LgB21.
    \end{equation*}
    The Sobolev multiplication and embedding theorems tell us that, because $a\in\LaB22$,  the above maps are continuous and the second inclusion is compact, so the composition is compact.

    Thus $T$ is Fredholm of index zero, so to show that $T$ is an isomorphism, it will suffice to prove that $T$ is injective. Working in one less degree of regularity, since $d^*d\colon\LgB{2,\Dir}2\to\LgB2{}$ is an isomoprhism, we know that there is a constant $\epsilon_\Delta$ such that $\normlb{d^*d\gamma'}2{}\ge \epsilon_\Delta\normlb{\gamma'}22$ for all $\gamma'\in\LgB{2,\Dir}2$. On the other hand, we know that
    \begin{equation*}
      \normlb{*[*a\wedge d\gamma']}2{}\le C_\Lie C_S\normlb a4{}\normlb{d\gamma'}21\le C_\Lie C_S\normlb a4{}\normlb{\gamma'}22,
    \end{equation*}
    where $C_\Lie$ is the operator norm of the Lie bracket bilinear form and $C_S$ is the operator norm of the Sobolev embedding $\LB21\hookrightarrow\LB4{}$. Hence, by requiring that $\epsilon<\epsilon_\Delta(C_\Lie C_S)^{-1}$, we see that $\normlb a4{}<\epsilon$ implies that $\normlb{*[*a\wedge d\gamma']}2{}<\epsilon_\Delta\normlb{\gamma'}22$, so
    \begin{equation*}
      \normlb{T\gamma'}2{}\ge\normlb{d^*d\gamma'}2{}-\normlb{*[*a\wedge d\gamma']}2{}>0.
    \end{equation*}
    Thus, $T$ is injective, and hence an isomorphism. Thus, the implicit function theorem gives us a solution $g=e^{-\gamma}$ to \eqref{gaugefixingconstantboundaryequation} depending smoothly on $\alpha$ in a neighborhood of $\alpha=0$. That is, we have a gauge transformation that is the identity on the boundary sending a connection $A+\alpha$ in an $\LaB22$ neighborhood of $A$ into Coulomb gauge, as desired.
  \end{proof}
\end{lemma}

\begin{corollary}\label{opencorollary}
  There exists an $\epsilon>0$ such that the space of $\LaB22$ connections $A=d+a$ with $\normlb a21<\epsilon$ satisfying Proposition \ref{gaugefixingconstantboundary} is open in $\LaB22$.
  \begin{proof}
    Let $\epsilon_4$ be the smaller of the constants in Lemmas \ref{aprioribounds} and \ref{implicitfunctiontheoremconstantboundary}, and let $C$ be the constant from Lemma \ref{aprioribounds}. Because $A\mapsto F_A$ is continuous as a map $\LaB21\to\LFB2{}$, and $a\mapsto i^*a$ is continuous as a map $\LaB21\to\LapB2{1/2}$, we can require that $\epsilon$ be small enough so that $\normlb a21<\epsilon$ implies $\normlb{F_A}2{}+\normlpb{i^*a}2{1/2}<\epsilon_4(C_SC)^{-1}$, where $C_S$ is the norm of the Sobolev embedding $\LB21\hookrightarrow\LB4{}$.

    Let $A$ be a connection with $\normlb a21<\epsilon$ satisfying Proposition \ref{gaugefixingconstantboundary}, so there is a gauge transformation $g$ sending $A$ to $\tA$ such that $i^*g$ is the identity and $d^*\ta=0$. We will show that a neighborhood of $\tA$ satisfies Proposition \ref{gaugefixingconstantboundary}, and then pull it back to a neighborhood of $A$. Condition \eqref{constantboundaryfixinginequality} of Proposition \ref{gaugefixingconstantboundary} implies that
    \begin{equation*}
      \normlb\ta4{}\le C_S\normlb\ta21\le C_SC\left(\normlb{F_A}2{}+\normlpb{i^*a}2{1/2}\right)<\epsilon_4.
    \end{equation*}
    Hence, we can apply Lemma \ref{implicitfunctiontheoremconstantboundary} to $\tA$ and find an open $\LaB22$ neighborhood of $\tA$ such that for any connection $B$ in the neighborhood, there is a $\LGB23$ gauge transformation that is the identity on the boundary and sends $B$ to a connection $\tB$ in Coulomb gauge. To verify condition \eqref{constantboundaryfixinginequality} of Proposition \ref{gaugefixingconstantboundary}, note that the implicit funcion theorem tells us that $g$ depends continously on $B$. Hence, by shrinking the neighborhood of $\tA$, we can guarantee that $g$ is close to the identity in $\LGB23$, and hence that $\tB$ is close to $B$ and hence to $\tA$ in $\LaB22$. In particular, since $\LB4{}$ is a weaker norm than $\LB22$, we can choose the neighborhood of $\tA$ small enough so that $\tB$ also satisfies $\normlb\tb4{}<\epsilon_4$. Then, we can apply Lemma \ref{aprioribounds} to $\tB$, to find that
    \begin{equation*}
      \normlb\tb21\le C\left(\normlb{F_{\tB}}2{}+\normlb{i^*\tb}2{1/2}\right)=C\left(\normlb{F_B}2{}+\normlb{i^*b}2{1/2}\right),
    \end{equation*}
    because $\normlb{F_B}2{}$ is invariant under gauge transformations, and $i^*g=1$ implies that $i^*\tb=i^*b$.
    
    We conclude that Proposition \ref{gaugefixingconstantboundary} holds on an open $\LaB22$ neighborhood of $\tA$. Since the conclusion of Proposition \ref{gaugefixingconstantboundary} is invariant under $\LGB23$ gauge transformations $g$ with $i^*g=1$, we can pull back this open neighborhood of $\tA$ via $g^{-1}$ to a neighborhood of $A$ satisfying Proposition \ref{gaugefixingconstantboundary}, as desired.
  \end{proof}
\end{corollary}

Finally, we prove that the set of connections satisfying Proposition \ref{gaugefixingconstantboundary} is closed.

\begin{lemma}\label{closedconstantboundary}
  The space of $\LaB21$ connections satisfying Proposition \ref{gaugefixingconstantboundary} is closed in $\LaB21$. Likewise, the space of $\LaB22$ connections with $\normlb a21<\epsilon$ satisfying Proposition \ref{gaugefixingconstantboundary} is closed in $\LaB22$.
  \begin{proof}
    Let $A_i\to A$ be a sequence of connections converging in $\LaB21$ such that there exist $\LGB22$ gauge transformations $g_i$ sending $A_i$ to $\tA_i$ satisfying $i^*g_i=1$, $d^*\ta_i=0$, and $\normlb{\ta_i}21\le C\left(\normlb{F_{A_i}}2{}+\normlpb{i^*a_i}2{1/2}\right)$. Our goal is to find a limit $g$ sending $A$ to $\tA$ that also satisfies these conditions.

    Because the $A_i$ converge in $\LaB21$, we know that the $F_{A_i}$ converge and are hence bounded in $\LFB2{}$, and the $i^*a_i$ converge and are hence bounded in $\LapB2{1/2}$. Hence, the above inequality tells us that the $\ta_i$ are bounded in $\LaB21$, so we can pass to a subsequence where the $\ta_i$ converge weakly in $\LaB21$. Let $\ta$ be the limit. By Lemma \ref{weakgaugeconvergence}, after passing to a subsequence, the $g_i$ converge weakly in $\LGB22$ to a gauge transformation $g$ sending $A$ to $\tA$.
    
    Since $i^*\colon\LGB41\to\LGpB4{3/4}$ is continuous and linear, the condition $i^*g_i=1$ is preserved under weak limits, giving us $i^*g=1$. Likewise, we have $d^*\ta=0$. Finally, to see that inequality \eqref{constantboundaryfixinginequality} of Proposition \ref{gaugefixingconstantboundary} is preserved in the weak limit, we note that norms are lower semicontinuous under weak limits, and that $A_i$ converging strongly to $A$ in $\LaB21$ implies that $\normlb{F_{A_i}}2{}+\normlpb{i^*a_i}2{1/2}$ converges to $\normlb{F_A}2{}+\normlpb{i^*a}2{1/2}$. Hence, we have $\normlb\ta21\le C\left(\normlb{F_A}2{}+\normlpb{i^*a}2{1/2}\right)$, as desired.
    
    We now proceed to prove closedness in higher regularity. Let $A_i\to A$ be a sequence of connections converging in $\LaB22$ with $\normlb{a_i}21<\epsilon$ and $\normlb a21<\epsilon$, such that there exist $\LGB23$ gauge transformations $g_i$ sending $A_i$ to $\tA_i$ satisfying $i^*g_i=1$, $d^*\ta_i=0$, and $\normlb{\ta_i}21\le C\left(\normlb{F_{A_i}}2{}+\normlpb{i^*a_i}2{1/2}\right)$.

    As in the proof of Corollary \ref{opencorollary}, this inequality, along with a small enough $\epsilon$, guarantess that $\normlb{\ta_i}4{}$ is small enough to apply Lemma \ref{aprioribounds}, giving us
    \begin{equation*}
      \begin{split}
        \normlb{\ta_i}22&\le C\left(\normlb{\nabla_{\tA_i}F_{\tA_i}}2{}+\normlb{F_{\tA_i}}2{}+\normlpb{i^*{\ta_i}}2{3/2}\right)\\
        &=C\left(\normlb{\nabla_{A_i}F_{A_i}}2{}+\normlb{F_{A_i}}2{}+\normlpb{i^*{a_i}}2{3/2}\right),
      \end{split}
    \end{equation*}
    using the fact that $\normlb{F_A}2{}$ and $\normlb{\nabla_AF_A}2{}$ are gauge invariant quantities, and that $i^*a_i=i^*\ta_i$. Since $A_i$ converges to $A$ in $\LaB22$, we conclude that the right-hand side of the inequality is bounded, and hence a subsequence of the $\ta_i$ converges weakly in $\LaB22$. Let $\ta$ be its limit. The above argument for $\LaB21$ connections gives us a gauge transformation $g\in\LGB22$ sending $A$ to $\tA$ satisfying all of the conditions of Proposition \ref{gaugefixingconstantboundary}, so it only remains to show that $g$ is actually in $\LB23$. We prove this claim in two steps from the equation $dg=ga-\ta g$. First, note that the multiplication $\LB22\times\LB22\to\LB31$ is continuous. Hence, since $g,a,\ta\in\LB22$, we know that $dg\in\LEaB31$, so $g\in\LGB32$. Next, since the multiplication $\LB32\times\LB22\to\LB22$ is continuous, we have that $dg\in\LEaB22$, so $g\in\LGB23$, as desired.
  \end{proof}
\end{lemma}

These lemmas complete the proof of Proposition \ref{gaugefixingconstantboundary}.

\begin{proof}[Proof of \ref{gaugefixingconstantboundary}]
  The space of $\LB22$ connections $A$ with $\normlb a21<\epsilon$ is connected, and by Corollary \ref{opencorollary} and Lemma \ref{closedconstantboundary} the space of $\LB22$ connections with $\normlb a21<\epsilon$ satisfying Proposition \ref{gaugefixingconstantboundary} is both open and closed, and hence contains all $\LB22$ connections with $\normlb a21<\epsilon$. Meanwhile, any $\LB21$ connection $A$ with $\normlb a21<\epsilon$ is the limit in $\LaB21$ of sequence $A_i$ of $\LaB22$ connections with $\normlb a21<\epsilon$. Because the $A_i$ satisfy Proposition \ref{gaugefixingconstantboundary}, Lemma \ref{closedconstantboundary} tells us that so does $A$.
\end{proof}

We also prove that the gauge transformation constructed by Proposition \ref{gaugefixingconstantboundary} is unique, at least with an appropriate choice of constants. We require $\LaB4{}$ bounds on the Coulomb gauge representatives, but note that these follow from condition \ref{constantboundaryfixinginequality} and the bounds on $\normlb a21$ in Proposition \ref{gaugefixingconstantboundary}. In addition, for use in the future, we will assume that $i^*g$ is a constant gauge transformation on $\pB$ but not necessarily the identity.

\begin{proposition}\label{uniqueconstantboundary}
  There exists a constant $\epsilon$ such that if $A=d+a$ and $B=d+b$ are two $\LaB21$ connections gauge equivalent via a gauge transformation $g\in\LGB22$ satisfying
  \begin{enumerate}
  \item bounds $\normlb a4{},\normlb b4{}<\epsilon$,
  \item the boundary condition that $i^*g$ is equal to a constant $c\in G$ on $\pB$, and
  \item the Coulomb condition $d^*a=d^*b=0$,
  \end{enumerate}
  then $g$ is the constant gauge transformation $c$ on all of $B^4$.
  \begin{proof}
    We have the gauge equivalence equation
    \begin{equation*}
      dg=ga-bg.
    \end{equation*}
    Thus, using $d^*a=d^*b=0$, we have
    \begin{equation}\label{dsdg}
      d^*dg=-{*d}(g{*a}-{*bg})=-{*}(dg\wedge*a+gd(*a)-d(*b)g+*b\wedge dg)
      =-{*}(dg\wedge *a+*b\wedge dg).
    \end{equation}
    Hence,
    \begin{equation*}
      \normlb{d^*dg}2{}\le\left(\normlb a4{}+\normlb b4{}\right)\normlb{dg}4{}
      \le C_S\left(\normlb a4{}+\normlb b4{}\right)\normlb{dg}21.
    \end{equation*}
    On the other hand, since $H^1(B^4,\pB)=0$, Corollary \ref{Di} tells us that
    \begin{equation*}
      d+d^*\colon\LEe{2,\Dir}1{B^4}\to\LEe2{}{B^4}
    \end{equation*}
    is a Fredholm operator with no kernel on one-forms. The boundary condition on $g$ implies that $i^*dg=d_\pB i^*g=d_\pB c=0$, so $dg$ is in $\LEaB{2,\Dir}1$. Thus, there is a constant $C_G$ independent of $g$ such that
    \begin{equation*}
      \normlb{dg}21\le C_G\normlb{(d+d^*)dg}2{}=C_G\normlb{d^*dg}2{}.
    \end{equation*}
    Combining these inequalities, we have
    \begin{equation*}
      \normlb{dg}21\le C_GC_S\left(\normlb a4{}+\normlb b4{}\right)\normlb{dg}21.
    \end{equation*}
    Thus, requiring $\epsilon\le\tfrac14(C_GC_S)^{-1}$, the condition $\normlb a4{},\normlb b4{}<\epsilon$ implies that
    \begin{equation*}
      \normlb{dg}21\le\frac12\normlb{dg}21,
    \end{equation*}
    so $dg=0$. Thus $g$ is constant on $B^4$. Since $i^*g=c$ on $\pB$, we conclude that $g=c$ on all of $B^4$, as desired.
  \end{proof}
\end{proposition}

\subsection{Coulomb gauge with Coulomb gauge on the boundary}\label{dirichletcoulombsection}
In this section, we prove a second gauge fixing result, where we show that if a connection has small energy, then it is gauge equivalent to a connection $\tA=d+\ta$ that satisfies the Coulomb condition $d^*\ta=0$ on $B^4$ and whose restriction to the boundary $i^*\ta$ satisfies the the Coulomb condition $d^*_\pB(i^*\ta)=0$ on $\pB$, where $d^*_\pB$ denotes the adjoint of the differential $d_\pB$ on $\pB$ with respect to the metric on $\pB$. As in the previous section, we finish with Corollary \ref{dirichletcoulombuniqueness} where we show that the Dirichlet Coulomb gauge representative found by Theorem \ref{dirichletcoulombfixing} is unique up to constant gauge transformations.

\begin{theorem}\label{dirichletcoulombfixing}
  Let $B^4$ be a $4$-ball with an arbitrary Riemannian metric. There exist constants $\epsilon$ and $C$ such that if $A$ is any $\LaB21$ connection with $\normlb{F_A}2{}<\epsilon$, then there exists an $\LaB21$ connection $\tA$ gauge equivalent to $A$ by an $\LGB22$ gauge transformation such that
  \begin{enumerate}
  \item $\tA$ is in Dirichlet Coulomb gauge, that is, $d^*\ta=0$ on $B^4$ and $d^*_\pB(i^*\ta)=0$ on $\pB$, and
  \item $\normlb\ta21\le C\normlb{F_A}2{}$.
  \end{enumerate}
  Moreover, if $A$ is in $\LaB22$, then $g\in\LGB23$.
\end{theorem}
Gauge fixing with the Dirichlet Coulomb condition $d^*\ta=0$ and $d^*_\pB i^*\ta=0$ is shown in Uhlenbeck's paper \cite[Theorem 2.7]{u82b}, but again with $L^\infty$ bounds. Marini \cite{m92} improves this result to $L^2_1$ connections, but with the additional assumption that, on the boundary, $\normlpb{i^*F_A}2{}<\epsilon$. We remove this condition, so $\normlpb{i^*F_A}2{}$ need not even be finite. As in the previous section and \cite{u82}, we first work in higher regularity and prove that the space of $\LaB22$ connections satisfying Theorem \ref{dirichletcoulombfixing} is both open and closed in $\LaB22$, and then prove the result for $\LaB21$ connections by showing that the space of $\LaB21$ connections satisfying \ref{dirichletcoulombfixing} is closed in $\LaB21$. We begin by strenghtening the a priori bounds in Lemma \ref{aprioribounds} to this setting.

\begin{lemma}\label{aprioribounds2}
  There exist constants $\epsilon$ and $C$, such that if $A=d+a$ is an $\LB21$ connection with $\normlb a4{}<\epsilon$ in Dirichlet Coulomb gauge, that is, $d^*a=0$ and $d^*_\pB i^*a=0$, then
  \begin{equation*}
    \normlb a21\le C\normlb{F_A}2{}.
  \end{equation*}
  Furthermore, if $A$ is an $\LaB22$ connection, then
  \begin{equation*}
    \normlb a22\le C\left(\normlb{\nabla_AF_A}2{}+\normlb{F_A}2{}\right).
  \end{equation*}
  \begin{proof}
    Recall \eqref{Dibound1} and \eqref{Dibound2}.
    \begin{align*}
      \normlb a21&\le C_G\left(\normlb{da}2{}+\normlpb{i^*a}2{1/2}\right),\\
      \normlb a22&\le C_G\left(\normlb{da}21+\normlpb{i^*a}2{3/2}\right).
    \end{align*}
    The key idea is to absorb the $i^*a$ terms by proving that $d^*_\pB i^*a=0$ implies that $\normlpb{i^*a}2{1/2}\le C_F\normlb{da}2{}$ and $\normlpb{i^*a}2{3/2}\le C_F\normlb{da}21$ for some constant $C_F$. Since $d_\pB^{}+d^*_\pB$ is elliptic and $H^1(\pB)=0$, we know that $d_\pB^{}+d^*_\pB$ is a Fredholm operator with no kernel on one forms. Thus, there is a constant $C_g$ such that
    \begin{multline*}
      \normlpb{i^*a}2{3/2}\le C_g\normlpb{(d_\pB^{}+d^*_\pB)(i^*a)}2{1/2}=C_g\normlpb{d_\pB i^*a}2{1/2}\\
      =C_g\normlpb{i^*(da)}2{1/2}\le C_gC_T\normlb{da}21,
    \end{multline*}
    where $C_T$ is the operator norm of the trace map $\LB21\to\LpB2{1/2}$. We would like to do the same argument in lower regularity, but the trace map $\LB2{}\to\LpB2{-1/2}$ is unbounded. However, we can still get the inequality $\normlpb{d_\pB i^*a}2{-1/2}\le C_T\normlb{da}2{}$ using the Hodge decomposition, as we show in Lemma \ref{dialeda}. For now, we continue with this assumption. Thus, by the same argument,
    \begin{equation*}
      \normlpb{i^*a}2{1/2}\le C_g\normlpb{(d_\pB^{}+d^*_\pB)(i^*a)}2{-1/2}
      =C_g\normlpb{d_\pB i^*a}2{-1/2}\le C_gC_T\normlb{da}2{}.
    \end{equation*}
    Hence, we have
    \begin{align}
      \normlb a21&\le C_G(C_gC_T+1)\normlb{da}2{},\label{a21bound}\\
      \normlb a22&\le C_G(C_gC_T+1)\normlb{da}21.\label{a22bound}
    \end{align}

    At this point, we can follow the argument of Lemma \ref{aprioribounds} with $C_G(C_gC_T+1)$ in place of $C_G$. By choosing $\epsilon$ small enough, we can have $\normlb a4{}<\epsilon$ imply
    \begin{equation*}
      \normlb{da}2{}\le\normlb{F_A}2{}+\tfrac12C_G^{-1}(C_gC_T+1)^{-1}\normlb a21.
    \end{equation*}
    Combing with \eqref{a21bound} and rearranging, we have
    \begin{equation*}
      \normlb a21\le 2C_G(C_gC_T+1)\normlb{F_A}2{},
    \end{equation*}
    as desired.

    Likewise, in higher regularity, we can choose $\epsilon$ small enough to guarantee
    \begin{equation*}
      \normlb{da}21\le\normlb{\nabla_AF_A}2{}+\normlb{F_A}2{}+\tfrac34C_G^{-1}(C_gC_T+1)^{-1}\normlb a22.
    \end{equation*}
    Combining with \eqref{a22bound} and rearranging, we have
    \begin{equation*}
      \normlb a22\le4C_G(C_gC_T+1)\left(\normlb{\nabla_AF_A}2{}+\normlb{F_A}2{}\right),
    \end{equation*}
    as desired.
  \end{proof}
\end{lemma}

\begin{lemma}\label{dialeda}
  Let $X$ be a compact smooth manifold with boundary, and let $\alpha$ be a differential form in $\LeX21$. There is a constant $C_T$ independent of $\alpha$ such that
  \begin{equation*}
    \normlpx{d_\pX i^*\alpha}2{-1/2}\le C_T\normlx{d\alpha}2{}.
  \end{equation*}
  \begin{proof}
    It suffices to consider smooth $\alpha$ because smooth forms are dense in $\LeX21$ and the linear maps $d_\pX\circ i^*\colon\LeX21\to\LepX2{-1/2}$ and $d\colon\LeX21\to\LeX2{}$ are continuous. By Proposition \ref{hodgedecomposition},
    \begin{equation*}
      \alpha=dd^*G^\Neu\alpha+d^*dG^\Neu\alpha+\pi^\Neu_{\mc H}\alpha.
    \end{equation*}
    Let $\beta=d^*dG^\Neu\alpha$, which is smooth by Proposition \ref{hodgedecomposition} because $\alpha$ is smooth, so both $i^*d\alpha$ and $i^*d\beta$ are well-defined. By the above equation, we see that
    \begin{gather*}
      d\alpha=d\beta,\\
      d_\pX i^*\alpha=i^*d\alpha=i^*d\beta=d_\pX i^*\beta.
    \end{gather*}
    Hence, it suffices to prove our lemma for $\beta$.

    It is clear that $d^*\beta=0$. Moreover, by the boundary conditions on the range of $G^\Neu$ in Proposition \ref{hodgedecomposition}, we see that
    \begin{equation*}
      i^*{*\beta}=i^*{*d^*}dG^\Neu\alpha=\pm i^*d{*d}G^\Neu\alpha=\pm d_\pX i^*{*d}G^\Neu\alpha=\pm d_\pX i^*d^*{*G^\Neu}\alpha=0.
    \end{equation*}
    Hence, we can apply Corollary \ref{Di}, noting that Proposition \ref{hodgedecomposition} gives us that $\beta$ is orthogonal to $\mc H^\Neu$. In other words, we have a Fredholm operator
    \begin{equation*}
      d+d^*\colon\LeX{2,\Neu}1\to\LeX2{},
    \end{equation*}
    and $\beta$ is orthogonal to its kernel, so there is a constant $C_G$ independent of $\beta$ such that
    \begin{equation*}
      \normlx\beta21\le C_G\normlx{(d+d^*)\beta}2{}=C_G\normlx{d\beta}2{}.
    \end{equation*}
    At this point, proving the claim for $\beta$ is straightforward. Let $C$ be the operator norm of $d_\pX\circ i^*\colon\LeX21\to\LepX2{-1/2}$. We have
    \begin{equation*}
      \normlpx{d_\pX i^*\beta}2{-1/2}\le C\normlx\beta21\le CC_G\normlx{d\beta}2{},
    \end{equation*}
    as desired, letting $C_T=CC_G$. Since $d\alpha=d\beta$ and $d_\pX i^*\alpha=d_\pX i^*\beta$, we also have
    \begin{equation*}
      \normlpx{d_\pX i^*\alpha}2{-1/2}\le C_T\normlx{d\alpha}2{}
    \end{equation*}
    for smooth $\alpha$, and the aforementioned density argument gives us the inequality for all $\alpha\in\LeX21$.
  \end{proof}
\end{lemma}

As in the previous section, our next step is to prove openness in higher regularity.

\begin{lemma}\label{implicitfunctiondirichletfixing}
  There exist constants $\epsilon_4$ and $\epsilon_3$ with the following significance. Let $A=d+a$ be an $\LaB22$ connection with $\normlb a4{}<\epsilon_4$, $\normlpb{i^*a}3{}<\epsilon_3$, $d^*a=0$, and $d^*_\pB i^*a=0$. Then there exists an open $\LaB22$ neighborhood of $A$ such that any connection $B$ in this neighborhood has an $\LGB23$ gauge transformation $g$ that sends $B$ to a connection $\tB$ satisfying the Dirichlet Coulomb conditions $d^*\tb=0$ and $d^*_\pB i^*\tb=0$. Moreover, $g$ depends smoothly on $B$.
  \begin{proof}
    We adapt the proof of Lemma \ref{implicitfunctiontheoremconstantboundary}. Again, we search for a $g$ of the form $e^{-\gamma}$ for $\gamma\in\LgB23$. For $\alpha$ small in $\LaB22$, we want a solution $\gamma$ to the system
    \begin{equation}\label{dirichletcoulombsystem}
      \begin{split}
        d^*(e^{-\gamma}(a+\alpha)e^\gamma-(de^{-\gamma})e^\gamma)&=0,\\
        d^*_\pB i^*(e^{-\gamma}(a+\alpha)e^\gamma-(de^{-\gamma})e^\gamma)&=0.
      \end{split}
    \end{equation}
    This time, we need to deal with the fact that Dirichlet Coulomb representatives are unique only up to constant gauge transformations, so in order to obtain an isomorphism, we need to make sure that our spaces of infinitesimal gauge transformations do not contain nonzero constant gauge transformations. For $k\ge1$, let $\LgpB{2,\perp}{k-1/2}$ denote those $\LgpB2{k-1/2}$ functions that are $\LpB2{}$-orthogonal to $\mc H^0(\pB)$, that is, orthogonal to the constant functions on the $3$-sphere. Let $\LgB{2,\perp}k$ denote the inverse image of the closed subspace $\LgpB{2,\perp}{k-1/2}$ under the restriction map $i^*\colon\LgB2k\to\LgpB2{k-1/2}$. We consider the map
    \begin{gather*}
      \LgB{2,\perp}3\times\LaB22\to\LgB21\times\LgpB{2,\perp}{1/2},\\
      (\gamma,\alpha)\mapsto\left(d^*(e^{-\gamma}(a+\alpha)e^\gamma-(de^{-\gamma})e^\gamma),d^*_\pB i^*(e^{-\gamma}(a+\alpha)e^\gamma-(de^{-\gamma})e^\gamma)\right).
    \end{gather*}
    Again, this map is smooth because the relevant Sobolev spaces above the borderline. Moreover, the range of the second component is indeed in $\LgpB{2,\perp}{1/2}$ because it is in the range of $d^*_\pB$, and it is easy to verify that if $\phi$ is a constant map, then $\pairlb{d^*_\pB\beta}\phi=\pairlb\beta{d_\pB\phi}=0$ for all $\beta$.

    To apply the implicit function theorem, we show that the derivative of this map with respect to the $\gamma$ variable at $(\gamma,\alpha)=(0,0)$ is an isomorphism. Call this map $T$. This map is
    \begin{align*}
      T&\colon\LgB{2,\perp}3\to\LgB21\times\LgpB{2,\perp}{1/2},\\
      T&\colon\gamma'\to\left(d^*[a,\gamma']+d^*d\gamma',d^*_\pB i^*[a,\gamma']+d^*_\pB i^*d\gamma'\right).
    \end{align*}
    As in the proof of Lemma \ref{implicitfunctiontheoremconstantboundary}, our goal is to show that $T$ is a Fredholm operator of index zero by decomposing $T$ as a sum $T=T_0+K$ where
    \begin{align*}
      T_0&\colon\gamma'\mapsto\left(d^*d\gamma',d^*_\pB i^*d\gamma')\right),\\
      K&\colon\gamma'\mapsto\left(d^*[a,\gamma'],d^*_\pB i^*[a,\gamma']\right).
    \end{align*}
    We then show that $T_0$ is an isomorphism and $K$ is compact. Note that $d^*_\pB i^*d\gamma'=d^*_\pB d_\pB^{}i^*\gamma'$. Hence, $T_0\colon\LgB{2,\perp}3\to\LgB21\times\LgpB{2,\perp}{1/2}$ is the composition of the maps
    \begin{align*}
      (\Delta,i^*)&\colon\LgB{2,\perp}3\to\LgB21\times\LgpB{2,\perp}{5/2},\\
      (\Id,\Delta_\pB)&\colon\LgB21\times\LgpB{2,\perp}{5/2}\to\LgB21\times\LgpB{2,\perp}{1/2}.
    \end{align*}
    The fact that $\Delta_\pB\colon\LgpB{2,\perp}{k+3/2}\to\LgpB{2,\perp}{k-1/2}$ is an isomorphism follows from the usual Hodge decomposition on closed manifolds, since, by definition, we restrict the domain and range to the orthogonal complement of the harmonic functions $\mc H^0(\pB)$, that is, the constant functions.

    As for $(\Delta,i^*)$, as before, we know from Propositions \ref{cohomology} and \ref{hodgedecomposition} and the fact that $\mc H^0(B^4,\pB)=0$ that $\Delta\colon\LgB{2,\Dir}{k+2}\to\LgB2k$ is an isomorphism for $k\ge0$, where $\LgB{2,\Dir}{k+2}=\LgB2{k+2}\cap\ker i^*=\LgB{2,\perp}{k+2}\cap\ker i^*$. The injectivity of $(\Delta,i^*)$ follows. For surjectivity, we use a standard argument using the surjectivity of $i^*$. Indeed, the inverse trace map \cite[Theorem 7.53]{a75} gives us surjectivity of $i^*\colon\LgB2{k+2}\to\LgpB2{k+3/2}$ for $k\ge-1$. Since $\LgB{2,\perp}{k+2}$ is defined as in the inverse image of $\LgpB{2,\perp}{k+3/2}$ under this map, we know that $i^*\colon\LgB{2,\perp}{k+2}\to\LgpB{2,\perp}{k+3/2}$ is also surjective. Given $(\beta,\gamma_\partial)\in\LgB2k\times\LgpB{2,\perp}{k+3/2}$, let $\gamma_1\in\LgB{2,\perp}{k+2}$ be such that $i^*\gamma_1=\gamma_\partial$. Meanwhile, we use the surjectivity of $\Delta\colon\LgB{2,\Dir}{k+2}\to\LgB2k$ to find a $\gamma_2\in\LgB{2,\Dir}{k+2}$ such that $\Delta\gamma_2=\Delta\gamma_1-\beta$. Then, $\gamma_1-\gamma_2\in\LgB{2,\perp}{k+2}$, and we have $\Delta(\gamma_1-\gamma_2)=\beta$ and $i^*(\gamma_1-\gamma_2)=i^*\gamma_1=\gamma_\partial$, as desired. Setting $k=1$ gives us that $T_0\colon\LgB{2,\perp}3\to\LgB21\times\LgpB{2,\perp}{1/2},\\$ is an isomorphism.

    Next, we show that $K$ is compact. In the proof of Lemma \ref{implicitfunctiontheoremconstantboundary}, we computed that if $d^*a=0$, then
    \begin{equation*}
      d^*[a,\gamma']=*[*a\wedge d\gamma'].
    \end{equation*}
    The same argument shows that if $d^*_\pB i^*a=0$, then
    \begin{equation*}
      d^*_\pB i^*[a,\gamma']=d^*_\pB[i^*a,i^*\gamma']=*_\pB[*_\pB i^*a\wedge d_\pB i^*\gamma'],
    \end{equation*}
    where $*_\pB$ denotes the Hodge star operator on the sphere $\pB$. As before, we view $\gamma'\mapsto d^*[a,\gamma']$ as the composition
    \begin{equation*}
      \LgB{2,\perp}3\xrightarrow d\LaB22\hookrightarrow\LaB2{3/2}\xrightarrow{*[*a\wedge\cdot]}\LgB21,
    \end{equation*}
    and the Sobolev multiplication and embedding theorems, along with $a\in\LgB22$ and the smoothness of $*$, tells us that the maps above are continuous, and the second one is compact, so the composition is compact. Likewise, $\gamma'\mapsto d^*_\pB i^*[a,\gamma']$ is the composition of the maps
    \begin{multline*}
      \LgB{2,\perp}3\xrightarrow{i^*}\LgpB{2,\perp}{5/2}\xrightarrow{d_\pB}\LapB2{3/2}\\
      \hookrightarrow\LapB21\xrightarrow{*_\pB[*_\pB i^*a\wedge\cdot]}\LgpB2{1/2}.
    \end{multline*}
    Again, the inclusion is compact, so the composition is compact.

    We conclude that $K$ is compact, so $T$ is indeed a Fredholm operator of index zero. Hence, to show that $T$ is an isomorphism, it suffices to show that $T$ is injective. We show that this is indeed the case, assuming $\normlb a4{}<\epsilon_4$ and $\normlpb{i^*a}3{}<\epsilon_3$ for $\epsilon_4$ and $\epsilon_3$ small enough. Now setting $k=0$ in the above argument gives us that, in one degree lower regularity, $T_0\colon\LgB{2,\perp}2\to\LgB2{}\times\LgpB{2,\perp}{-1/2}$ is also an isomorphism, so there exists an $\epsilon_\Delta$ such that
    \begin{equation*}
      \norm{T_0\gamma'}_{\LB2{}\times\LpB2{-1/2}}\ge\epsilon_\Delta\normlb{\gamma'}22.
    \end{equation*}
    Next we bound $\norm{K\gamma'}_{\LB2{}\times\LpB2{-1/2}}$ from above. Since $*$ is an isometry, we have
    \begin{equation*}
      \normlb{d^*[a,\gamma']}2{}=\normlb{*[*a\wedge d\gamma']}2{}
      \le C_\Lie\normlb a4{}\normlb{d\gamma'}4{}\le C_\Lie C_S\normlb a4{}\normlb{\gamma'}22,
    \end{equation*}
    where $C_\Lie$ is the operator norm of the Lie bracket $[\cdot,\cdot]$ and $C_S$ is the operator norm of the Sobolev embedding $\LB21\hookrightarrow\LB4{}$. Likewise,
    \begin{multline*}
      \normlpb{d^*_\pB i^*[a,\gamma']}2{-1/2}\le C_s\normlpb{*_\pB[*_\pB i^*a\wedge d_\pB i^*\gamma']}{3/2}{}\\
      \le C_sC_\Lie\normlpb{i^*a}3{}\normlpb{d_\pB i^*\gamma'}3{}\le C_s^2C_\Lie C_T\normlpb{i^*a}3{}\normlb{\gamma'}22,
    \end{multline*}
    where $C_s$ denotes the operator norms of the embeddings $\LpB{3/2}{}\hookrightarrow\LpB2{-1/2}$ and $\LpB2{1/2}\hookrightarrow\LpB3{}$ and $C_T$ is the norm of the trace operator $i^*\colon\LgB22\to\LgpB2{3/2}$. Hence, by choosing $\epsilon_4$ and $\epsilon_3$ small enough, we can guarantee that $\normlb a4{}<\epsilon_4$ and $\normlpb{i^*a}3{}<\epsilon_3$ imlies that
    \begin{equation*}
      \norm{K\gamma'}_{\LB2{}\times\LpB2{-1/2}}\le\tfrac12\epsilon_\Delta\normlb{\gamma'}22.
    \end{equation*}
    As a consequence, since $T=T_0+K$, we know that
    \begin{equation*}
      \norm{T\gamma'}_{\LB2{}\times\LpB2{-1/2}}\ge\tfrac12\epsilon_\Delta\normlb{\gamma'}22,
    \end{equation*}
    so $T$ is injective. Since $T$ has Fredholm index zero, we know that $T$ is an isomorphism, so the implicit function theorem gives us a solution $g=e^{-\gamma}$ to the system \eqref{dirichletcoulombsystem} that depends smoothly on $\alpha$ in a neighborhood of $\alpha=0$. That is, for any connection in a neighborhood of $A$, we have a gauge transformation sending it to a connection satisfying the Dirichlet Coulomb conditions, as desired.
  \end{proof}
\end{lemma}

Note that if instead of using the multiplication map $\LpB3{}\times\LpB3{}\to\LpB2{-1/2}$ above we had used the multiplication map $\LpB6{-1/2}\times\LpB2{1/2}\to\LpB2{-1/2}$ we could weaken the condition that $\normlpb{i^*a}3{}$ be small to the condition that $\normlpb{i^*a}6{-1/2}$ be small. Alternatively, because of the continuity of the maps $\LB21\hookrightarrow\LB4{}$ and $\LB21\xrightarrow{i^*}\LpB2{1/2}\hookrightarrow\LpB3{}$, we can replace the conditions $\normlb a4{}<\epsilon_4$ and $\normlpb{i^*a}3{}<\epsilon_3$ in the above lemma with $\normlb a21<\epsilon_{21}$ for a suitable $\epsilon_{21}$.

\begin{corollary}\label{opencorollary2}
  The space of $\LaB22$ connections $A$ with $\normlb{F_A}2{}<\epsilon$ satisfying Theorem \ref{dirichletcoulombfixing} is open in $\LaB22$.
  \begin{proof}
    We let $\epsilon_4$ be the smaller of the constants in Lemma \ref{implicitfunctiondirichletfixing} and Lemma \ref{aprioribounds2}, let $\epsilon_3$ be the constant in Lemma \ref{implicitfunctiondirichletfixing}, and let $C$ be the constant in \ref{aprioribounds2}. We can choose $\epsilon$ small enough such that $\normlb\ta21<C\epsilon$ implies $\normlb\ta4{}<\epsilon_4$ and $\normlpb{i^*\ta}3{}<\epsilon_3$.

    Let $A$ be an $\LaB22$ connection with $\normlb{F_A}2{}<\epsilon$ satisfying Theorem \ref{dirichletcoulombfixing}. Then there exists an $\LGB23$ gauge transformation $g$ sending $A$ to $\tA$ such that $d^*\ta=0$, $d^*_\pB i^*\ta=0$, and $\normlb\ta21\le C\normlb{F_A}2{}<C\epsilon$. As discussed earlier, this implies that $\ta$ is small enough to apply Lemmas \ref{aprioribounds2} and Lemma \ref{implicitfunctiondirichletfixing}. Hence, we apply Lemma \ref{implicitfunctiondirichletfixing} to $\tA$ to find an open $\LaB22$ neighborhood of $\tA$ such that for any connection $B$ in the neighborhood has a gauge transformation $g$ that sends $B$ to a connection $\tB$ satisfying the Dirichlet Coulomb conditions $d^*\tb=0$ and $d^*_\pB i^*\tb=0$. Since $g$, and hence $\tB$, depends smoothly on $B$, by shrinking the neighborhood of $\tA$ we can guarantee that $\tB$ is close to $\tA$ in $\LaB22$. Hence, we can guarantee that $\tB$ also satisfies the bounds $\normlb\tb4{}<\epsilon_4$, so we can apply Lemma \ref{aprioribounds2} to $\tb$ to get
    \begin{equation*}
      \normlb\tb21\le C\normlb{F_{\tB}}2{}=C\normlb{F_B}2{}.
    \end{equation*}
    Hence, Theorem \ref{dirichletcoulombfixing} holds for $B$ in this neighborhood of $\tA$. Since Theorem \ref{dirichletcoulombfixing} is gauge invariant, we can pull back this neighborhood of $\tA$ via $g^{-1}$ to an open neighborhood of $A$ that satisfies Theorem \ref{dirichletcoulombfixing}, as desired.
  \end{proof}
\end{corollary}

\begin{lemma}\label{closedlemma2}
  The space of $\LB21$ connections satisfying Theorem \ref{dirichletcoulombfixing} is closed in $\LaB21$. Likewise, the space of $\LB22$ connections with $\normlb{F_A}2{}<\epsilon$ satisfying Theorem \ref{dirichletcoulombfixing} is closed in $\LaB22$.
  \begin{proof}
    Let $A_i\to A$ be a sequence of connections converging in $\LaB21$ such that there exist $\LGB22$ gauge transformations $g_i$ sending $A_i$ to $\tA_i=d+\tA_i$ such that $d^*\ta_i=0$, $d^*_\pB i^*\ta_i=0$, and $\normlb{\ta_i}21\le C\normlb{F_{A_i}}2{}$. Because the $A_i$ and hence the $F_{A_i}$ converge, we know that the $\ta_i$ are bounded in $\LaB21$. Hence, after passing to a subsequence, the $\ta_i$ converge weakly in $\LaB21$ to $\ta$. Applying Lemma \ref{weakgaugeconvergence}, there exists a gauge transformation $g\in\LGB22$ sending $A$ to $\tA=d+\ta$.

    The operators $d^*$ and $d^*_\pB i^*$ are continuous and linear, so the conditions $d^*\ta_i=0$ and $d^*_\pB i^*\ta_i=0$ are preserved in the weak $\LaB21$ limit $d^*\ta=0$ and $d^*_\pB\ta=0$. Finally, because norms are lower semicontinuous under weak limits, we have
    \begin{equation*}
      \normlb\ta21\le\lim\inf\normlb{\ta_i}21\le C\lim\normlb{F_{A_i}}2{}=C\normlb{F_A}2{}.
    \end{equation*}
    Hence, $A$ indeed satisfies Theorem \ref{dirichletcoulombfixing}.

    Meanwhile, for the higher regularity claim, let $A_i\to A$ be a sequence of connections converging in $\LaB22$ with $\normlb{F_{A_i}}2{}<\epsilon$ and $\normlb{F_A}2{}<\epsilon$, such that there exist $\LGB23$ gauge transformations $g_i$ sending $A_i$ to $\tA_i=d+\ta_i$ with $d^*\ta_i=0$, $d^*_\pB i^*\ta_i=0$, and $\normlb{\ta_i}21\le C\normlb{F_{A_i}}2{}<C\epsilon$. Again, let $\epsilon_4$ be the constant from Lemma \ref{aprioribounds2}, and require $\epsilon$ to be small enough so that $\normlb{\ta_i}21<C\epsilon$ guarantees that $\normlb{\ta_i}4{}<\epsilon_4$. Then, applying Lemma \ref{aprioribounds2}, we have
    \begin{equation*}
      \normlb{\ta_i}22\le C\left(\normlb{\nabla_{\tA_i}F_{\tA_i}}2{}+\normlb{F_{\tA_i}}2{}\right)=C\left(\normlb{\nabla_{A_i}F_{A_i}}2{}+\normlb{F_{A_i}}2{}\right).
    \end{equation*}
    The convergence of the $A_i$ in $\LaB22$ guarantees the convergence of the right-hand side, so the $\ta_i$ are bounded in $\LaB22$, and hence, after passing to a subsequence, they have a weak limit $\ta$. In particular, the $\tA_i$ converge to $\tA$ in $\LaB21$, so we can apply the above argument to conclude that there is a $\LGB22$ gauge transformation $g$ sending $A$ to $\tA$ and $\tA$ satisfies the conditions of Theorem \ref{dirichletcoulombfixing}. Finally, the same argument as in Lemma \ref{closedconstantboundary} shows that because $A$ and $\tA$ are in $\LaB22$, the gauge transformation $g$ is in fact in $\LGB23$.
  \end{proof}
\end{lemma}

Now that we have that the set of low-energy connections satisfying Theorem \ref{dirichletcoulombfixing} are both open and closed, we need to show connectedness, which we did not need to worry about in the previous section because there we were working with the connected space of connections satisfying $\normlb a21<\epsilon$. This issue is made more subtle because we have an arbitrary metric on $B^4$, so the space of connections satisfying $\normlb{F_A}2{}<\epsilon$ need not in general be connected. We proceed by using the dilation technique in Uhlenbeck \cite{u82} to show that the space of low-energy connections on the ball with the \emph{standard} metric are connected, and then exploit the relationship between the standard metric and our arbitrary metric on $B^4$ to complete the proof of Theorem \ref{dirichletcoulombfixing}.

\begin{lemma}\label{connectedlemma}
  Let $B^4_\std$ denote the $4$-ball with the standard metric. Let $\epsilon>0$ be a constant, and let $k\ge1$. The set of $\LaBs2k$ connections $A$ such that $\norm{F_A}_{\LBs2{}}<\epsilon$ is connected.
  \begin{proof}
    Let $A=d+a$ be an $\LaBs2k$ connection with $\norm{F_A}_{\LBs2{}}<\epsilon$. For $0\le\lambda\le1$, let $f_\lambda\colon B^4_\std\to B^4_\std$ be the scaling map $f_\lambda(x)=\lambda x$. Using the fixed trivialization of the principal bundle $P$, we can identify $f_\lambda^*P$ with $P$.  Let $A_\lambda=f_\lambda^*A$, so $A_\lambda=d+\lambda\cdot a(\lambda x)$. Note in particular that $A_1=A$, and that $A_0=d$. Moreover, energy is conformally invariant, so, viewing $f_\lambda$ as an conformal isomorphism $B^4_\std\to\lambda\cdot B^4_\std$, we have
    \begin{equation*}
      \norm{F_{f_\lambda^*A}}_{\LBs2{}}=\norm{F_A}_{L^2(\lambda\cdot B^4_\std)}\le\norm{F_A}_{\LBs2{}}< \epsilon.
    \end{equation*}

    We claim that $A_\lambda$ is a continuous path of $\LaBs2k$ connections. Let $A_\lambda=d+a_\lambda$. At positive $\lambda$, the continuity of the map $\lambda\mapsto a_\lambda$ follows from the continuity of dilation on $L^p$ spaces. To show continuity at $\lambda=0$, we compute that, for $0\le j\le k$, we have
    \begin{multline*}
        \norm{\nabla^ja_\lambda}_{\LBs2{}}^2=\int_{B^4_\std}\abs{\lambda^{j+1}(\nabla^ja)(\lambda x)}^2\,dx\\
        =\lambda^{2(j+1)}\lambda^{-4}\int_{\lambda\cdot B^4_\std}\abs{(\nabla^ja)(\lambda x)}^2\,d(\lambda x)
        =\lambda^{2(j-1)}\norm{\nabla^ja}_{L^2(\lambda\cdot B^4_\std)}^2\le\lambda^{2(j-1)}\norm{\nabla^ja}_{\LBs2{}}^2.
    \end{multline*}
    When $j>1$, it is clear that $\norm{\nabla^ja_\lambda}^2_{\LBs2{}}\to0$ as $\lambda\to0$. When $j=1$, we note that $\norm{\nabla a_\lambda}_{\LBs2{}}=\norm{\nabla a}_{L^2(\lambda\cdot B^4_\std)}^2$, which also approaches zero as $\lambda\to0$. Finally, note that, by assumption, $a\in\LaBs21\subset\LaBs4{}$, so
    \begin{multline*}
      \norm{a_\lambda}_{\LBs2{}}=\lambda^{-1}\norm a_{L^2(\lambda\cdot B^4_\std)}\le\lambda^{-1}\norm a_{L^4(\lambda\cdot B^4_\std)}\norm{1}_{L^4(\lambda\cdot B^4_\std)}\\
      =\lambda^{-1}\norm a_{L^4(\lambda\cdot B^4_\std)}\left(\lambda^4\vol(B^4_\std)\right)^{1/4}=\vol\left(B^4_\std\right)^{1/4}\norm a_{L^4(\lambda\cdot B^4_\std)}.
    \end{multline*}
    Hence, $\norm{a_\lambda}_{\LBs2{}}\to0$ as $\lambda\to0$, completing the proof that $\norm{a_\lambda}_{\LBs2k}\to0$ as $\lambda\to0$.
  \end{proof}
\end{lemma}

We can now put together these lemmas to prove Theorem \ref{dirichletcoulombfixing}.

\begin{proof}[Proof of Theorem \ref{dirichletcoulombfixing}]
  Let $\epsilon'$ be the smaller of the constants in Corollary \ref{opencorollary2} and Lemma \ref{closedlemma2}. These results tells us that the space of connections satisfying Theorem \ref{dirichletcoulombfixing} is both relatively open and closed in the space of $\LaB22$ connections with $\normlb{F_A}2{}<\epsilon'$. Meanwhile, we can use the fact that $\wedge$ and integrals of differential forms are independent of the metric on $B^4$ to compare the energies with respect to our metric on $B^4$ and the standard metric:
  \begin{align*}
    \normlb{F_A}2{}^2&=\int_{B^4}F_A\wedge *F_A\le\normlbs{*}\infty{}\normlbs{F_A}2{}^2,\\
    \normlbs{F_A}2{}^2&=\int_{B^4}F_A\wedge *_\std F_A\le\normlb{*_\std}\infty{}\normlb{F_A}2{}^2.
  \end{align*}
  Let $\mc A_\delta^\std$ denote $\{A\in\LB22\mid\normlbs{F_A}2{}<\delta\}$. By the above inequalities, there exists an $\epsilon$ and a $\delta$ such that
  \begin{equation*}
    \left\{A\in\LB22\mid\normlb{F_A}2{}<\epsilon\right\}\subseteq\mc A_\delta^\std\subseteq\left\{A\in\LB22\mid\normlb{F_A}2{}<\epsilon'\right\}.
  \end{equation*}
  Hence, the space of connections satisfying Theorem \ref{dirichletcoulombfixing} is both relatively open and closed in $\mc A_\delta^\std$, which is connected by Lemma \ref{connectedlemma}. Hence, all connections in $\mc A_\delta^\std$ satisfy Theorem \ref{dirichletcoulombfixing}, and, in particular, so do all $\LB22$ connections with $\normlb{F_A}2{}<\epsilon$.

  Meanwhile, because $A\mapsto F_A$ is continuous as a map
  \begin{equation*}
    \LaB21\to\LFB2{},
  \end{equation*}
  any $\LB21$ connection with $\normlb{F_A}2{}<\epsilon$ is the $\LaB21$ limit of $\LB22$ connections $A_i$ with $\normlb{F_A}2{}<\epsilon$. These connections $A_i$ satisfy Theorem \ref{dirichletcoulombfixing}, so by Lemma \ref{closedlemma2} so does $A$.
\end{proof}

As in the preceding section, we finish by proving that the gauge transformation constructed by Theorem \ref{dirichletcoulombfixing} is unique up to a constant gauge transformation. To do so, we first prove uniqueness up to constants on the boundary.

\begin{proposition}\label{uniqueboundaryfixing}
  There exists a constant $\epsilon$ such that if $A=d+a$ and $B=d+b$ are two $\LapB2{1/2}$ connections gauge equivalent via a gauge transformation $g\in\LGpB2{3/2}$ satisfying
  \begin{enumerate}
  \item bounds $\normlpb a3{},\normlpb b3{}<\epsilon$, and
  \item the Coulomb conditions $d^*_\pB a=d^*_\pB b=0$,
  \end{enumerate}
  then $g$ is constant on $\pB$.
  \begin{proof}
    Equation \eqref{dsdg} in Proposition \ref{uniqueconstantboundary} is also valid on $\pB$ given $d^*_\pB a=d^*_\pB b=0$, so we have
    \begin{equation*}
      d^*_\pB d^{}_\pB g=-*_\pB(d_\pB g\wedge *_\pB a+*_\pB b\wedge d_\pB g).
    \end{equation*}
    Then,
    \begin{equation*}
      \begin{split}
        \normlpb{d^*_\pB d^{}_\pB g}2{-1/2}&\le C_S\normlpb{d^*_\pB d^{}_\pB g}{3/2}{}\\
        &\le C_S\left(\normlpb a3{}+\normlpb b3{}\right)\normlpb{d_\pB g}3{}\\
        &\le C_S^2\left(\normlpb a3{}+\normlpb b3{}\right)\normlpb{d_\pB g}2{1/2}.
      \end{split}
    \end{equation*}
    where $C_S$ is the operator norm of the Sobolev embeddings $\LpB{3/2}{}\hookrightarrow\LpB2{-1/2}$ and $\LpB2{1/2}\hookrightarrow\LpB3{}$.

    The theory of elliptic operators on closed manifolds tells us that the operator
    \begin{equation*}
      d^{}_\pB+d^*_\pB\colon\LEe2{1/2}\pB\to\LEe2{-1/2}\pB
    \end{equation*}
    is Fredholm. Moreover, it has no kernel on exact forms because
    \begin{equation*}
      \pairlpb{(d_\pB^{}+d_\pB^*{})(d_\pB g)}g=\pairlpb{d_\pB g}{d_\pB g}.
    \end{equation*}
    We conclude that there is a constant $C_G$ such that
    \begin{equation*}
      \normlpb{d_\pB g}2{1/2}\le C_G\normlpb{(d^{}_\pB+d^*_\pB)(d_\pB)g}2{-1/2}=C_G\normlpb{d^*_\pB d^{}_\pB g}2{-1/2}.
    \end{equation*}
    Putting these inequalities together, we have
    \begin{equation*}
      \normlpb{d_\pB g}2{1/2}\le C_GC_S^2\left(\normlpb a3{}+\normlpb b3{}\right)\normlpb{d_\pB g}2{1/2}.
    \end{equation*}
    Hence, requiring that $\epsilon\le\tfrac14(C_GC_S^2)^{-1}$ gives
    \begin{equation*}
      \normlpb{d_\pB g}2{1/2}\le\tfrac12\normlpb{d_\pB g}2{1/2},
    \end{equation*}
    so $d_\pB g=0$ and $g$ is constant on $\pB$, as desired.
  \end{proof}
\end{proposition}

We can now prove the uniqueness up to constants of the gauge transformation constructed in Theorem \ref{dirichletcoulombfixing} on all of $B^4$. Again, we assume $\LaB4{}$ and $\LapB3{}$ bounds, but these are implied by the $\LaB21$ bound given to us by Theorem \ref{dirichletcoulombfixing}.

\begin{corollary}\label{dirichletcoulombuniqueness}
  There exists constants $\epsilon_4$ and $\epsilon_3$ such that if $A=d+a$ and $B=d+b$ are two $\LaB21$ connections connections gauge equivalent via a gauge transformation $g\in\LGB22$ satisfying
  \begin{enumerate}
  \item bounds $\normlb a4{},\normlb b4{}<\epsilon_4$ and $\normlpb{i^*a}3{},\normlpb{i^*b}3{}<\epsilon_3$, and
  \item the Dirichlet Coulomb conditions $d^*a=d^*b=0$ and $d^*_\pB i^*a=d^*_\pB i^*b=0$,
  \end{enumerate}
  then $g$ is constant on $B^4$.
  \begin{proof}
    Choose $\epsilon_3$ small enough to apply Proposition \ref{uniqueboundaryfixing}. Then $i^*g$ is a constant gauge transformation $c\in G$ on $\pB$. Then choose $\epsilon_4$ small enough to apply Proposition \ref{uniqueconstantboundary}, giving us that $g$ is the constant gauge transformation $c$ on all of $B^4$.
  \end{proof}
\end{corollary}

\section*{Acknowledgments}
I would like to thank my dissertation advisor Tomasz Mrowka for his guidance and the huge amount of math I have learned from him over these past five years. I would also like to thank Paul Feehan for his detailed feedback on this project and for his encouragement and support. Finally, I would like to thank William Minicozzi, Larry Guth, Emmy Murphy, Antonella Marini, Tristan Rivi\`ere, and Karen Uhlenbeck for the helpful conversations. This material is based upon work supported by the National Science Foundation under grants No.~1406348 (PI Mrowka) and 0943787 (RTG). I was also supported by the NDSEG fellowship and by MIT.

%I'd like to thank the rest of my committee, Bill Minicozzi and Emmy Murphy, as well as Larry Guth, Gigliola Staffilani, and Richard Melrose for the math they've taught me and the helpful conversations, and my mentor Dmitry Zenkov for his support over the years.

%I'd like to thank my parents for getting me to this point, as well as the rest of my family, particularly my cousin Ilya Grigoriev for his ever helpful advice. I'd like to thank Hannah Alpert for all of the practice talks we exchanged and her patience listening to the latest piece of math I was battling. Thanks go out also to the rest of my roommates, Dmitry Vaintrob, Eric Larson, Isabel Vogt, and Daniel Jacobson, as well as to Manya Tyutyunik and the rest of my friends, new and old, that have made my time here a fantastic five years of my life.

\bibliographystyle{plain}
\bibliography{gaugetheory}

\begin{thebibliography}{10}

\bibitem{a75}
Robert~A. Adams.
\newblock {\em Sobolev spaces}.
\newblock Academic Press [A subsidiary of Harcourt Brace Jovanovich,
  Publishers], New York-London, 1975.
\newblock Pure and Applied Mathematics, Vol. 65.

\bibitem{bf91}
Franco Brezzi and Michel Fortin.
\newblock {\em Mixed and hybrid finite element methods}, volume~15 of {\em
  Springer Series in Computational Mathematics}.
\newblock Springer-Verlag, New York, 1991.

\bibitem{cm08}
Tobias~H. Colding and William~P. Minicozzi, II.
\newblock Width and finite extinction time of {R}icci flow.
\newblock {\em Geom. Topol.}, 12(5):2537--2586, 2008.

\bibitem{d83}
S.~K. Donaldson.
\newblock Self-dual connections and the topology of smooth {$4$}-manifolds.
\newblock {\em Bull. Amer. Math. Soc. (N.S.)}, 8(1):81--83, 1983.

\bibitem{d93}
S.~K. Donaldson.
\newblock The approximation of instantons.
\newblock {\em Geom. Funct. Anal.}, 3(2):179--200, 1993.

\bibitem{dk90}
S.~K. Donaldson and P.~B. Kronheimer.
\newblock {\em The geometry of four-manifolds}.
\newblock Oxford Mathematical Monographs. The Clarendon Press, Oxford
  University Press, New York, 1990.
\newblock Oxford Science Publications.

\bibitem{f14a}
Paul M.~N. Feehan.
\newblock Global existence and convergence of smooth solutions to
  {Y}ang-{M}ills gradient flow over compact four-manifolds, 2014.
\newblock \url{http://arxiv.org/abs/1409.1525}.

\bibitem{fl14b}
Paul M.~N. Feehan and Thomas~G. Leness.
\newblock Superconformal simple type and {W}itten's conjecture, 2014.
\newblock \url{http://arxiv.org/abs/1408.5085}.

\bibitem{fu84}
Daniel~S. Freed and Karen~K. Uhlenbeck.
\newblock {\em Instantons and four-manifolds}, volume~1 of {\em Mathematical
  Sciences Research Institute Publications}.
\newblock Springer-Verlag, New York, 1984.

\bibitem{fm88}
Robert Friedman and John~W. Morgan.
\newblock On the diffeomorphism types of certain algebraic surfaces. {I}.
\newblock {\em J. Differential Geom.}, 27(2):297--369, 1988.

\bibitem{gr86}
Vivette Girault and Pierre-Arnaud Raviart.
\newblock {\em Finite element methods for {N}avier-{S}tokes equations},
  volume~5 of {\em Springer Series in Computational Mathematics}.
\newblock Springer-Verlag, Berlin, 1986.
\newblock Theory and algorithms.

\bibitem{i09}
Takeshi Isobe.
\newblock Topological and analytical properties of {S}obolev bundles. {I}.
  {T}he critical case.
\newblock {\em Ann. Global Anal. Geom.}, 35(3):277--337, 2009.

\bibitem{im97}
Takeshi Isobe and Antonella Marini.
\newblock On topologically distinct solutions of the {D}irichlet problem for
  {Y}ang-{M}ills connections.
\newblock {\em Calc. Var. Partial Differential Equations}, 5(4):345--358, 1997.

\bibitem{im10}
Takeshi Isobe and Antonella Marini.
\newblock Small coupling limit and multiple solutions to the {D}irichlet
  problem for {Y}ang-{M}ills connections in four dimensions. {III}, 2010.
\newblock \url{https://arxiv.org/abs/1006.2569}.

\bibitem{im12a}
Takeshi Isobe and Antonella Marini.
\newblock Small coupling limit and multiple solutions to the {D}irichlet
  problem for {Y}ang-{M}ills connections in four dimensions. {I}.
\newblock {\em J. Math. Phys.}, 53(6):063706, 39, 2012.

\bibitem{im12b}
Takeshi Isobe and Antonella Marini.
\newblock Small coupling limit and multiple solutions to the {D}irichlet
  problem for {Y}ang-{M}ills connections in four dimensions. {II}.
\newblock {\em J. Math. Phys.}, 53(6):063707, 39, 2012.

\bibitem{j91}
J{\"u}rgen Jost.
\newblock {\em Two-dimensional geometric variational problems}.
\newblock Pure and Applied Mathematics (New York). John Wiley \& Sons, Ltd.,
  Chichester, 1991.
\newblock A Wiley-Interscience Publication.

\bibitem{kn63}
Shoshichi Kobayashi and Katsumi Nomizu.
\newblock {\em Foundations of differential geometry. {V}ol. {I}}.
\newblock Wiley Classics Library. John Wiley \& Sons, Inc., New York, 1996.
\newblock Reprint of the 1963 original, A Wiley-Interscience Publication.

\bibitem{m92}
Antonella Marini.
\newblock Dirichlet and {N}eumann boundary value problems for {Y}ang-{M}ills
  connections.
\newblock {\em Comm. Pure Appl. Math.}, 45(8):1015--1050, 1992.

\bibitem{p23}
Oskar Perron.
\newblock Eine neue {B}ehandlung der ersten {R}andwertaufgabe f\"ur {$\Delta
  u=0$}.
\newblock {\em Math. Z.}, 18(1):42--54, 1923.

\bibitem{r13}
Johann Radon.
\newblock Theorie und anwendungen der absolut additiven mengenfunktionen.
\newblock {\em Sitzungsber. Akad. Wiss. Wien.}, 122, 1913.

\bibitem{r14}
Tristan Rivi\`ere.
\newblock The variations of the {Y}ang-{M}ills lagrangian.
\newblock KIAS lecture notes.
  \url{https://people.math.ethz.ch/~riviere/papers/yang-mills-course-kias-06-14.pdf},
  2014.

\bibitem{su81}
J.~Sacks and K.~Uhlenbeck.
\newblock The existence of minimal immersions of {$2$}-spheres.
\newblock {\em Ann. of Math. (2)}, 113(1):1--24, 1981.

\bibitem{s70}
H.~A. Schwarz.
\newblock Ueber einen {G}renz\"ubergang durch alternirendes verfahren.
\newblock {\em Vierteljahrsschrift der Naturforschenden Gesellschaft in
  Z\"urich}, 15:272--286, 1870.

\bibitem{s82}
Steven Sedlacek.
\newblock A direct method for minimizing the {Y}ang-{M}ills functional over
  {$4$}-manifolds.
\newblock {\em Comm. Math. Phys.}, 86(4):515--527, 1982.

\bibitem{s02}
Vsevolod~V. Shevchishin.
\newblock Limit holonomy and extension properties of {S}obolev and
  {Y}ang-{M}ills bundles.
\newblock {\em J. Geom. Anal.}, 12(3):493--528, 2002.

\bibitem{s94}
Michael Struwe.
\newblock The {Y}ang-{M}ills flow in four dimensions.
\newblock {\em Calc. Var. Partial Differential Equations}, 2(2):123--150, 1994.

\bibitem{t84}
Clifford~Henry Taubes.
\newblock Path-connected {Y}ang-{M}ills moduli spaces.
\newblock {\em J. Differential Geom.}, 19(2):337--392, 1984.

\bibitem{t88}
Clifford~Henry Taubes.
\newblock A framework for {M}orse theory for the {Y}ang-{M}ills functional.
\newblock {\em Invent. Math.}, 94(2):327--402, 1988.

\bibitem{t89}
Clifford~Henry Taubes.
\newblock The stable topology of self-dual moduli spaces.
\newblock {\em J. Differential Geom.}, 29(1):163--230, 1989.

\bibitem{t96}
Michael~E. Taylor.
\newblock {\em Partial differential equations. {I}}, volume 115 of {\em Applied
  Mathematical Sciences}.
\newblock Springer-Verlag, New York, 1996.
\newblock Basic theory.

\bibitem{u82}
Karen~K. Uhlenbeck.
\newblock Connections with {$L^{p}$} bounds on curvature.
\newblock {\em Comm. Math. Phys.}, 83(1):31--42, 1982.

\bibitem{u82b}
Karen~K. Uhlenbeck.
\newblock Removable singularities in {Y}ang-{M}ills fields.
\newblock {\em Comm. Math. Phys.}, 83(1):11--29, 1982.

\bibitem{w16}
Alex Waldron.
\newblock Instantons and singularities in the {Y}ang-{M}ills flow.
\newblock {\em Calc. Var. Partial Differential Equations}, 55(5):Art. 113, 31,
  2016.

\bibitem{w94}
Edward Witten.
\newblock Monopoles and four-manifolds.
\newblock {\em Math. Res. Lett.}, 1(6):769--796, 1994.

\end{thebibliography}

\end{document}